\documentclass[11pt,a4paper]{amsart}

\usepackage{fullpage}
\usepackage{amsmath}
\usepackage{amsfonts}
\usepackage{amssymb}
\usepackage{mathrsfs}
\usepackage{mathtools}
\usepackage{bbm}
\usepackage{tikz-cd}
\usepackage{url}
\usepackage[english]{babel}
\usepackage[utf8]{inputenc}
\usepackage{csquotes}
\usepackage{enumitem}
\usepackage{multicol}
\usepackage{calc}

\setlist[itemize,1]{label={--\,}}
\setlist[enumerate,1]{label=(\roman*)}

\usepackage[backend=biber, maxbibnames=50, style=alphabetic]{biblatex}
\renewbibmacro{in:}{}

\usepackage{hyperref}
\usepackage{cleveref}
\DeclareRobustCommand{\gobblefour}[4]{}

\makeatletter
\renewcommand{\tocsection}[3]{%
	\indentlabel{\@ifnotempty{#2}{\ignorespaces#1 \makebox[\widthof{00.}][l]{#2.}\quad}}#3}
\renewcommand{\tocsubsection}[3]{%
	\indentlabel{\@ifnotempty{#2}{\ignorespaces#1 \makebox[\widthof{00.0.}][l]{#2.}\quad}}#3}
\makeatother

\setlist[itemize]{leftmargin=*}
\setlist[enumerate]{leftmargin=*}

\newtheorem{thm}{Theorem}[section]
\newtheorem{defin}[thm]{Definition}
\newtheorem{notation}[thm]{Notation}
\newtheorem{prop}[thm]{Proposition}
\newtheorem{cor}[thm]{Corollary}

\newtheorem{lemma}[thm]{Lemma}

\theoremstyle{definition}
\newtheorem{ex}[thm]{Example}
\newtheorem{rem}[thm]{Remark}

\newcommand{\F}{\mathbb F}
\newcommand{\N}{\mathbb N}
\newcommand{\bP}{\mathbb P}
\newcommand{\Q}{\mathbb Q}

\newcommand{\Z}{\mathbb Z}

\newcommand{\adm}{\mathrm{adm}}
\newcommand{\Ad}{\mathrm{Alg}^{\mathrm{ff}}}
\newcommand{\aff}{\mathrm{aff}}
\newcommand{\Aff}{\mathrm{Aff}}
\newcommand{\Alg}{\mathrm{Alg}}

\newcommand{\Art}{\mathrm{Art}}
\newcommand{\cont}{\mathrm{cont}}

\newcommand{\crys}{\mathrm{cris}}

\newcommand{\Def}{\mathrm{Def}}
\newcommand{\ff}{\mathrm{ff}}
\newcommand{\form}{\mathrm{form}}
\newcommand{\Frm}{\mathrm{FSch}}
\newcommand{\Frob}{\mathrm{Frob}}
\newcommand{\Fil}{\mathrm{Fil}}

\newcommand{\Gal}{\mathrm{Gal}}
\newcommand{\GL}{\mathrm{GL}}
\newcommand{\Hom}{\mathrm{Hom}}

\newcommand{\Mat}{\mathrm{M}}

\newcommand{\op}{\mathrm{op}}
\newcommand{\PDef}{\mathrm{PDef}}

\newcommand{\prof}{\mathrm{prof}}
\newcommand{\ps}{\mathrm{ps}}

\newcommand{\reg}{\mathrm{reg}}
\newcommand{\rig}{\mathrm{rig}}
\newcommand{\Rig}{\mathrm{Rig}}
\newcommand{\sep}{\mathrm{sep}}
\newcommand{\spec}{\mathrm{sp}}
\newcommand{\Sets}{\mathrm{Sets}}
\newcommand{\tri}{\mathrm{tri}}

\newcommand{\univ}{\mathrm{univ}}
\newcommand{\wo}{\mathrm{sqS}}

\DeclareMathOperator{\Aut}{Aut}
\DeclareMathOperator{\End}{End}
\DeclareMathOperator{\ev}{ev}
\DeclareMathOperator{\Frac}{Frac}
\DeclareMathOperator{\Spf}{Spf}
\DeclareMathOperator{\Spm}{Spm}
\DeclareMathOperator{\Spec}{Spec}
\DeclareMathOperator{\tr}{tr}

\newcommand{\xto}{\xrightarrow}
\newcommand{\into}{\hookrightarrow}
\newcommand{\onto}{\twoheadrightarrow}

\newcommand{\ovl}{\overline}
\newcommand{\wtl}{\widetilde}
\newcommand{\wht}{\widehat}
\newcommand{\ccirc}{\kern0.5ex\vcenter{\hbox{$\scriptstyle\circ$}}\kern0.5ex}

\newcommand{\cA}{\mathcal{A}}
\newcommand{\cB}{\mathcal{B}}
\newcommand{\calC}{\mathcal{C}}
\newcommand{\cD}{\mathcal{D}}
\newcommand{\cE}{\mathscr{E}}

\newcommand{\cL}{\mathcal{L}}
\newcommand{\cM}{\mathcal{M}}
\newcommand{\cO}{\mathcal{O}}

\newcommand{\cR}{\mathcal{R}}

\newcommand{\cT}{\mathcal{T}}
\newcommand{\cX}{\mathcal{X}}
\newcommand{\cY}{\mathcal{Y}}
\newcommand{\cU}{\mathcal{U}}
\newcommand{\cV}{\mathcal{V}}
\newcommand{\cW}{\mathcal{W}}
\newcommand{\cZ}{\mathcal{Z}}

\newcommand{\bV}{{\mathbf{V}}}

\newcommand{\Qp}{\overline{\Q}_p}
\newcommand{\Zp}{\overline{\Z}_p}

\newcommand{\sq}{\square}

\newcommand{\fX}{{\mathfrak{X}}}
\newcommand{\fm}{{\mathfrak{m}}}
\newcommand{\fp}{{\mathfrak{p}}}
\newcommand{\fU}{{\mathfrak{U}}}

\title{Lattices in rigid analytic representations}
\author{Andrea Conti and Emiliano Torti}
\thanks{Andrea Conti, IWR, Heidelberg University, \url{andrea.conti@iwr.uni-heidelberg.de}, ORCID: 0000-0001-9279-9451; Emiliano Torti, Laboratoire GAATI, Université de la Polynésie française, Tahiti, \url{torti@gaati.org}, ORCID: 0000-0003-4431-3001}

\begin{document}
	
\begin{abstract}
For a profinite group $G$ and a rigid analytic space $X$, we study when an $\cO_X(X)$-linear representation $V$ of $G$ admits a lattice, i.e. an $\cO_\cX(\cX)$-linear model for a suitable formal model $\cX$ of $X$ in the sense of Berthelot. We give a positive answer, under mild assumptions, when $X$ is strictly quasi-Stein and regular. As a consequence, we are able to describe explicit open rational subdomains of $X$ over which $V$ is constant after reduction modulo a power of $p$. We give applications in two different directions. First, we prove explicit results on the reduction modulo powers of $p$ of sheaves of crystalline and semistable representations of fixed weight. Second, %we focus on the sheaves of Galois representations on eigenvarieties. %, which are important examples of wide open spaces thanks to a result of Bella\"iche and Chenevier. 
we deduce a result on the pseudorepresentation carried by the Coleman--Mazur eigencurve, that can be made explicit whenever equations for a rational subdomain of the eigencurve are given.
\end{abstract}

	\maketitle

\setcounter{tocdepth}{1}
\tableofcontents

\section{Introduction}

Fix a prime $p$, and let $f$ be a modular eigenform for $\GL_{2/\Q}$ of weight $k\ge 2$ and level $N$, with $p\nmid N$. The associated $p$-adic representation of $\Gal(\ovl\Q/\Q)$ is unramified away from $Np$, and crystalline at $p$. When $f$ varies in a $p$-adic family along the Coleman--Mazur eigencurve $\cE$, the corresponding global Galois representation also varies in a $p$-adic analytic way: this means that, locally on a sufficiently small irreducible subdomain $X$ of $\cE$, we can find a continuous representation
\[ \rho_X\colon \Gal(\ovl\Q/\Q) \to\GL_2(\cO_X(X)) \]
that specializes to the representations attached to the classical eigenforms in $X$. 
If $X$ is a whole irreducible component of the eigencurve, we cannot in general find a $\rho_X$ as above. However, there are other objects that we can attach to $X$: we can either write a 2-dimensional \emph{pseudorepresentation}
\[ T_X\colon \Gal(\ovl\Q/\Q) \to\cO_X(X) \]
that interpolates the traces of the representations corresponding to the classical points, or patch all representations defined on small subdomains of $X$ to obtain a \emph{sheaf of $\Gal(\ovl\Q/\Q)$-representations}, i.e. a coherent $\cO_X$-module equipped with an $\cO_X$-linear action of $\Gal(\ovl\Q/\Q)$. 

While investigating how eigenforms vary $p$-adically, it is interesting to understand how the above picture looks after we ``reduce it'' modulo a power of $p$. %Since $X$ as above is a rigid space over $\Q_p$, it makes no sense to reduce $\cO_X$-linear representations modulo powers of $p$. 
When $X$ is a point, say $\Spm\Q_p$, it is well-known that for a continuous representation $\rho\colon\Gal(\ovl\Q/\Q)\to\GL(V)$, $V$ a $\Q_p$-vector space, one can always find a $\Gal(\ovl\Q/\Q)$-stable $\Z_p$-lattice $\cV$ in $V$, hence attach to $V$ the reduction $\cV\otimes_{\Z_p}\Z/p^n$ for every $n\ge 1$. Such a reduction depends on the choice of a lattice: even if $n=1$, it is uniquely determined only up to semisimplification. In the special case when $\rho=\rho_{f,p}$ is attached to an eigenform $f$, one can try to describe such a reduction in terms of the automorphic data attached to $f$, i.e. its weight and Hecke eigensystem. A somewhat easier task is that of looking at the \emph{local} Galois representation $\rho_{f,p}\vert_{G_{\Q_p}}$, after choosing a decomposition group $G_{\Q_p}$. If $f$ is not $p$-new, then the isomorphism class of $\rho_{f,p}\vert_{G_{\Q_p}}$ can be described in terms of the weight $k$ and the $U_p$-eigenvalue $a_p$ of $f$. The most interesting case is that when $f$ is not ordinary at $p$, i.e. when its $U_p$-eigenvalue $a_p$ has positive valuation, so that $\rho_{f,p}\vert_{G_{\Q_p}}$ is an irreducible representation $V_{k,a_p}$.

%It is already a hard problem to describe the semisimplified mod $p$ reduction of the $p$-adic Galois representation $\rho_{f,p}$, associated with an eigenform $f$, in terms of the data attached to $f$ (i.e., its weight and Hecke eigensystem). 
%In order to simplify the question, one can look instead at the \emph{local} Galois representation $\rho_{f,p}\vert_{G_{\Q_p}}$, where the most interesting case is that when $f$ is not ordinary at $p$, i.e. when its $U_p$-eigenvalue $a_p$ has positive valuation, so that $\rho_{f,p}\vert_{G_{\Q_p}}$ is irreducible. A standard construction produces the isomorphism class of $\rho_{f,p}\vert_{G_{\Q_p}}$ in terms of $k$ and $a_p$. 

More generally, to any $k\ge 2$ and $a_p\in\Qp$ of positive valuation, we can attach a 2-dimensional, irreducible crystalline representation $V_{k,a_p}$, and all 2-dimensional, irreducible crystalline representations of $\Gal(\Qp/\Q_p)$ are of this form, up to a twist with a crystalline character. We refer to the introduction of \cite{berloccon} for a summary of the work of various people towards determining the semisimplified mod $p$ reduction of $V_{k,a_p}$ in terms of $k$ and $a_p$.

We concern ourselves here with a different point of view: we would like to understand how the $V_{k,a_p}$ and their reductions vary when $k$ and $a_p$ are allowed to vary $p$-adically. In this direction, we recall the following result. We let $L$ be a $p$-adic field with valuation ring $\cO_L$. We write $\fm_L$ for the maximal ideal of $\cO_L$, $\pi_L$ for a uniformizer of $L$, and $e$ for the ramification index of $L/\Q_p$. Set $\alpha(k-1)=\sum_{n\ge 1}\lfloor\frac{k-1}{p^{n-1}(p-1)}\rfloor$.

\begin{thm}[{\cite[Theorem 1.1]{tortired}}]\label{introred}
Let $k$ be an integer at least 2, $a_{p,0}, a_{p,1}\in\fm_L$, and $n\in\Z_{\geq 1}$. If $v_p(a_{p,1}-a_{p,0})>2v_p(a_{p,0})+\alpha(k-1)+en$, then there exist $\Gal(\Qp/\Q_p)$-stable lattices $\cV_0$ and $\cV_1$ in $V_{k,a_{p,0}}$ and $V_{k,a_{p,1}}$, respectively, and an isomorphism
\[ \cV_0\otimes_{\cO_L}\cO_L/\pi_L^n\cong\cV_1\otimes_{\cO_L}\cO_L/\pi_L^n\]
of $\cO_L/\pi_L^n[\Gal(\Qp/\Q_p)]$-modules.
\end{thm}
Apart from the explicit radius, the idea behind Theorem \ref{introred} is that the modulo $\pi_L^n$ reduction of $V_{k,a_p}$ does not change if we let $k$ and $a_p$ vary in sufficiently small $p$-adic neighborhoods of their respective ambient spaces. As it is apparent from the above discussion, such a notion of ``local constancy'' of the modulo $p^n$ reduction depends on a suitable choice of lattices for the representations involved. 
In the paper \cite{tortilatt}, the second-named author shows that this is a general phenomenon: if $X$ is a rigid analytic space and $\bV$ a sheaf of $G$-representations over $X$, then for every $n\ge 1$ one can find a (not admissible, in general) affinoid covering $\fU$ of $X$ such that, for every $U\in\fU$, the restriction $\bV\vert_{U}$ is \emph{constant modulo $p^n$}, roughly in the following sense: given any two specializations $x$ and $y$ of $U$, the fibers $\bV_x$ and $\bV_y$ admit lattices $\cV_x$ and $\cV_y$ that give isomorphic representations of $G$ after reduction modulo $p^n$. One has to take some care in making this notion precise when allowing $x$ and $y$ to be defined over arbitrary $p$-adic fields. The key to finding the two lattices is that one can actually find a lattice over the whole of $U$, in the sense of the following definition. Let $X$ be a rigid analytic space over $L$, and let $\bV$ be a $G$-representation over $X$. 

\begin{defin}
A \emph{sheaf of lattices} for $\bV$ is the datum of a formal model $\cX$ of $X$ and a sheaf $\cV$ of $G$-representations over $X$ whose generic fiber, equipped with the action of $G$ induced from $\cV$, is isomorphic to $\bV$. We say that such a $\cV$ is a \emph{lattice} for $\bV$ if $\cV$ is associated to a finite free $\cO_\cX(\cX)$-module equipped with an $\cO_\cX(\cX)$-linear action of $G$.
\end{defin}

If $\cV$ is a lattice in $\bV$, defined over some formal model $\cX$ of $X$, then given any rig-point of $x$, in the sense of \cite[Section 7.1.10]{DJ95}, one can attach to it the fiber of $\cV$ at $x$, a representation of $G$ over a finite free $\cO$-module, $\cO$ the valuation ring of a $p$-adic field.

In the current paper, we deal with two main questions. The first one is under what conditions, given a profinite group $G$, a rigid analytic space $X$ over a $p$-adic field $L$ and a continuous representation
\[ \rho_X\colon G\to\GL_2(\cO_X(X)), \]
we can find a lattice for $\rho_X$. 
%, i.e. a formal model $\cX$ of $X$, in the sense of Berthelot, and a continuous representation
%\[ \rho_\cX\colon G\to\GL_2(\cO_\cX(\cX)) \]
%that induces $\rho_X$ via the natural homomorphism $\cO_\cX(\cX)\to\cO_X(X)$. 
This is known to be a difficult problem:  %without conditions on the family of $G$-representations considered 
even over an affinoid $X$ and without taking into account the action of $G$, it is not possible in general to find a lattice in a coherent sheaf: see a counterexample of Hellmann in \cite[Section 3]{Hel16}. The obstacle is that in general the first cohomology group of an integral coherent sheaf does not vanish. 
However, by a result of Chenevier \cite[Lemme 3.18]{chendens}, we can always find a \emph{sheaf of lattices} for $\rho_X$, i.e. a coherent, locally free sheaf $\cV$ over a formal model $\cX$ of $X$, equipped with a $G$-action inducing $\rho_X$. The locally free sheaf $\cV$ is not a priori free, i.e. the local lattices might not glue.
%Understanding when a family of $G$-representations over a Banach algebra $S$ (e.g., in the case when $X$ as above is affinoid) admits an integral model played a central role in the work of Berger and Colmez \cite{bercol}, Kedlaya and Liu \cite{kedliufam} and Chenevier \cite{chendens}. In particular, Chenevier proves that one can always find a \emph{sheaf of lattices} for $\rho_X$, i.e. a coherent, locally free sheaf $\cV$ over a formal model $\cX$ of $X$, equipped with a $G$-action inducing $\rho_X$. 
%Our first main result is the following.
The following key proposition allows us to work with an actual lattice under reasonable conditions on the underlying space $X$ and the representation $\rho_X$. 

\begin{prop}[cf. Proposition \ref{prop:factorial}]\label{intro1}
Assume that $X$ is strictly quasi-Stein, i.e. a strictly increasing union of affinoids such that the restriction maps are compact and dense. Let $\bV$ be an absolutely irreducible, residually multiplicity-free $G$-representation on $X$ of rank $d<p$. If $\bV$ is not residually absolutely irreducible, assume that $\cO_X^+(X)$ is a UFD. Then $\bV$ on $X$ admits a lattice over the formal model $\Spf\cO_X^+(X)$.
\end{prop}

Proposition \ref{prop:factorial} is a simple consequence of results of Nyssen, Rouquier, and Bella\"iche--Chenevier. The condition on $\cO_X^+(X)$ being a UFD is satisfied for instance if $X$ is a wide open disc (see Example \ref{exdisc}) or more generally if $X$ is regular (on top of being strictly quasi-Stein). We briefly discuss this in Remark \ref{rem:regular}. This will be enough for our applications to representations of arithmetic interest.

Strictly quasi-Stein spaces are a slightly special case of ``wide open'' domains (for which the density condition is not imposed), and include all generic fibers of formal spectra of profinite rings, such as wide open discs (see the discussion in \Cref{quasisteinrep}). 
Wide open domains appear in \cite{bellchen} under the terminology ``nested'', and variations of this notion are referred to as ``overconvergent domains'' in the literature. We expect that the assumption that $X$ is strictly quasi-Stein in Theorem \ref{intro1} cannot be made much weaker, since a biproduct of the proof is the fact that a sheaf of $G$-representations on $X$ is a $G$-representation, which holds on quasi-Stein spaces by a classical result of Kiehl (see Remark \ref{quasioptimal} for a comment on this).
%For a quasi-Stein space $X$, a classical result of Kiehl shows that every sheaf of $G$-representations on $X$ is actually a $G$-representation.
%\andr{To obtain a lattice in Theorem \ref{intro1}, we first attach to $\bV$ a representation into a ``generalized matrix algebra (GMA)'' over $\cO_X^+(X)$, produced under the residual multiplicity-free assumption by the results of Bella\"iche and Chenevier \cite[Section 1]{bellchen}. Then we rely on the existence of a sheaf of lattices, given by the aforementioned result of Chenevier, to show that such a GMA is necessarily an actual matrix algebra.}
%Note that in Theorem \ref{intro1}, we start with a $G$-representation rather than just a sheaf. However, if one assumes the slightly stronger condition that $X$ is \emph{strictly quasi-Stein} (i.e. a strictly increasing union of affinoids such that the restriction maps are compact and of dense image, which is the case for instance for wide open discs), then a classical result of Kiehl shows that every sheaf of $G$-representations on $X$ is actually a $G$-representation (see Section \ref{quasisteinrep}).

We remark that the situation is much simpler if we start with a continuous \emph{pseudorepresentation} $T\colon G\to\cO_X(X)$: by topological reasons, the image of $T$ lies in the subring $\cO_X^+(X)$ of power-bounded elements in $\cO_X(X)$, and one can write such a ring as $\cO_\cX(\cX)$ for a formal model $\cX$ of $X$, under reasonable assumptions on $X$ (see Lemma \ref{Tform}).

The problem we treat is related to some questions about families of $(\varphi,\Gamma)$-modules over rigid spaces. Berger and Colmez \cite{bercol} attach a family of \'etale $(\varphi,\Gamma)$-modules to a sheaf of $G_{\Q_p}$-representations over a rigid (or adic) space $X$, in a functorial way, and Kedlaya--Liu \cite{kedliufam} and Hellmann \cite{hellfam} study the essential image of such a functor, i.e. which families of \'etale $(\varphi,\Gamma)$-modules over $X$ can be traced back to a sheaf of $G_{\Q_p}$-representations. Kedlaya and Liu observe that the obstruction to doing so is given by a residual condition: the mod $p$ representation attached to $\rho_X$ is necessarily constant along $X$ whenever $\rho_X$ is attached to a sheaf of $G_{\Q_p}$-representation. One way they circumvent this obstruction is by restricting themselves to the case where $S$ is a \emph{local coefficient algebra}, i.e. a ring of the form $S = R\otimes_{\Z_p}\Q_p$, where $R$ is a local Noetherian $\Z_p$-algebra. If $X$ is a wide open, then the above construction produces a sheaf of $G_{\Q_p}$-representations over it, but care must be taken since $\cO_X(X)$ is not itself a local coefficient algebra, despite $\cO_X^+(X)$ being local: if $X$ is written as an increasing union of affinoids $\Spm(A_i)$ for $i\ge 0$, with compact restriction maps $A_{i+1}\to A_i$, then the ring $S=\cO_X(X)$ is the projective limit of the Banach algebras $A_i$, hence a Fr\'echet $\Q_p$-algebra that is not a Banach $\Q_p$-algebra. We refer to Example \ref{exlog} for a standard example of this situation, and to the discussion around it for a more detailed description of how our work is related with the results of \cite{kedliufam}.

%The strength of our Theorem \ref{intro1} above consists in showing the existence of a global lattice over a non-affinoid space whose ring of functions $S$ is not even a local coefficient algebra. Indeed, if a wide open space $X$ is written as an increasing union of affinoids $\Spm(A_i)$ for $i\ge 0$, with compact restriction maps $A_{i+1}\to A_i$, then the ring $S=\cO_X(X)$ is the projective limit of the Banach algebras $A_i$, hence a Fr\'echet $\Q_p$-algebra that is not a Banach $\Q_p$-algebra. We refer to Example \ref{exlog} for a standard example of this situation, and to the discussion around it for a more detailed description of our work is related with the results of \cite{kedliufam}.

Our second question, and motivation for the first one, is to study the reduction of a $G$-representation $\rho_X$ as above modulo a power of $p$. More precisely, for every point $x$ of $X$ and integer $n\ge 1$, we would like to identify, as explicitly as possible, a neighborhood of $x$ over which the modulo $p^n$ reduction is constant. We give two notions of constancy. 
%The second one is a stronger condition that is very natural. 
Unsurprisingly, both notions depend on the choice of a lattice for $\rho_X$. %: we say that a $G$-representation over a rigid analytic space is \emph{pointwise constant mod $p^n$} if its reduction modulo a power of the uniformizer 
For every finite extension $E$ of $L$ of ramification index $e_{E/L}$, we define the natural number $\gamma_{E/L}=(n-1)e_{E/L}+1$: it is the smallest exponent $m$ for which the injection $\cO_L\into\cO_E$ induces an injection $\cO_L/\pi_L^n\into\cO_E/\pi_E^m$, and it is always at least $n$.

\begin{defin}[cf. Definition \ref{lcformdef}]
We say that a sheaf of $G$-representations $\cV$ over a formal scheme $\cX$ is:
\begin{enumerate}[label=(\roman*)]
	\item \emph{pointwise constant mod $p^n$} if, for every finite extension $E$ of $L$, with uniformizer $\pi_E$, the isomorphism class of $\cV_{E,x}/\pi_E^{\gamma_{E/L}(n)}$ as an $\cO_E/\pi_E^{\gamma_{E/L}(n)}[G]$-module is independent of the choice of an $E$-rig-point $x$ of $\cX_E$;
	\item \emph{constant mod $\pi_L^n$} if there exists a finite, free $\cO_L/\pi_L^n$-module $V^{(n)}$, equipped with an $\cO_L/\pi_L^n$-linear action of $G$, such that $\bV/\pi_L^n\cong V^{(n)}\otimes_{\cO_L/\pi_L^n}\cO_\cX/\pi_L^n\cO_\cX$.
\end{enumerate}
We say that a $G$-representation $\bV$ over a rigid analytic space $X$ has one of the above properties if it admits a lattice $\cV$ which does.
\end{defin}

The first condition corresponds to that studied by Torti \cite{tortilatt} in the context of rigid analytic families, and earlier by Taix\'{e}s i Ventosa--Wiese \cite{taiwie} and Chen--Kiming--Wiese \cite{ckwmod} for classical eigenforms. In these papers, the congruence is formulated in terms of congruence over a certain ``integral closure'' $\ovl{\Z/p^n\Z}$ of $\Z/p^n\Z$. The choice of $\gamma_{E/L}(n)$ is such that constancy mod $\pi_L^n$ lies (strictly) between pointwise constancy mod $p^n$ and $p^{n+1}$, as we explain in Section \ref{gammasec}. %The fact that condition (ii) implies (i) follows from the properties of the map $\cO_L/\pi_L^n\into\cO_E/\pi_E^m$; we detail the argument in Section \ref{gammasec}

Our result concerning constancy modulo $p^n$ is based on a very simple remark, that is better illustrated via an example. Let $f$ be a power-bounded function on the affinoid rigid unit disc $D(0,1)$, i.e. an element of $\Z_p\langle T\rangle$. We can restrict $f$ to a function over an affinoid disc $D(0,p^{-1})$ of radius $p^{-1}$, which amounts to seeing $f$ as an element of $\Z_p[[T/p]]$: clearly, we can write $f=f(0)+pg$ for some $g\in\Z_p[[T/p]]$, so that reducing $f$ modulo $p$ over $D(0,p^{-1})$ gives back the constant $f(0)$. However, this already holds over the wide open unit disc $D^\circ(0,1)$, i.e. the affinoid unit disc minus the boundary annulus: indeed, if $E$ is any $p$-adic field, evaluating $f$ at an $E$-point of $D^\circ(0,1)$ assigns to $T$ an element of positive valuation in $E$. In particular, reducing modulo the maximal ideal of the valuation ring of $E$ gives back once more $f(0)$. Contrary to what happens over $D(0,p^{-1})$, $f$ itself is not a ``constant'' modulo $p$; reducing $f\in\Z_p\langle T\rangle$ modulo $p\Z_p\langle T\rangle$ can produce an arbitrary element of $\F_p[T]$.

We generalize this construction by replacing $D(0,1)$ with a rigid analytic space over $L$ of the form $\cX^\rig$ for an affine formal $\cO_L$-scheme $\cX$, the origin $x=0$ with an arbitrary $x\in X(L)$, and defining for every such $x$ and every $n\ge 1$ an explicit rational domain neighborhood $V^{(n)}_x$, and an explicit wide open neighborhood $U^{(n)}_x$ containing $V^{(n)}_x$, that play the role of $D(0,p^{-1})$ and $D^\circ(0,1)$ in the above picture for $n=1$. Such neighborhoods depend on the choice of a formal model $\cX$ of $X$; we refer to Definition \ref{resdef}. For $n=1$, one recovers the ``tube'' $U^{(1)}_x$ of $x$ in the sense of Berthelot. 
We prove the following.

\begin{thm}[cf. Theorem \ref{Unconst}]\label{intro2}
	Let $X$ be a rigid analytic space over $L$, and $\bV$ a $G$-representation on $X$ admitting a free lattice defined over an affine formal model $\cX$ of $X$. Then, for every $x\in X(L)$ and $n\in\Z_{\ge 1}$, $\bV$ is pointwise constant modulo $p^n$ over $U_{x,\cX}^{(n)}$, and constant modulo $\pi_L^n$ over $V_{x,\cX}^{(n)}$.
\end{thm}

%\begin{rem}
Under certain conditions on $\bV$, the domains in Theorem \ref{intro2} are optimal, i.e. one cannot have (pointwise) constancy mod $p^n$ on any larger neighborhood of $x$. We discuss this in Section \ref{optimalsec}.
%\end{rem}
%
%\begin{rem}
%For the applications we will consider later on, we point out the following important consequences of the results we just stated. 
If moreover one assumes $X$ to be strictly quasi-Stein, the condition about the existence of a free lattice is always verified under the assumptions of  Theorem \ref{intro1}. 
%Moreover, if one only cares about pointwise constancy for reductions modulo prime powers, the theorem \ref{intro2} also holds for a sheaf $\bV$ of $G$-representations that is not necessarily a $G$-representation (Corollary \ref{woloc}). Indeed, since integral model always exists locally, one can then proceed to apply locally Theorem \ref{intro2}. \\
%If one assumes that $X$ is not only wide open, but strictly quasi-Stein, then every sheaf of $G$-representations is automatically a $G$-representation, hence one can work in this greater generality (see Corollary \ref{quasistein}).
%\end{rem}

Hellmann \cite[Theorem 1.2]{hellfam} proves that a family of étale $(\varphi,\Gamma)$-modules over the tube of a point in a rigid space can always be converted into a $G_{\Q_p}$-representation. Since the obstruction to doing so is generally attached to the mod $p$ representation not being constant along the family, his result is compatible with ours for $n=1$. We actually rely on \emph{loc. cit.} in proving Corollary \ref{woloc}.

In Section \ref{secstcri}, we apply the above work to some examples of arithmetic interest, notably families of 2-dimensional crystalline and semistable representations of fixed weight of $\Gal(\Qp/\Q_p)$. We hope that our explicit results might help investigate ``deep'' congruences between classical eigenforms, i.e. congruences that are not explained by $p$-adic deformation along the eigencurve (for instance, congruences between eigenforms of the same weight).

%The fact that we only work with 2-dimensional representations in the applications is due to the fact that they can be described in terms of very few parameters, and this allows our results to take a simple form; there would be no other obstacle in working with higher-dimensional families. 

As mentioned above, 2-dimensional, irreducible crystalline representations of fixed weight $k\ge 2$ are parameterized by elements of the wide open unit disc $D^\circ(0,1)$ over $\Qp$. If $a\in D^\circ(0,1)(\Qp)$ and $U_a$ is a disc centered at $a$, of a sufficiently small explicit radius, then results of Berger, Berger--Li--Zhu and Torti \cite{berloccon,berlizhu,tortired}, i.e. the case $m=1$ of Theorem \ref{introred} above, give that the residual representation $\ovl\rho_{k,a}$ is constant along $U_a$. 
We equip $U_a$ with a $G$-representation and apply Theorem \ref{intro1} and Theorem \ref{intro2} to it. %in order to do so, we equip it with a G-representation. %the relevant deformation spaces with G-representations. 
In order to do so, we borrow an idea from \cite{berloccon} and \cite{tortilatt} and map them to the deformation space of \emph{trianguline} representations, as constructed by Colmez and Chenevier. Such a deformation space parameterizes trianguline $(\varphi,\Gamma)$-modules, and in dimension 2 it contains a dense open subspace corresponding to étale ones. Via arguments of Breuil--Hellmann--Schraen, one can identify a subspace $S_{2,\ovl\rho}^\sq$ of the étale locus along which the residual representation is a fixed $\ovl\rho$. Since $S_{2,\ovl\rho}^\sq$ is equipped with a map to the framed deformation space of $\ovl\rho$, one can pull back the universal framed deformation to obtain a $G$-representation on $S_{2,\ovl\rho}^\sq$, and on $U_a$ via the map $U_a\to S_{2,\ovl\rho}^\sq$ provided by the universal property. %\andr{need KPX to construct such a map?? interpolate parameters}

%\andr{doublecheck if need scaling by ramification}

Even though we stick to working over discs $U_a$ as above for simplicity, we could study local constancy of crystalline representations more precisely by relying on the work of Rozenzstajn. In \cite{rozcrys}, the author studies the locus $X(k,\ovl\rho)$ of $a\in D^\circ(0,1)(\Qp)$ where the semisimplification of $\ovl\rho_{k,a}$ is constant equal to a fixed $\ovl\rho$. She proves that $X(k,\ovl\rho)$ is a \emph{standard subset} of $D^\circ(0,1)$, i.e. a wide open disc with a finite number of closed discs removed from it (it is, in particular, a wide open), and studies its \emph{complexity}, a quantity that encodes information on how many discs are involved in the definition of $X(k,\ovl\rho)$, how large their fields of definition are, and the Hilbert--Samuel multiplicity of their rings of power-bounded functions, which in turn can be bounded in terms of the Serre weights of $\ovl\rho$ via the Breuil--Mézard conjecture. We could apply our Theorem \ref{intro1} to the wide open $X(k,\ovl\rho)$, and define as in Theorem \ref{intro2} subdomains on which the mod $p^n$ representations is (pointwise) constant. Since the data determining the complexity of $X(k,\ovl\rho)$ plays a role in the definition of such subdomains, it would be interesting to see if the combination of the results of \emph{loc. cit.} with ours has any interesting consequences. We plan to come back to this point in future work.

We work in a similar way in the semistable case, replacing $U_a$ with a constant-$\ovl\rho$ disc $U_\cL$ around an $\cL$-invariant $\cL\in\bP^1(\Qp)$. In this case, we only know of an explicit radius for $U_\cL$ in the case $\cL=\infty$, thanks to the work of Bergdall--Levin--Liu \cite{berlevliu}. 

We refer to the body of the paper for the complete statement of our results in the crystalline (Proposition \ref{appcr}, that refines a theorem of Torti \cite{tortired}, see Remarks \ref{tortirem} and \ref{rember}) and semistable cases (Proposition \ref{appst}, Theorem \ref{bll}). As an example of the results that we obtain, we give the following consequence of our work and \cite{berlevliu}, that makes \cite[Proposition 1.4]{tortilatt} explicit.

\begin{thm}[cf. Theorem \ref{bll}]\label{intro3}
Assume that $k\ge 4$ and $p\ne 2$. %\andr{odd throughout??}
Then, if $\cL\in\bP^{1,\rig}(L)$ satisfies
\[ v_p(\cL)<2-\frac{k}{2}-v_p((k-2)!)+1-n, \]
there exist lattices $\cV_\cL$ and $\cV_\infty$ in $V_{k,\cL}$ and $V_{k,\infty}$, respectively, such that $\cV_\cL\otimes_{\cO_L}\cO_L/\pi_L^{\gamma_{L/\Q_p}(n)}\cong\cV_\infty\otimes_{\cO_L}\cO_L/\pi_L^{\gamma_{L/\Q_p}(n)}$ as $\cO_L/\pi_L^{\gamma_{L/\Q_p}(n)}[\Gal(\Qp/\Q_p)]$-modules.
\end{thm}

%Recall that one can always replace $\gamma_{L/\Q_p}(n)$ with $n$ to obtain a simpler result. 

When $p$ is small, we apply a result of Chenevier to deduce from our work on representations of $\Gal(\Qp/\Q_p)$ some explicit local constancy results for representations of $\Gal(\ovl\Q/\Q)$ that are crystalline or semistable at $p$ (see Corollaries \ref{crysglob}, \ref{stglob}).

Finally, we give some simple consequences of our results for pseudorepresentations along rigid analytic spaces, that are not necessarily obtained as traces of sheaves of $G$-representations. We give an application to the mod $p^n$ variation of the pseudorepresentation carried by the Coleman--Mazur eigencurve (Proposition \ref{ecurvelc}). Our result is not very explicit unless one has strong information about the geometry of an irreducible component of the eigencurve; if such a component is ordinary (i.e. the generic fiber of a Hida family), one can make the result a little bit more concrete, as in the following corollary. Note that this is again a result about global Galois representations. 

\begin{cor}[cf. Corollary \ref{corhida}]\label{intro4}
Let $X$ be an ordinary component of a Coleman--Mazur eigencurve, and let $\kappa\in\cW(\Qp)$ be a weight with the property that $X$ has a single point $x$ of weight $\kappa$. Then, for $n$ sufficiently large, and for every $n\ge 1$ if $X$ is étale over $\cW$, the pseudorepresentation $T$ of $\Gal(\ovl\Q/\Q)$ carried by $X$ is pointwise constant mod $p^n$ over $\omega_X^{-1}(D^\circ(\kappa,p^{1-n}))$, and constant mod $p^n$ over $\omega_X^{-1}(D(\kappa,p^{-n}))$. \\
In particular, if $L$ is a $p$-adic field, the Hecke eigensystems away from $p$ of any two overconvergent eigenforms attached to $L$-points of $X$ of weight in $D^\circ(\kappa,p^{1-n})$ are congruent mod $\pi_L^n$.
\end{cor}

As Theorems \ref{intro1} and \ref{intro2} are independent of the choice of profinite group $G$ and of the rank of the $G$-representation $\bV$, one could hope to apply them to situations of arithmetic significance that go beyond the 2-dimensional, $G=\Gal(\Qp/\Q_p)$ case. Obstacles to doing so would be that the parameter spaces of fixed weight crystalline and semistable representations of higher dimension do not have as simple a description as in the 2-dimensional case, and that there is as yet no explicit description of constant $\ovl\rho$ families of such representations. 

We also remark that one might be able to make some constructions and statements more uniform and general by working with the category of adic spaces, that encompasses both formal schemes and their Berthelot generic fibers. This is the approach that Hellmann follows in \cite{hellfam}. We did not do so here, as the classical setup seemed to make the results more immediate to grasp, and we did not see a clear advantage in working with adic spaces in the applications. 

\smallskip

\noindent\textbf{Acknowledgments.} The second author wishes to thank Alexander Rahm and Gaetan Bisson at the Universit\'e de la Polyn\'esie fran\c{c}aise for the support. 
We both thank Alexandre Maksoud and Anna Medvedovsky for some interesting discussions around the results of this paper, and the anonymous referee for their thorough reading and numerous comments.

\smallskip

\noindent\textbf{Notation and terminology.}
%\andr{BUT so Chenevier says mod p is constant without saying a lattice exists???}
Given a category $\calC$, we will write $\calC^\op$ for its opposite category. \\
For any field $K$, we denote by $\ovl K$ an algebraic closure of $K$, and by $G_K=\Gal(\ovl K/K)$ the absolute Galois group of $K$. We write $\Mat_n(A)$ for the ring of $n\times n$ matrices over a ring $A$. \\
%When speaking of rigid analytic spaces we are thinking of them in Tate's sense, but we will use Raynaud's theory at some points. 
We work with rigid analytic space in the sense of Tate. %, but heavily rely on Raynaud's theory of formal models. 
For a rigid analytic space $X$ over a $p$-adic field $L$, we denote by $\cO_X$ the structure sheaf of $X$, and by $\cO_X^+$ the sheaf of power-bounded regular functions on $X$. %of norm bounded by 1 everywhere. 
We will typically denote formal schemes and sheaves on them by calligraphic letters (a sheaf $\cV$ on a formal scheme $\cX$), and rigid analytic spaces and sheaves of them by straight or bold letters (a sheaf $\bV$ on a rigid analytic space $X$). \\
We generally write $D(x,r)$ for an affinoid disc of some center $x$ and radius $r\in p^\Q$, and $D^\circ(x,r)$ for the corresponding ``wide open'' disc (i.e. $D(x,r)$ minus its boundary annulus). If $r=p^{-m}$ for some $m\in\Q$, we refer to $m$ as the \emph{valuation radius} of either $D(x,r)$ or $D^\circ(x,r)$.

\section{Preliminaries on formal schemes and rigid analytic spaces}\label{secrig}

%For any field $k$, we denote by $\Art_{k}$ the category of local Artinian $W(k)$-algebras with residue field $k$, and by $\wht\Art_k$ its completion. %If $\cO$ is a complete, local, Noetherian ring with residue field $k$, we denote by $\Art_{\cO}$?? check kisin
Let $L$ be a $p$-adic field with ring of integers $\cO_L$ and residue field $k_L$. We denote by $\pi_L$ a uniformizer of $L$.
We will need at various points the categories whose objects are as follows and morphisms are the obvious ones:
\begin{itemize}
\item $\Art_{\cO_L}$: local, Artinian $\cO_L$-algebras with residue field $k_L$;
\item $\wht\Art_{\cO_L}$: the completion of $\Art_{\cO_L}$;
	\item $\Ad_{\cO_L}$: Noetherian, adic $\cO_L$-algebras with ideal of definition $I$, such that $A/I$ is a finitely generated $k_L$-algebra ($\mathrm{ff}$ stands for ``formally finite'').
%	\item $\Adm_{\cO_L}$: admissible, complete, local, Noetherian $\cO_L$-algebras;
	\item $\Aff_{\cO_L}$: affine formal schemes, i.e. of the form $\Spf\cA$ for $\cA\in\Ad_{\cO_L}$. To any such scheme we attach the \emph{special fiber} $\Spec\cA/I$, $I$ an ideal of definition of $\cA$.
	\item $\Frm_{\cO_L}$: locally Noetherian, adic formal schemes (i.e. formal schemes constructed by gluing objects in $\Aff^\form_{\cO_L}$) whose special fiber (constructed by gluing special fibers on an affine covering) is a scheme locally of finite type over $\Spec k_L$. %quasi-paracompact admissible formal schemes over $\cO_L$;
	\item $\Aff_L$: affinoid $L$-algebras.
	\item $\Rig^\aff_L$: affinoid rigid analytic spaces over $L$, i.e. of the form $\Spm A$ for $A\in\Aff_L$.
	\item $\Rig_L$: rigid analytic spaces over $L$.
\end{itemize}
We recall that an $\cO_L$-algebra $\cA$ is an object of $\Ad_{\cO_L}$ if and only if it is a quotient of $\cO_L\langle\zeta_1,\ldots,\zeta_r\rangle[[\xi_1,\ldots,\xi_s]]$ for some $r,s\in\Z_{\ge 0}$. 
Taking the formal (respectively, maximal) spectrum gives an equivalence of categories $\Aff^\form_{\cO_L}\cong(\Ad_{\cO_L})^\op$ (respectively, $\Aff^\rig_L\cong\Aff_L^\op$), while $\Frm_{\cO_L}$ and $\Rig_L$ are obtained respectively from $\Aff^\form_{\cO_L}$ and $\Aff^\rig_L$ by gluing.  
%
%MAYBE ASSUME we already know how to define functor??? maybe not every functor on art k gives functor on rig, but in this case we know how to define it
%
%Raynaud's theorem \cite[Section 8.4, Theorem 3]{boschbook} 
Berthelot's construction, as presented in \cite[Sections 7.1.1 and 7.2.3]{DJ95}, provides us with a \emph{generic fiber} functor $\Frm_{\cO_L}\to\Rig_L, \cX\mapsto\cX^\rig$ 
%\begin{align*} \Frm_{\cO_L}&\to\Rig_L \\
%	\cX&\mapsto\cX^\rig \end{align*}
that factors through the localization of $\Frm_{\cO_L}$ at admissible formal blow-ups. %, and induces an equivalence of categories $\Frm^{\mathrm{bw}}_{\cO_L}\cong\Rig_L$. 

Given a rigid analytic space $X$ over $L$, a \emph{formal model} of $X$ is an object $\cX$ of $\Frm_{\cO_L}$ equipped with an isomorphism $\cX^\rig\cong X$. For any such $\cX$, one defines a canonical homomorphism
\begin{equation}\label{dejong} \mathcal{O}_\mathcal{X}(\cX)\otimes_{\cO_L}L\to\cO_X(X), \end{equation}
as in \cite[Section 7.1.8]{DJ95}. The image of $\cO_\cX(\cX)$ under \eqref{dejong} lies in $\cO_X^+(X)$, and the resulting map $\cO_\cX(\cX)\to\cO_X^+(X)$ is an isomorphism if $\cX$ is flat and normal \cite[Theorem 7.4.1]{DJ95}. %\andr{does admissible imply flat and formally normal??}

\subsection{Wide open and strictly quasi-Stein spaces}\label{quasisteinrep}

We prefer Berthelot's generic fiber functor over Raynaud's since it allows us to work with ``wide open'' rigid analytic spaces, as in the definition just below. However, there is no simple description of the image of Berthelot's functor, i.e. no simple condition under which a rigid analytic space admits a formal model. %This will not be a problem for us, since we will restrict our focus to rigid analytic spaces that are constructed as rigid generic fibers of affine formal schemes.

\begin{defin}\label{defwo}
We say that a rigid analytic space $X$ over $L$ is \emph{wide open} if it is reduced and can be admissibly covered by an increasing union of affinoids $\bigcup_{i\in\N}X_i$ with the property that the restriction map $\cO_{X_{i+1}}(X_{i+1})\to\cO_{X_i}(X_i)$ is compact for every $i\ge 0$. %We denote by $\Rig^\wo_L$ the full subcategory of $\Rig_L$ whose objects are the wide open rigid analytic spaces. \andr{this used? better to introduce quasi-stein?}
\end{defin}

%\andr{have to require reduced+f-affine??}

%\andr{add reducedness everywhere}

If $X$ is a wide open rigid analytic space and $(X_i)_i$ is a sequence as in Definition \ref{defwo}, then $\cO_X(X)$ is the projective limit of the affinoid Tate algebras $\cO_{X_i}(X_i)$ with respect to the restriction maps, and $\cO_X^+(X)$ is the projective limit of the $\cO_{X_i}^+(X_i)$ also with respect to restrictions. It is in particular Hausdorff, since each of the affinoid Tate algebras $\cO_{X_i}^+(X_i)$ is Hausdorff. 

%Moreover, by following Berthelot's construction we obtain an isomorphism
%\[ X\cong(\Spf\cO_X^+(X))^\rig, \]

\begin{rem}\label{wideopenfact}
	If $X$ is wide open, then by \cite[Lemma 7.2.11]{bellchen} (where the term ``nested'' is used instead of wide open), $\cO_X^+(X)$ is a compact ring. Being compact and Hausdorff, $\cO_X^+(X)$ is a profinite ring. 
\end{rem}

We denote by $\Alg^\prof_{\cO_L}$ the full subcategory of profinite objects in $\Ad_{\cO_L}$, and by $\Aff^\prof_{\cO_L}$ the category of affine formal schemes of the form $\Spf\cA$, with $\cA\in\Alg^\prof_{\cO_L}$.

\begin{rem}\label{woiff}
%\begin{enumerate}[label=(\roman*)]
An object of $\Ad_{\cO_L}$ is profinite if and only if it belongs to $\wht\Art_{\cO_L}$, if and only if it is equipped with the topology induced by its maximal ideal, if and only if it is a quotient of $\cO_L[[\xi_1,\ldots,\xi_s]]$ for some $s\in\Z_{\ge 1}$.
Let $\cX=\Spf\cA$ for some $\cA\in\Ad_{\cO_L}$. Its generic fiber is:
\begin{itemize}
\item affinoid if and only if $(\pi_L)$ is an ideal of definition of $\cA$ (i.e., $\cA$ is a quotient of $\cO_L\langle\zeta_1,\ldots,\zeta_r\rangle$ for some $r\in\Z_{\ge 1}$), in which case $\cX^\rig=\Spm\cA\into\cA\otimes_{\cO_L}L$.

\item wide open if and only if $\cA\in\Alg^\prof_{\cO_L}$. %is a quotient of $\cO_L[[\xi_1,\ldots,\xi_s]]$ for some $s\in\Z_{\ge 1}$. 
We already explained the ``only if''. For the converse, we take $A=\cA$ in the construction of \cite[Section 7.1.1]{DJ95}, so that the generic fiber $(\Spf\cA)^\rig$ is given in \cite[Definition 7.1.3]{DJ95} as an increasing admissible union $\bigcup_{i\in\N}\Spm C_i$. A direct check shows that if $\cA$ is a quotient of $\cO_L[[\xi_1,\ldots,\xi_s]]$, then the maps $C_{i+1}\to C_i$ are compact.
%
%Via the above observation, we can write a sequence of equivalences of categories
%\[ (\Alg^\prof_{\cO_L})^\op\to\Aff^\prof_{\cO_L}\to\Rig^\wo_L\to(\Ad^\prof_{\cO_L})^\op, \]
%where the first functor is $\Spf$, the second one is Berthelot's rigid analytic generic fiber, the last one maps $X$ to $\cO_X^+(X)$, and is a quasi-inverse to the composition of the first two.
\end{itemize}
%\end{enumerate}
\end{rem}

Let $e$ be the ramification index of $L/\Q_p$ and $\pi_L$ a uniformizer. Let $R$ be a positive real number of the form $p^{m/e}$ for some $m\in\Z$.
Basic examples for the two cases appearing in Remark \ref{woiff} are:
\begin{itemize}
\item The rigid analytic affinoid disc $D(0,R)$ of center 0 and radius $R$, with its formal model 
%defined as
%\[ D(0,R)=\Spm L\langle T,U\rangle/(T-\pi_L^mU), \]
%where $\pi_L$ is any uniformizer of $L$. An obvious formal model for $D(0,R)$ is
\[ \cD(0,R)=\Spf\cO_L\langle T,U\rangle/(T-\pi_L^mU), \]
If $R=1$, such a model is initial in a sense that we make precise in Remark \ref{maxmodel}. 
We can also define $D(0,R)$ over $\Q_p$ as $D(0,R)_{\Q_p}=\Spm\Q_p\langle T,U\rangle/(T^e-p^mU)$, 
with the formal model $\cD(0,R)_{\Q_p}=\Spf\Z_p\langle T,U\rangle/(T^e-p^mU)$. 
\item The rigid analytic wide open disc $D(0,R)^\circ$ of center 0 and radius $R$, defined as $(\Spf\cO_L[[\pi_L^mT]])^\rig$. 
It is a wide open, since it can be written as an increasing union of affinoid discs of radii $p^{r}$, with $r\in\Q$ and $r\to m/e$. We can define all such discs over $\Q_p$ as in the previous point.
\end{itemize}

\noindent We introduce a slightly stronger property than wide openness.

\begin{defin}[{cf. \cite[Definition 2.1.17]{emeloc}}]%[{cf. \cite[Definition 7]{ochsethesis}}]
A rigid analytic space $X$ over $L$ is said to be \emph{quasi-Stein} if it is reduced and can be admissibly covered by an increasing union of affinoids $\bigcup_{i\in\N}X_i$ with the property that the restriction map $\cO_{X_{i+1}}(X_{i+1})\to\cO_{X_i}(X_i)$ is of dense image for every $i$. We say that $X$ is \emph{strictly quasi-Stein} if one can choose the $X_i$ in such a way that $\cO_{X_{i+1}}(X_{i+1})\to\cO_{X_i}(X_i)$ is moreover compact for every $i$. 
\end{defin}

In other words, $X$ is strictly quasi-Stein if it is wide open and, moreover, one can take the restriction maps in Definition \ref{defwo} to be of dense image. Note that the strictness condition excludes affinoid spaces, which are a standard example of quasi-Stein spaces. %A wide open disc of any dimension and radius, with its obvious affinoid covering, is strictly quasi-Stein. 
%We recall the following result on quasi-Stein spaces.

%Our interest in the above definition lies in the following classical result.

%\begin{thm}[{\cite[Satz 2.4]{kiehlAB}}]
%If $\bV$ is a coherent sheaf on a quasi-Stein rigid space $X$, then the cohomology groups $H^i(X,\bV)$ vanish for $i\ge 1$.
%\end{thm}

\begin{rem}\label{rem:kiehl}
By \cite[Satz 2.4]{kiehlAB}, if $\bV$ is a coherent sheaf on a quasi-Stein rigid space $X$, then the cohomology groups $H^i(X,\bV)$ vanish for $i\ge 1$.
In particular, every sheaf of $G$-representations on a quasi-Stein space $X$ (see \Cref{def:Gsheaf} below) is a $G$-representation. This will allow us to apply our results to certain simple rigid analytic spaces, carrying sheaves of crystalline or semistable representations, without worrying whether such sheaves are free and associated (as in the proof of Corollary \ref{woloc}). 
\end{rem}

If $\cA\in\Alg^\prof_{\cO_L}$, then $(\Spf\cA)^\rig$ is strictly quasi-Stein: indeed, the transition maps in the construction of \cite[Section 7.1.1]{DJ95} are immediately seen to have dense image. In particular, wide open discs are strictly quasi-Stein. We prove a partial converse below. 
We say that a rigid space or formal scheme $X$ is integral if it is reduced and irreducible. 
%For lack of a reference, we show that for certain rigid spaces $X$, applying Berthelot's construction to $\cO_X^+(X)$ gives back the space $X$. 

\begin{prop}\label{converse}
Assume that $X$ is integral and strictly quasi-Stein. There is an isomorphism of rigid $L$-spaces $X\cong(\Spf\cO_X^+(X))^\rig$.
\end{prop}

\begin{proof}
Write $X=\bigcup_iX_i$ as an increasing union of integral affinoid subspaces such that the transition maps are compact and dense. Each $X_i$ admits $\Spf\cO_X^+(X_i)$ as a formal model. Since restriction to $X_i$ gives a map $\mathrm{res}_i:\cO_X^+(X)\to\cO_X^+(X_i)$, taking formal spectra and rigid generic fibers gives us a map $\iota_i:X_i\to(\Spf\cO_X^+(X))^\rig=:Y$. Since the maps $\mathrm{res}_i$ are compatible with the restrictions $\cO_X^+(X_{i+1})\to\cO_X^+(X_i)$, the $\iota_i$ are compatible with the inclusions $X_i\into X_{i+1}$, and we obtain a map 
\[ \iota:X=\bigcup_iX_i\to(\Spf\cO_X^+(X))^\rig. \]
We claim that the points of both sides are in bijection with the maximal ideals of $\cO_X^+(X)\otimes_{\Z_p}\Q_p$. For the right-hand side, this is \cite[Lemma 7.1.9]{DJ95}. For the left-hand side, let $\fm$ be a maximal ideal of $\cO_X^+(X)\otimes_{\Z_p}\Q_p$. Since $\cO_X^+(X)\otimes_{\Z_p}\Q_p$ is the ring of rigid analytic functions on $X$, it is the inverse limit of the $\cO_X(X_i)$. In particular, there exists $i$ such that the image $\fm_i$ of $\fm$ under $\cO_X^+(X)\otimes_{\Z_p}\Q_p\to\cO_X(X_i)$ is non-zero. The image of the resulting map $\cO_X^+(X)\otimes_{\Z_p}\Q_p/\fm\to\cO_X(X_i)/\fm_i$ is a subfield of $\cO_X(X_i)/\fm_i$, and is dense because of \cite[Satz 2.4.1]{kiehlAB}. %Theorem 4 in https://kskedlaya.org/18.727/kiehl.pdf (original reference?). 
In particular, $\cO_X(X_i)/\fm_i$ is itself a field, hence $\fm_i$ is maximal and is attached to a point of $X_i\subset X$. Conversely, if $\fm_i$ is a maximal ideal of $X_i$, the kernel of $\cO_X^+(X)\times_{\Z_p}\Q_p\to\cO_X(X_i)/\fm_i$ is a maximal ideal, since the map is non-zero by density of the image.

The above bijections are clearly compatible with $\iota$.
If $x$ is a point of $X$, the stalk $\cO_{X,x}$ is obtained by localizing $\cO_X^+(X)\otimes_{\Z_p}\Q_p$ at the maximal ideal $\fm$ attached $x$ and completing $p$-adically. Since the maximal ideal attached to $\iota(x)$ is again $\fm$, $\cO_{Y,\iota(x)}$ is obtained via the same operation and $\iota$ induces an isomorphism $\cO_{Y,\iota(x)}\cong\cO_{X,x}$. Therefore, $\iota$ is an open immersion that is bijective on points, hence an isomorphism.
\end{proof}

We denote by $\Rig^\wo_L$ the full subcategory of $\Rig_L$ whose objects are the integral strictly quasi-Stein rigid analytic spaces.

%The proposition implies that $\Spf\cO_X^+(X)$ is a formal model of $X$. Such a model is initial in the sense of Remark \ref{maxmodel}.

%As we explain below, the association $X\mapsto\cO_X^+(X)$ allows us to give an algebraic characterization of wide open rigid analytic spaces. %$\Rig^\wo_L$.

%\begin{defin}\label{defoc}
%We denote by $\Alg^\prof_{\cO_L}$ the full subcategory of profinite objects in $\Ad_{\cO_L}$, and by $\Aff^\prof_{\cO_L}$ the category of affine formal schemes of the form $\Spf\cA$, with $\cA\in\Ad^\prof_{\cO_L}$.
%\end{defin}

%The above remark justifies the following definition.
%
%\begin{defin}
%	We say that a rigid analytic space is \emph{f-affine} if it admits an affine formal model. We denote by \Rig^{f-aff}_L the full subcategory of \Rig_L whose objects are f-affine rigid analytic spaces.
%\end{defin}

%\andr{reference for existence?? does not exist, check formally finite--mention}

\begin{rem}\label{rem:DJ719}
Let $\cA\in\Ad_{\cO_L}$ be reduced, and $X=(\Spf\cA)^\rig$. For a finite extension $E/L$ and a point $x$ of $X$, the map $\cA\to\cO_X^+(X)\xto{x}\cO_E\onto k_E$ factors through a map $\ovl x:\cA/\pi_L\to k_E$. 
By \cite[Lemma 7.1.9]{DJ95}, the points of $X$ are in bijection with the maximal ideals of $\cA\otimes_{\Z_p}\Q_p$, which in turn are in bijection with the height 1 primes of $\cA$ not containing $\pi_L$. If $\fp_x$ is the prime ideal of $\cA$ corresponding to the point $x$, then the kernel of $\ovl x$ is the maximal ideal of $\cA/\pi_L$ generated by $\fp_L$. The following commutative diagram summarizes the various correspondences:
		\begin{center}\label{ovlx}
		\begin{tikzcd}
		\cA\otimes_{\Z_p}\Q_p\arrow[r,"x"] & \cO_X(X)\arrow[r,"x"] & E \\
			\cA\arrow[r]\arrow[d]\arrow[u] & \cO_X^+(X)\arrow[r, "x"]\arrow[u,hookrightarrow] & \cO_E\arrow[u,hookrightarrow]\arrow[d,twoheadrightarrow] \\
			\cA/\pi_L\arrow[rr,"\ovl x"] & \mbox{ } & k_E
		\end{tikzcd}
	\end{center}
\end{rem}

\begin{rem}\label{maxmodel}\mbox{ }
Let $\cA\in\Ad_{\cO_L}$, and $X=(\Spf\cA)^\rig$. Then $\cO_X^+(X)\in\Ad_{\cO_L}$ and $\Spf\cO_X^+(X)$ is an affine formal model of $X$. If $\cX$ is any other affine formal model of $X$, then \eqref{dejong} provides us with a map $\cO_\cX(\cX)\to\cO_X^+(X)$ in $\Ad_{\cO_L}$, hence a map $\Spf\cO_X^+(X)\to\cX$ in $\Aff^\form_{\cO_L}$ that induces an isomorphism on the generic fibers. Therefore, $\Spf\cO_X^+(X)$ is an initial object in the category of affine formal models of $X$.
\end{rem}

\subsection{Functors on Artin rings, formal schemes and rigid analytic spaces}\label{funrig}

Let $F\colon\Art_{\cO_L}\to\Sets$ be a functor. We can extend $F$ to a continuous functor $\wht F\colon\wht\Art_{\cO_L}\to\Sets$ in a unique way: if $\fm_A$ is the maximal ideal of $A\in\wht\Art_{\cO_L}$, we set $\wht F(A)=\varprojlim_n F(A/\fm_A^n)$, where the transition maps come from functoriality. Note that the Noetherian objects in $\Art_{\cO_L}$ belong to $\Ad_{\cO_L}$.

Assume that $F$ is pro-represented by a universal pair $(R,\xi)$, where $R$ is a Noetherian object in $\wht\Art_{\cO_L}$ and $\xi\in F(R)$. Since $R\in\Ad_{\cO_L}$, it gives rise to an affine formal scheme $\Spf R\in\Frm_{\cO_L}$ and a rigid analytic space $(\Spf R)^\rig\in\Rig_L$, so that we can consider the functors
\begin{gather*} F^\ff\coloneqq\Hom(R,-)\colon\Ad_{\cO_L}\to\Sets, \quad F^\form\coloneqq\Hom(-,\Spf R)\colon\Frm_{\cO_L}\to\Sets, \\
F^\rig\coloneqq\Hom(-,(\Spf R)^\rig)\colon\Rig_L\to\Sets. \end{gather*}
We attach to $\xi\in F(R)$ elements $\xi^\form\in F^\form(\Spf R)=\Hom(\Spf R,\Spf R)$ and $\xi^\rig\in F^\rig((\Spf R)^\rig)=\Hom((\Spf R)^\rig,(\Spf R)^\rig)$ in the obvious way. By definition, $(\Spf R,\xi^\form)$ is a universal pair for $F^\form$, and $((\Spf R)^\rig,\xi^\rig)$ is a universal pair for $F^\rig$.

The above definition are quite trivial, but allow us to define representable functors on some formal schemes and rigid spaces starting from a representable functor on Artin rings. It would be more satisfying to start with a (not necessarily) representable functors on Artin rings, and produce functors on formal schemes (by covering them with formal schemes of the type $\Spf\cA$, $\cA\in\Ad_{\cO_L}$) and rigid spaces (by covering them with rigid spaces of the form $(\Spf\cA)^\rig$, $\cA\in\Ad_{\cO_L}$, i.e. with wide open subspaces; this is what is sometimes called an ``overconvergent'' covering in the literature). However, we do not know of a result granting the existence of the required coverings in general. For the same reason, we are unable to check that the above (tautologically representable) functors actually parameterize (pseudo-)deformations of a given residual representation to all formal schemes or rigid spaces. It is however, sufficient to our purposes to know that they do so over certain subcategories of $\Frm_{\cO_L}$ and $\Rig_{\cO_L}$, namely the category of formal spectra of Noetherian pro-$p$ rings, and that of wide open rigid analytic spaces (see Propositions \ref{psdefequiv} and \ref{defequiv}). %: in the cases of interest to us, $F$ will be a representable group (pseudo-)deformation functor, and the values of $F^\form$ on the formal spectra of pro-$p$ rings, and of $F^\rig$ on wide open rigid spaces, will also parameterize group (pseudo-)deformations in a reasonable sense (see Propositions \ref{psdefequiv} and \ref{defequiv}). %, and \ref{defsqequiv}).

\subsection{Reminders on pseudorepresentations}

In order to simplify the exposition, we work with pseudorepresentations in the classical sense rather than the more modern notion of Chenevier's determinants. We will assume that the characteristic of our coefficient field is large enough compared to the dimension of the pseudorepresentation, in order to avoid the issues described in \cite[Introduction, footnote 5]{chendet}.

Let $d$ be a positive integer. 
Let $A$ be a topological ring in which $d!$ is invertible, and $R$ a topological $A$-algebra. We refer to \cite{taylorgal,chendet} for the definition of a  $d$-dimensional pseudorepresentation $T:R\to A$. 
%
%We denote by $\cS_{d+1}$ the group of permutations of the set $\{0,\ldots,d\}$, and by $\vareps(\sigma)$ the sign of an element $\sigma\in\cS_{d+1}$.
%
%\begin{defin}
%	A \emph{(continuous) $A$-valued, $d$-dimensional pseudorepresentation of $R$} is a (continuous) $A$-linear map $T\colon R\to A$ satisfying:
%	\begin{enumerate}[label=(\roman*)]
%		\item $T(1)=d$,
%		\item for every $r_1,r_2\in R$, $T(r_1r_2)=T(r_2r_1)$,
%		\item for every $r_0,\ldots,r_d\in R$, $\sum_{\sigma\in\cS_{d+1}}\vareps(\sigma)T_\sigma(r_0,\ldots,r_d)=0$,
%	\end{enumerate}
%	where $T_\sigma$ is defined as follows: if $\sigma=(i_1^1,\ldots,i_1^{m_1})\cdots(i_k^1,\ldots,i_k^{m_k})$ is a decomposition of $\sigma$ in disjoint cycles, then $T_\sigma=T(i_1^1\cdots i_1^{m_1})+\ldots+T(i_k^1\cdots i_k^{m_k})$.
%	
If $G$ is a topological group, an \emph{$A$-valued pseudorepresentation of $G$} is a pseudorepresentation of the group algebra $A[G]$. We also refer to its restriction $T\colon G\to A$ as a pseudorepresentation. 
%	
%	We say that a $d$-dimensional pseudorepresentation $T$ is \emph{irreducible} if it cannot be written as a sum $T=T_1+T_2$ of two pseudorepresentations (necessarily, of dimensions $d_1,d_2\in\Z_{\ge 1}$ with $d_1+d_2=d$). When $A$ is a field, we say that $T$ is \emph{absolutely irreducible} if $T\otimes_A\ovl A$ is irreducible.
%\end{defin}
%
%In order for a map $T\colon G\to A$ to be a pseudorepresentation, it is enough that conditions (ii,iii) are satisfied on elements of $G$. Continuity can also be checked on $G$.
%
%Given a pseudorepresentation $T\colon R\to A$, we recall that the \emph{kernel} of $T$ is the $A$-submodule of $R$ defined by
%\[ \ker{T}=\{y\in R\,\vert\,T(xy)=T(x)\,\forall x\in R \}. \]
%When $G$ is a group and $T\colon G\to A$ a pseudorepresentation, we denote by $\ker T$ the normal closed subgroup $\ker{T}=\{y\in G\,\vert\,T(xy)=T(x)\,\forall x\in G \}$. 
%If $T$ is continuous, $\ker{T}$ is closed.
%
Given a pseudorepresentation $T\colon R\to A$ and an $A$-algebra $B$, we denote by $T\otimes_AB$ the pseudorepresentation $T\otimes_AB\to B$ induced from $T$. For later use, we prove a lemma.

%\begin{lemma}\label{kerext}
%	Let $A$ be a complete, local, Noetherian ring and let $T\colon R\to A$ be a pseudorepresentation. Let $B$ be a flat $A$-algebra. Then $\ker(T\otimes_AB)=(\ker T)\otimes_AB$.
%\end{lemma}
%
%This is essentially \cite[Remark 1.2.2]{bellchen}, which relies on \cite[Proposition 2.11]{rouquier} and also gives alternative assumptions on $A$ and $B$ for which the result holds.
%
%\begin{proof}
%	Obviously, $\ker T\otimes_AB\subset\ker(T\otimes_AB)$, so that $T\otimes_AB$ induces a pseudorepresentation $(T\otimes_AB)^\prime\colon R\otimes_AB/((\ker T)\otimes_AB)\to B$. 
%	Since $B$ is $A$-flat,
%	\begin{equation}\label{kerT} R\otimes_AB/((\ker T)\otimes_AB)\cong (R/\!\ker T)\otimes_AB.
%	\end{equation}
%	
%	Let $T^\prime\colon R/\!\ker T\to A$ be the faithful pseudorepresentation induced from $T$.
%	By \cite[Remark 1.2.2]{bellchen}, $T^\prime\otimes_AB\colon (R/\!\ker T)\otimes_AB\to B$ is faithful. Obviously, $T^\prime\otimes_AB$ is identified with $(T\otimes_AB)^\prime$ via \eqref{kerT}. We conclude that $(T\otimes_AB)^\prime$ is faithful, which means that $\ker(T\otimes_AB)\subset(\ker T)\otimes_AB$, giving the missing inclusion.
%\end{proof}

We recall another definition. Let $A$ be a local ring with maximal ideal $\fm_A$ and residue field $\F$, equipped with the $\fm_A$-adic topology. Let $T\colon G\to A$ a continuous pseudorepresentation.

\begin{defin}[{\cite[Definition 1.4.1]{bellchen}}]
	We say that $T$ is \emph{residually multiplicity-free}, or that $\ovl T$ is multiplicity-free, if there exist pairwise non-isomorphic, absolutely irreducible representations $\ovl\rho_i\colon G\to\Mat_{d_i}(\F)$, $i=1,\ldots,k$, such that $\ovl T=\sum_{i=1}^k\tr\ovl\rho_i$.
\end{defin}

%Observe in particular that the Jordan--H\"older factors of the semisimplification of \ovl\rho are required to be \emph{absolutely} irreducible.

%Finally, we refer to \cite[Definition 1.3.1]{bellchen} for the definition of a \emph{generalized matrix algebra}, GMA in short. Let $R$ be a GMA over a ring $A$. As per \emph{loc. cit.}, $R$ comes equipped with a \emph{trace} function $R\to A$. Following \cite[Section 1.3.3]{bellchen}, given a positive integer $n$ and an $A$-algebra $B$, we call \emph{trace representation} of a GMA an $A$-algebra homomorphism $R\to\Mat_n(B)$ that commutes with the trace maps on the two sides.

\section{Sheaves of $G$-representations}

Let $G$ be a compact topological group. 
In this section, we introduce the (standard) notion of sheaf of $G$-representations over a rigid analytic space $X$, and use a result of Bellaiche--Chenevier to find a ``lattice'' in it (i.e. an integral model) under some reasonable assumptions on the sheaf and on its underlying space (see Proposition \ref{prop:factorial}).

In the following definitions, $X$ can be either a rigid analytic space over $L$, or an admissible $\cO_L$-formal scheme. By ``open'' in $X$ we will mean either an admissible open if $X$ is a rigid space, or an open admissible formal subscheme if $X$ is a formal scheme.

\begin{defin}\label{def:Gsheaf}
	An $\cO_X$-linear action of $G$ on an $\cO_X$-module $\bV$ is the datum, for every open $U\subset X$, of a continuous homomorphism
	\[ \rho_U\colon G\to\Aut_{\cO_X(U)-\mathrm{mod}}(\bV(U)), \]
	such that, for every inclusion $V\subset U$ of opens in $X$, the composition of $\rho_{U}$ with the restriction $\bV(U)\to\bV(V)$ is conjugate to $\rho_{V}$ in $\Aut_{\cO_X(V)}\bV(V)$. %\andr{Need locally free in the definition?}
	
	A \emph{sheaf of $G$-representations on $X$} is a locally free $\cO_X$-module $\bV$ of finite rank carrying an $\cO_X$-linear action of $G$. %module rather than sheaf????
\end{defin}

Let $X$ be either a $\cO_L$-formal scheme or a rigid analytic space over $L$, and let $M$ be a free $\cO_X(X)$-module $M$ of finite rank equipped with a continuous homomorphism
\[ G\to\Aut_{\cO_X(X)}(M). \]
We attach to $M$ the sheaf $\bV_M$ of $G$-representations on $X$ obtained from the presheaf $U\mapsto M\otimes_{\cO_X(X)}\cO_X(U)$, equipped with the $G$-action on the first factor. 

\begin{defin}
	A \emph{$G$-representation on $X$} is a sheaf of $G$-representations on $X$ of the form $\bV_M$ for $M$ as above. Given such a $\bV=\bV_M$, we denote with $\rho_\bV:G\to\Aut_{\cO_X(X)}(M)$ the associated homomorphism. We sometimes refer to $\rho_\bV$ itself as a $G$-representation on $X$. \\
	%the datum of a free $\cO_X(X)$-module $M$ of finite rank and a continuous homomorphism
	%\[ \rho_M\colon G\to\Aut_{\cO_X(X)}(M). \]
	%We attach to $\rho_M$ the sheaf $\bV$ of $G$-representations on $X$ obtained from the presheaf $U\mapsto M\widehat\otimes_{\cO_X(X)}\cO_X(U)$, equipped with the $G$-action given by $\rho_M$ on the first factor. We sometimes refer to $\bV$ arising in this way as to a $G$-representation on $X$. \\
	%free associated sheaf of $G$-representations on $X$. \\
	Assume that $X$ is reduced and irreducible, so that $\Frac(\cO_X(X))$ is a field that we denote by $K_X$. We say that a $G$-representation $\bV$ on $X$ is \emph{irreducible} if $\rho_\bV\otimes_{\cO_X(X)}K_X$ is irreducible. We say that $\bV$ is \emph{absolutely irreducible} if $\rho_\bV\otimes_{\cO_X(X)}K_X^\sep$ is absolutely irreducible. %\andr{use pseudoreps to define irred??}
\end{defin}

\begin{rem}\label{coverGrep}
	Let $\bV$ be a sheaf of $G$-representations on $X$. Then there exists a covering $\fU$ of $X$ such that, for every $U\in\fU$, $\bV\vert_U$ is a $G$-representation. Indeed, by definition, a sheaf of $G$-representations on $X$ is a locally free $\cO_X$-module $\bV$ equipped with a $G$-action, i.e. there exists a covering $\fU$ of $X$, which is \emph{admissible} in the case when $X$ is a rigid space, such that $\bV\vert_U$ is free as an $\cO_U$-module for every $U\in \fU$. By refining the covering $\fU$, one can assume that every $U\in\fU$ is affine/affinoid, so that $\bV\vert_U$ is a $G$-representation.
\end{rem}

If $X=\Spm A$ is an affinoid, a $G$-representation on $X$ is what is called an \emph{$A$-representation of $G$} in \cite{kedliufam}.

\begin{rem}
	%\begin{enumerate}
	%		\item We stress again that being associated only depends on the space $X$; we attach the phrase ``associated'' to the sheaf so we can speak of $G$-representations without making assumptions on $X$ (if $\cO_X$ is not associated, no $G$-representation on $X$ will exist). %, so whenever we know a priori that \cO_X is generated by global sections (e.g. in the setting of Lemma \ref{globgen}), a G-representation is 
	%		\item By the above discussion, the datum of a $G$-representation on $X$ (such that $\cO_X(X)$ is associated) is the same as that of a free $\cO_X(X)$-module $M$ of finite rank and a continuous representation $\rho_\bV\colon G\to\Aut_{\cO_X(X)}M$. We call $\rho_\bV$ as above the representation associated with $\bV$. We sometimes abuse of terminology and refer to $\rho_\bV$ as a $G$-representation on $X$.
	%\item 
	It is immediate from the definition that a $G$-representation $\bV$ on $X$ is irreducible if and only if its only sub-$G$-representations are 0 and $\bV$. It is not obvious that an irreducible $\bV$ does not admit any non-trivial subsheaves of $G$-representations.
	%\end{enumerate}
\end{rem}

%WHEN is this automatically true?? constant reduction??

For every sheaf of $G$-representations $\bV$ on $X$, there exists an affine (if $X$ is a formal scheme) or affinoid (if $X$ is a rigid analytic space) covering $\cU$ of $X$ with the property that, for every $U\in\cU$, the restriction of $\bV$ to $U$ is a $G$-representation on $U$.

\begin{defin}\label{deftrace}
	When $X$ is either an affine $\cO_L$-formal scheme or an affinoid rigid analytic space over $L$, we define the \emph{trace} of a $G$-representation $\bV$ on $X$ as the trace $\tr\bV\colon G\to\cO_X(X)$ of the homomorphism $\rho_\bV\colon G\to\Aut_{\cO_X(X)}(\bV(X))$ associated with $\bV$. \\ For an arbitrary $\cO_L$-formal scheme or rigid analytic $L$-space $X$ and a sheaf $\bV$ of $G$-representations on $X$, we define the trace of $\bV$ as the function $\tr\bV\colon G\to\cO_X(X)$ defined, on every open $U$ over which $\bV\vert_U$ is a $G$-representation, as the trace of $\bV\vert_U$.
\end{defin}

\begin{rem}\mbox{ }
	\begin{itemize}
		\item The second part of Definition \ref{deftrace} produces, as desired, a map $G\to\cO_X(X)$. Indeed, by Remark \ref{coverGrep}, there exists a covering $\fU$ of $X$, admissible in case $X$ is a rigid space, such that $\bV\vert_U$ is a $G$-representation on $U$; therefore, for every $g\in G$, the collection of traces $\tr(\bV\vert_U)(g)$, $U\in\fU$, glues to a global function on $X$. Here we are using in a crucial way the admissibility of $\fU$ in the rigid case: see Section \ref{klsec} for a further discussion on this.
		\item Given a rigid analytic space $X$ and a sheaf $\bV$ of $G$-representations on $X$, one can compute the trace of $\bV$ on any sheaf of lattices, in the following sense: if $\cV$ is a sheaf of lattices in $\bV$, defined over a formal model $\cX$ of $X$, then the trace of $\bV$ coincides with that of $\cV$ after composition with the natural map $\cO_\cX(\cX)\to\cO_X(X)$. 
	\end{itemize}
\end{rem}

%\andr{weird: any sheaf of G-reps on irreducible is res constant?? so these phigammas that give chenevier's counterex don't give a sheaf of G-reps?}

\begin{lemma}\label{irredequiv}
	%We say that a $G$-representation \bV on X is irreducible if and only if the associated representation \rho_\bV is such that \rho_\bV\otimes_{\cO_X(X)}K_X is irreducible. Indeed, if \bV^\prime is a subsheaf of G-representations of \bV, then it has to be free itself ???
	Let $\bV$ be a $G$-representation on $X$ and $T$ the trace of $\bV$. The following are equivalent:
	\begin{enumerate}
		\item\label{irred1} The $G$-representation $\bV$ is absolutely irreducible.
		\item\label{irred2} The pseudorepresentation $T\otimes_{\cO_X(X)}K_X^\sep$ is irreducible.
		%		\item\label{irred3} the only subsheaves of $G$-representations of $\bV$ are 0 and $\bV$.
	\end{enumerate}
	%CAN THIS BE SAVED?? Assume now that: $X$ is either a wide open rigid analytic space over $L$ or the formal spectrum of a profinite, admissible $\cO_L$-algebra, and writing $\F$ for the residue field of $\cO_X(X)^+$, that there exist absolutely irreducible continuous representations $\ovl\rho_i\colon G\to\Mat_{d_i}(\F)$ such that $\ovl T=\sum_{i=1}^k\tr\ovl\rho_i$. Then \ref{irred1} and \ref{irred2} are also equivalent to:
	%	\begin{enumerate}[resume]
		%		\item\label{irred3} The only subsheaves of $G$-representations of $\bV$ are 0 and $\bV$.
		%	\end{enumerate}
\end{lemma}

%\andr{Use of ${}^\circ$ versus ${}^+$????}
%In the proof of \eqref{irred1}\iff\eqref{irred3} we actually only use the fact that T is residually the sum of traces of absolutely irreducible representations, not that such representations are pairwise non-isomorphic

\begin{proof}
	%We denote by T the trace of \bV. 
	We first prove that \ref{irred1}$\iff$\ref{irred2}. 
	By definition, $\bV$ is absolutely irreducible if and only if $\rho_\bV\otimes_{\cO_X(X)}K_X^\sep$ is irreducible. If $\rho_\bV\otimes_{\cO_X(X)}K_X^\sep$ is an extension of two non-zero $K_X^\sep$-linear representations $\rho_1$ and $\rho_2$ of $G$, then $T\otimes_{\cO_X(X)}K_X^\sep$ is the sum of the traces $T_1,T_2\colon G\to K_X^\sep$ of $\rho_1$ and $\rho_2$. This gives the implication \ref{irred2}$\implies$\ref{irred1}. %By Lemma \ref{????}, the images of T_1 and T_2 land in \cO_X(X), hence the trace of $\bV$, as a pseudorepresentation $G\to\cO_X(X)$, is a sum of two non-trivial pseudorepresentations. This proves by contradiction that \ref{irred2}\implies\ref{irred1}. 
	
	For the other direction, assume that $T\otimes_{\cO_X(X)}K_X^\sep$ is a sum of two non-trivial pseudorepresentations $T_1,T_2\colon G\to K_X^\sep$. By \cite[Theorem 1(2)]{taylorgal}, there exist $K_X^\sep$-linear representations $\rho_1$ and $\rho_2$ with traces $T_1$ and $T_2$, respectively. Since the $K_X^\sep$-linear representations $\rho_\bV\otimes_{\cO_X(X)}K_X^\sep$ and $\rho_1\oplus\rho_2$ share the same trace, their semisimplifications are isomorphic, again by \cite[Theorem 1(2)]{taylorgal}. In particular, $\rho_\bV\otimes_{\cO_X(X)}K_X^\sep$ is not absolutely irreducible. This proves by contradiction that \ref{irred1}$\implies$\ref{irred2}. 
	%
	%We prove that \ref{irred1}$\iff$\ref{irred3} under the extra assumptions given in the statement; in particular, $\cO_X(X)^\circ$ is now a local profinite ring by Remark \ref{wideopenfact}. %that $X$ is either a wide open rigid analytic space over $L$ or the formal spectrum of a profinite, admissible $\cO_L$-algebra. 
	%Assume first that $\bV$ admits a subsheaf of $G$-representations $\bV_1$ distinct from 0 and $\bV$. Let $T_1$ be the trace of $\bV_1$. Then by a simple calculation one finds that $T_2\coloneqq T-T_1$ is also a pseudorepresentation of positive dimension. Since $\ovl T=\ovl T_1+\ovl T_2$, there exist 
\end{proof}

The above lemma motivates the following definition.

\begin{defin}
	%We say that a sheaf of $G$-representations $\bV$ on $X$ is irreducible if the only subsheaves of $G$-representations of $\bV$ are the zero sheaf and $\bV$ itself.
	We say that a sheaf of $G$-representations $\bV$ on $X$, with associated pseudorepresentation $T$, is \emph{absolutely irreducible} if $T\otimes_{\cO_X(X)}K_X^\sep$ is irreducible.
\end{defin}

Every $G$-representation $\cV$ on an admissible formal $\cO_L$-scheme $\cX$ induces a $G$-representation on the generic fiber of $X$: replace the corresponding $\cO_\cX(\cX)$-module $\cM$ with its generic fiber $\cM\otimes_{\cO_\cX(\cX)}\cO_X(X)$, where the tensor product is taken via the natural map $\cO_\cX(\cX)\to\cO_X(X)$. We refer to the resulting representation as the generic fiber of $\cV$, and denote it by $\cV^\rig$. By performing this operation over a covering of $\cX$, one can attach to a sheaf of $G$-representations $\cV$ on $\cX$ a sheaf of $G$-representations on $X$, independent of the chosen covering, that we refer to as the generic fiber of the original sheaf and denote again by $\cV^\rig$.

\subsection{Sheaves of lattices}

Let $\bV$ be a sheaf of $G$-representations on $X$. %The following definition is inspired by \cite[Lemma 3.18]{chenapp}, \cite[Definition 3.5]{hellfam}.

\begin{defin}
	We say that $\bV$ \emph{admits a lattice} if there exists an affine formal model $\cX$ of $X$ and a $G$-representation $\cV$ on $\cX$ such that $\cV^\rig$ is isomorphic to $\bV$ as a $G$-representation on $X$.
	
	%$\Spf\cA$ of $X$ and 
	%, a free $\cA$-module $M$ of finite rank, and a continuous homomorphism
	%\[ \rho_\cA^\circ\colon G\to\Aut_\cA(M) \]
	%with the property that the generic fiber of $M$, with the action of $G$ induced by that on $M$, is isomorphic to $\cF$ as a sheaf of $G$-representations on $X$. explain gen fiber!!!! IN THIS CASE is X affinoid or wide open??
	
	\noindent We say that $\bV$ \emph{admits a sheaf of lattices} if there exists a formal model $\cX$ of $X$ and a sheaf $\cV$ of $G$-representations on $\cX$ such that $\cV^\rig\cong\bV$ as a sheaf of $G$-representations on $X$.
	%a locally free $\cO_\cX$-module $M$ of finite rank, and a continuous homomorphism
	%\[ \rho_\cX^\circ\colon G\to\Aut_{\cO_\cX}(M) \]
	%with the property that the generic fiber of $M$, with the action of $G$ induced by that on $M$, is isomorphic to $\cF$ as a sheaf of $G$-representations on $X$. 
	
	\noindent In both cases we say that the lattice, or sheaf of lattices, $\cV$ is \emph{defined over $\cX$}.
\end{defin}

Our definition of sheaf of lattices is essentially taken from \cite[Lemme 3.18]{chenappuni}. We also refer the reader to \cite[Definition 3.5]{hellfam} for a formulation in terms of adic spaces, and to \cite{tortilatt} where the same notion is referred to as an \emph{integral subfamily} of $\bV$.

Note that if $\bV$ admits a lattice, then $\bV$ is necessarily a $G$-representation on $X$ (not just a sheaf of $G$-representations). 

\begin{rem}\label{pullback}
	Let $f\colon X\to Y$ be a morphism of rigid analytic spaces over $L$. Let $\bV$ be a sheaf of $G$-representations on $Y$, and $\cV$ a sheaf of lattices in $\bV$ defined over a formal model $\cY$ of $Y$. The coherent sheaf $f^\ast\bV$, with the action of $G$ induced from $\bV$, is a sheaf of $G$-representations on $X$. If $Y$ is the generic fiber of $\cY$ in the sense of Raynaud, and $X$ also admits a formal model in the sense of Raynaud, then one can pull back $\cV$ to a lattice for $f^\ast\bV$ defined over some Raynaud formal model for $X$. However, in Berthelot's theory it is unclear whether one can always pull back $\cV$.
\end{rem}

%\andr{Can one attach a sheaf to a pseudorepresentation on $X$?}

%\andr{Recall def of irreducible pseudorep}

%\begin{lemma}\label{freeness}
%If X is a reduced affinoid or a reduced wide open, every sheaf of G-representations with constant reduction???? on X is free. NOT TRUE???
%\end{lemma}
%
%\begin{proof}
%
%\end{proof}

\section{Deformations and pseudodeformations}

%We apply the definitions from Section \ref{funrig} to deformation and pseudodeformation functors on Artin rings, and describe the corresponding functors on formal schemes and rigid analytic spaces. 
We specialize the discussion of Section \ref{funrig} to deformations functors of (pseudo-)representations of a profinite group.

\subsection{Deformation functors on Artin rings}

%We recall a few standard definitions of deformation and pseudodeformation functors on Artin rings. 
%Following Mazur, we call a profinite group $G$ \emph{$p$-finite} if every open subgroup $H$ of $G$ only admits a finite number of open subgroups of index $p$. This condition guarantees that the Artinian deformation problems for representations of $G$ are pro-represented by Noetherian objects.
Until the end of the section, let $G$ be a profinite group that is $p$-finite in the sense of Mazur, and $\ovl\rho$ a continuous representation of $G$ on a $k_L$-vector space $V$ of some finite dimension $n$. 
Given an arbitrary ring $A$ and a representation $\rho$ of $G$ on a free $A$-module, we denote by $\tr\rho\colon G\to A$ the trace of $\rho$.

%\andr{change everywhere to PDef}

%\andr{Need to use determinants instead of pseudorepresentations?} 
Consider the functor $\PDef_{\ovl\rho}\colon\Art_{\cO_L}\to\Sets$, attaching to $A\in\Art_{\cO_L}$ the set of pseudorepresentations $T\colon G\to A$ such that the reduction of $T$ modulo the maximal ideal of $A$ is the trace of $\ovl\rho$. By \cite[Proposition 2.3.1]{boeckledef}, $\PDef_{\ovl\rho}$ is pro-represented by a universal pair $(R_{\ovl\rho}^\ps,T^\univ)$ with $R_{\ovl\rho}^\ps$ Noetherian, that we call the \emph{universal pseudodeformation} of $\ovl\rho$. 

If $\ovl b$ is a basis of $V$, then we define a functor $\Def_{\ovl\rho}^\sq\colon\Art_{\cO_L}\to\Sets$ associating with $A\in\Art_{\cO_L}$ the set of triples $(r,\iota,b)$ where $r$ is a representation of $G_K$ over an $A$-module $V$, $\iota$ is a $G_K$-equivariant isomorphism $V\otimes_Ak_L\cong\ovl V$ and $b$ is a lift of $\ovl b$ to a basis of $V$. By Schlessinger's criterion, $\Def_{\ovl\rho}^\sq$ is pro-represented by a universal quadruple $(R_{\ovl\rho}^\sq,\iota^{\sq,\univ},b^\univ,\rho^{\sq,\univ})$ with $R_{\ovl\rho}^\sq$ Noetherian, that we call the \emph{universal framed deformation} of $\ovl\rho$. 
Since $\tr\rho^{\sq,\univ}$ is a pseudorepresentation lifting $\tr\ovl\rho$, the universal property of $(R_{\ovl\rho}^\ps,T^\univ)$ gives a morphism $R_{\ovl\rho}^\ps\to R_{\ovl\rho}^\sq$ in $\wht\Art_{\cO_L}$. 

Consider the functor $\Def_{\ovl\rho}\colon\Art_{\cO_L}\to\Sets$ associating with $A\in\Art_{\cO_L}$ the set of pairs $(r,\iota)$ where $r$ is a representation of $G_K$ over an $A$-module $V$ and $\iota$ is a $G_K$-equivariant isomorphism $V\otimes_Ak_L\cong\ovl V$. If $\End_{k_L[G]}(\ovl\rho)=k_L$, then by \cite[Proposition 1]{mazdef}, $\Def_{\ovl\rho}$ is pro-represented by a universal triple $(R_{\ovl\rho},\iota^{\univ},\rho^{\univ})$ with $R_{\ovl\rho}$ Noetherian, that we call the \emph{universal deformation} of $\ovl\rho$. In this case, the relevant universal properties give morphisms $R_{\ovl\rho}^\ps\to R_{\ovl\rho}\to R_{\ovl\rho}^\sq$ in $\wht\Art_{\cO_L}$ that commute with the already given $R_{\ovl\rho}^\ps\to R_{\ovl\rho}^\sq$. When $\ovl\rho$ is irreducible, the morphism $R_{\ovl\rho}^\ps\to R_{\ovl\rho}$ is an isomorphism \cite[Theorem 2.4.1]{boeckledef}.

%\subsection{Deformation functors on rigid analytic spaces}
%
%We ``rigidify'' the deformation functors defined in the previous section, following the construction of Section \ref{funrig}. 
%%Recall that, in order for a rigidification to exists, we need to check the condition:
%%\begin{align}\label{blowup}\tag{$\ast$} 
%\text{for every admissible formal blow-up } f\colon\cX\to\cY,\text{ the map }F^\form(f)\colon F^\form(\cY)\to F^\form(\cX)\text{ is a bijection.}
%\end{align}

%\andr{Check Chenevier lect5 on berthelot's generic fiber of univ def ring??}

\subsection{Pseudorepresentations lifting residual representations}

As before, let $L$ be a $p$-adic field with ring of integers $\cO_L$ and residue field $k_L$, and let $G$ be a profinite, $p$-finite group. 
Fix a continuous pseudorepresentation $\ovl T\colon G\to k_L$. 

Let $\cA\in\Ad_{\cO_L}$. The structure map $\cO_L\to\cA$ induces a map $k_L\to\cA/\pi_L$. Let $T:G\to\cA$ a continuous pseudorepresentation.

\begin{defin}
We say that $T$ lifts $\ovl T$ if $T\otimes_\cA\cA/\pi_L:G\to\cA/\pi_L$ coincides with $\ovl T\otimes_{k_L}\cA/\pi_L$.
\end{defin}

%Note that $T$ lifts $\ovl T$ if and only if $T\otimes_\cA\cA\to\cA/\pi_L$ coincides with $\ovl T$ after specialization at every maximal ideal of $\cA/\pi_L$.

We recall that a rigid space $X$ is said to be \emph{quasi-paracompact} if there exists an admissible covering $\bigcup_i U_i$ of $X$ by quasi-compact open subspaces, such that every quasi-compact open subspace of $X$ only intersects finitely many $U_i$.

Let $X$ be a quasi-separated, quasi-paracompact, reduced rigid analytic space over $L$ and $T\colon G\to\cO_X(X)$ a continuous pseudorepresentation. %For every formal model $\cX$ of $X$, there is a natural homomorphism $\cO_\cX(\cX)\to\cO_X(X)$: when $\cA$ is an affinoid $\cO_L$-algebra and $\cX=\Spf\cA$ is affine, this is simply the inclusion $\cA\into\cA\otimes_{\Z_p}\Q_p$.

\begin{defin}\label{def:lift}
	We say that $T$ \emph{lifts} $\ovl T$ if there exists a formal model $\cX$ of $X$ such that $T$ factors through a pseudorepresentation $T_\cX\colon G\to\cO_\cX(\cX)$, and an affine covering $\{\Spf\cA_i\}_{i\in I}$ of $X$ such that, for every $i\in I$, the restriction of $T_\cX$ to $G\to\cA_i$ lifts $\ovl T$.
	%	If $X$ is a rigid analytic space over $L$ and $T\colon G\to\cO_X(X)$ is a pseudorepresentation, we say that $T$ \emph{lifts} $\ovl \rho$ if there exist a formal model $\cX$ of $X$ and a continuous representation $\rho:G\to\GL(\cO_\cX(\cX))$ such that the reduction of $\rho$ modulo $\pi_L$ is isomorphic to $\rho\otimes_{k_L}\cO_\cX(\cX)/\pi_L$.
\end{defin}

In the second part of the definition, we require the much stricter condition that a lattice with trace $T$ exists. This is necessary in order to compare certain deformation spaces in \Cref{sec:psdef}.

We show that one can check the first property above on points. Recall that we can evaluate an element of $\cO_\cX(\cX)$ at a point of $\cX^\rig$, via the map $\cO_\cX(\cX)\to\cO_X^+(X)$ of \eqref{dejong}. %hence that the choice of covering $\{\Spf\cA_i\}_i$ in \Cref{def:lift} is irrelevant.

In the following, let $X$ be a rigid analytic space, and $T:G\to\cO_X(X)$ a continuous pseudorepresentation. 
\begin{lemma}\label{lem:factor}
$T$ factors through a continuous pseudorepresentation $T^+\colon G\to\cO_X^+(X)$.
\end{lemma}
\begin{proof}
For every $L$-point $x$ of $X$, the specialization of $T$ at $x$ is a continuous pseudorepresentation $T_x\colon G\to L$. Since $G$ is compact, the image of $T_x$ lands in $\cO_L$. Then for every $g\in G$ and $x\in X(L)$, the evaluation of $T(g)$ at $x$ is power-bounded, so that $T(g)$ itself is power-bounded.
\end{proof}

\begin{lemma}\label{Tform}\mbox{ }
%\begin{enumerate}[label=(\roman*)]
%\item $T$ factors through a continuous pseudorepresentation $T^+\colon G\to\cO_X^+(X)$.
%\item 
If $X$ is flat, normal, and admits a flat formal model, then $T$ factors through $\cO_\cX(\cX)\to\cO_X(X)$ for a flat, normal formal model $\cX$. %of $X$.
%\andr{compare with chenevier's results on pseudoreps over affinoids??}
%\end{enumerate}
\end{lemma}

\begin{proof}
%For part (i), observe that for every $L$-point $x$ of $X$, the specialization of $T$ at $x$ is a continuous pseudorepresentation $T_x\colon G\to L$. Since $G$ is compact, the image of $T_x$ lands in $\cO_L$. Then for every $g\in G$ and $x\in X(L)$, the evaluation of $T(g)$ at $x$ is power-bounded, so that $T(g)$ itself is power-bounded.
%We prove (ii). 
Let $\cX_0$ be a flat formal model of $X$. Let $\cX$ be the normalization of $\cX_0$, constructed as in the first few lines of \cite[Section 2.1]{conradirr}. Then by \cite[Theorem 2.1.3]{conradirr} $\cX^\rig\to X$ is a normalization of $X$, hence an isomorphism since $X$ is already normal. 

Since $\cX$ is flat and normal, the morphism \eqref{dejong} induces an isomorphism $\cO_\cX(\cX)\to\cO_X^+(X)$ by \cite[Theorem 7.4.1]{DJ95}. Therefore, $T$ factors through a pseudorepresentation $T^+\colon G\to\cO_\cX(\cX)$ by Lemma \ref{lem:factor}.
	%\andr{It is enough to prove the statement locally? On every affinoid, the image of $T$ must lie in the power-bounded elements? But is it contained in the regular functions on a formal model?}
	%: indeed, if $f\in\cO_X(X)$ belongs to??? Let \{\Spf A_i\}
\end{proof}

\begin{prop}\label{liftallps}
	Assume that $X$ is flat, normal, and admits a flat formal model. 
	A pseudorepresentation $T\colon G\to\cO_X(X)$ lifts $\ovl T$ if and only if, for every $x\in X(\Qp)$, the specialization $T_x\colon G\to\Qp$ of $T$ at $x$ lifts $\ovl T$. %\andr{need to allow for extensions of $L$?}
\end{prop}

\begin{proof}
We first prove that, for a reduced $\cA\in\Ad_{\cO_L}$, a pseudorepresentation $T:G\to\cA$ lifts $\ovl T$ if and only if, for every point $x$ of $X=(\Spf\cA)^\rig$, the specialization of $T$ at $x$ lifts $\ovl T$. %If $x$ is such a point, then the map $\cA\to\cO_X^+(X)\xto{x}\Zp\onto\Fp$ factors through a map $\ovl x:\cA/\pi_L\to\Fp$. 
We use the notation of Remark \ref{rem:DJ719}. By diagram \eqref{ovlx}, the mod $\pi_E$-reduction of $T_x$ coincides with $\ovl x\circ T$, hence $T_x$ lifts $\ovl T$ for every $x$ if and only if $\ovl x\ccirc T=\ovl T\otimes_{k_L}k_E$ for every $x$.

If $T\otimes_\cA\cA/\pi_L=\ovl T\otimes_{k_L}\cA/\pi_L$, then $\ovl x\ccirc T=\ovl T\otimes_{k_L}k_E$ for every $x$. Conversely, if $\ovl x\ccirc T=\ovl T\otimes_{k_L}k_E$ for every $x$, then by the bijections from Remark \ref{rem:DJ719} the reduction of $T\otimes_\cA\cA/\pi_L$ modulo an arbitrary maximal ideal $\fm\subset\cA/\pi_L$ is $\ovl T\otimes \cA/(\pi_L,\fm)$. Since $\cA$ is reduced, the intersection of all of these maximal ideals is 0, hence $T\otimes_\cA\cA/\pi_L=\ovl T\otimes_{k_L}\cA/\pi_L$, as desired.

%Now let $X$ be a (quasi-separated, quasi-paracompact) reduced rigid analytic space. 
Now let $X$ be as in the statement, and let $\cX$ be a formal model of $X$ such that $T$ factors via $T_\cX:G\to\cO_\cX(\cX)$, as provided by Lemma \ref{Tform}. Let $\{\Spf\cA_i\}_{i\in I}$ a covering of $X$ by reduced affine formal schemes. If $T_x$ lifts $\ovl T$ for every point $x$ of $X$, then for every $i$ the restriction $T_i$ of $T_\cX$ to $\cA_i$ lifts $\ovl T$ by the result of the previous paragraph, hence $T$ lifts $\ovl T$.

On the other hand, if $T$ lifts $\ovl T$, then by definition there exist $\cX$ and $\{\Spf\cA_i\}_{i\in I}$ as above such that $T_i$ lifts $\ovl T$ for every $i$, and again by the result of the previous paragraph we deduce that $T_x$ lifts $\ovl T$ for every $x$.
\end{proof}

\begin{rem}\label{qsqp}\mbox{ }
\begin{itemize}
\item If $X$ is flat, quasi-separated and quasi-paracompact, then it admits a flat formal model $\cX_0\to\Spf\cO_L$ by the results of Raynaud and Gruson (see \cite[Corollary 5.10]{BLII}).
\item If $X\cong(\Spf\cA)^\rig$ for some $\cA\in\Ad_{\cO_L}$, then $\Spf\cO_X^+(X)$ is a flat formal model of $X$. %\andr{doublecheck}
\end{itemize}
\end{rem}

%Since the structure map $X\to\Spm K$ is flat, it admits a flat formal model $\cX_0\to\Spf\cO_L$ by the results of Raynaud and Gruson (see \cite[Corollary 5.10]{BLII}).

%WHAT did I want to do here?
%\begin{lemma}
%USE CHENEVIER? AFTER LIFTING T??
%\end{lemma}

\subsubsection{The rigid pseudodeformation functor}\label{sec:psdef}

Let $F=\PDef_{\ovl\rho}\colon\Art_{\cO_L}\to\Sets$, and let $(R^\ps_{\ovl\rho},T^\univ)$ be the universal pair for $F$. Consider the functors
\begin{align*} \PDef^\prof_{\ovl\rho}\colon\Frm^\prof_{\cO_L}\to\Sets, \quad
\cX\to\{\text{pseudodeformations }T\colon G\to\cO_\cX(\cX)\text{ of }\tr\ovl\rho\}, \\
\PDef^\wo_{\ovl\rho}\colon\Rig^\wo_L\to\Sets, \quad 
	X\to\{\text{pseudodeformations }T\colon G\to\cO_X(X)\text{ of }\tr\ovl\rho\}.
\end{align*}
%We denote by \PDef^\faff_{\ovl\rho} the restriction of 

%RECALL the universal formal and rigid pseudoreps--adapt notation

\begin{notation}
	We set $X_{\ovl\rho}^\ps=(\Spf R_{\ovl\rho}^\ps)^\rig$, and we still denote with $T^{\univ}$ the composition of the universal pseudodeformation $T^\univ$ of $\ovl\rho$ with $R_{\ovl\rho}^\ps\to\cO_{X_{\ovl\rho}^\ps}(X_{\ovl\rho}^\ps)$. 
\end{notation}

Note that if the category of strictly quasi-Stein rigid analytic $L$-spaces is replaced with that of $L$-affinoids, the following result is \cite[Theorem 2.2(iii)]{chenlec5}.

\begin{prop}\label{psdefequiv}
There are natural isomorphisms
\begin{align*} F^\form\vert_{\Frm^\prof_{\cO_L}}\cong \PDef^\prof_{\ovl\rho} \quad
	F^\rig\vert_{\Rig^\wo_L}\cong\PDef^\wo_{\ovl\rho}, \end{align*}
i.e. the functor on formal schemes (respectively, rigid spaces) obtained from $F$ is a pseudodeformation functor after restriction to the subcategory of formal spectra of Noetherian profinite $\cO_L$-algebras (respectively, strictly quasi-Stein rigid analytic $L$-spaces).
\end{prop}

\begin{proof}
By definition, an object $\cX\in\Frm^\prof$ is a formal scheme of the form $\Spf\cA$, with $\cA$ a Noetherian object of $\wht\Art_{\cO_L}$; in particular, the category $\Frm^\prof$ is a full subcategory of $(\wht\Art_{\cO_L})^\op$. Therefore, the first isomorphism follows from the fact that $\Hom(R_{\ovl\rho}^\ps,-)\cong F=\PDef_{\ovl\rho}$.

%By Remark \ref{}, if X\in\Rig^\wo_L then \Spf\cO_X(X)^+\in\wht\Art_{k_L}. 
By definition $F^\rig=\Hom(-,(\Spf R^\ps_{\ovl\rho})^\rig)$, so that to every element of $F^\rig(X)$ we can attach a homomorphism $f\colon X\to (\Spf R^\ps_{\ovl\rho})^\rig$, and the continuous pseudorepresentation $f^\ast T^\univ\colon G\to\cO_X(X)$ lifting $\tr \ovl\rho$. This provides us with a natural transformation $F^\rig\to\PDef^\wo_{\ovl\rho}$. We construct a natural transformation in the opposite direction, omitting the routine check that it is an inverse to the one above. %We show that every continuous pseudorepresentation T\colon G\to\cO_X(X) is obtained this way. 
By Lemma \ref{Tform} and Remark \ref{qsqp}, every continuous representation $T\colon G\to\cO_X(X)$ lifting $\tr \ovl\rho$ factors through a continuous pseudorepresentation $T^+\colon G\to\cO_X^+(X)$ lifting $\tr \ovl\rho$. Since $\cO_X^+(X)$ is an element of $\wht\Art_{\cO_L}$ by Remark \ref{wideopenfact}, $T$ defines an element of $\Hom(R_{\ovl\rho}^\ps,\cO_X^+(X))$, hence an element of $\Hom(X,(\Spf R^\ps_{\ovl\rho})^\rig)=F^\rig(X)$, as desired.
\end{proof}

\subsection{Sheaves lifting residual representations}
As before, let $V$ be an $L$-vector space with a continuous $G$-action, $\ovl V=V\otimes_{\cO_L}k_L$, and $\ovl\rho\colon G\to\GL(\ovl V)$ an $n$-dimensional, continuous $k_L$-linear representation.

\begin{defin}
	We say that $V$ \emph{lifts} $\ovl\rho$ if there exists a $G$-stable $\cO_L$-lattice $\cV\subset V$ such that $\cV\otimes_{\cO_L}k_L$ is isomorphic to $\ovl\rho$ as a $k_L$-linear representation of $G$. 
\end{defin}

Contrary to a common convention, we are not taking any semisimplification after tensoring with $k_L$: we want a lattice that actually lifts the chosen, and possibly non-semisimple, $\ovl\rho$. Since a $k_L$-representations of $G$ lifted by $V$ is only determined up to semisimplification, a same $V$ can lift various non-isomorphic representations.

Let $\cX$ be an $\cO_L$-formal scheme and let $\cV$ be a sheaf of $G$-representations on $\cX$.

\begin{defin}\label{formlift}
	We say that $\cV$ \emph{lifts} $\ovl\rho$ if, for every point $x$ of $\cX$, corresponding to a morphism $\Spf k\to\cX$ for a finite extension $k$ of $k_L$, the fiber $\cV_x$ is isomorphic to $\ovl\rho\otimes_{k_L}k$ as a $G$-representation.
%	
%	If $\cX=\Spf\cA$ is affine and $\cV$ is a $G$-representation on $\cA$, we say that $\cV$ \emph{lifts} $\ovl\rho$ if, for every  $\cV(\Spf\cA)\otimes_{\cO_L}k_L$ is isomorphic to $\ovl\rho$ as a $k_L$-representation of $G$. \andr{for every open ideal???}
%	
%	For general $\cX$ and $\cV$, we say that $\cV$ \emph{lifts} $\ovl\rho$ if there exists an affine covering $\{\Spf\cA_i\}_{i\in I}$ of $\cX$ such that, for every $i\in I$, $\cV\vert_{\Spf\cA_i}$ is a $G$-representation that lifts $\ovl\rho$.
\end{defin}

%By Lemma \ref{chenlatt}, every sheaf lifts some ovlr??

Now let $X$ be a rigid analytic space over $L$ and $\bV$ a sheaf of $G$-representations on $X$. 

\begin{defin}
	We say that $\bV$ \emph{lifts} $\ovl\rho$ if there exist a formal model $\cX$ of $X$ and a sheaf of lattices $\cV$ for $\bV$ over $\cX$ such that $\cV$ lifts $\ovl\rho$.
\end{defin}

\subsubsection{The rigid deformation functor}

Assume that $\ovl\rho$ is absolutely irreducible, of dimension $d<p$. 
Consider the functor $F:=\Def_{\ovl\rho}\colon\Art_{\cO_L}\to\Sets$, and let $(R_{\ovl\rho},\rho^{\univ})$ be the universal pair for $F$. Define functors
\begin{align*} \Def^\prof_{\ovl\rho}\colon\Frm^\prof_{\cO_L}&\to\Sets, \quad
	\cX\to\{\text{$G$-representations }\cV\text{ over }\cX\text{ lifting }\ovl\rho\}, \\
	 \Def^\wo_{\ovl\rho}\colon\Rig^\wo_L&\to\Sets, \quad
	X\to\{\text{$G$-representations }\bV\text{ over }X\text{ lifting }\ovl\rho\}.
\end{align*}
%We denote by \PDef^\faff_{\ovl\rho} the restriction of 

%RECALL the universal formal and rigid pseudoreps--adapt notation

Again, if the category of wide open rigid analytic $L$-spaces is replaced with that of $L$-affinoids, the following result is the absolutely irreducible case of \cite[Theorem 2.2(iii,iv)]{chenlec5}.

\begin{prop}\label{defequiv}
	There are natural isomorphisms
	\begin{align*} F^\form\vert_{\Frm^\prof_{\cO_L}}\cong \Def^\prof_{\ovl\rho}, \quad
		F^\rig\vert_{\Rig^\wo_L}\cong\Def^\wo_{\ovl\rho}, \end{align*}
	i.e. the functor on formal schemes (respectively, rigid spaces) obtained from $F$ is a deformation functor after restriction to the subcategory of formal spectra of Noetherian profinite $\cO_L$-algebras (respectively, wide open rigid analytic $L$-spaces).
\end{prop}

\begin{proof}
%	By definition, an object $\cX\in\Frm^\prof$ is a formal scheme of the form $\Spf\cA$, with $\cA$ a Noetherian object of $\wht\Art_{k_L}$; in particular, the category $\Frm^\prof$ is a full subcategory of $(\wht\Art_{k_L})^\op$. 
	Since deformations of $\ovl\rho$ to a Noetherian object $\cA$ of $\wht\Art_{\cO_L}$ are in an obvious natural bijection with deformations of $\ovl\rho$ over $\Spf\cA$, the first statement follows from the universal property of $R_{\ovl\rho}$.
	
	%By Remark \ref{}, if X\in\Rig^\wo_L then \Spf\cO_X(X)^+\in\wht\Art_{\cO_L}. 
	By definition $F^\rig=\Hom(-,(\Spf R_{\ovl\rho})^\rig)$, so that to every element of $F^\rig(X)$ we can attach a homomorphism $f\colon X\to (\Spf R_{\ovl\rho})^\rig$, and the $G$-representation $f^\ast\rho^\univ$ lifting $\ovl\rho$. This provides us with a natural transformation $F^\rig\to\Def^\wo_{\ovl\rho}$. We construct a natural transformation in the opposite direction, omitting the routine check that it is an inverse to the one above. 
	%We show that every continuous pseudorepresentation T\colon G\to\cO_X(X) is obtained this way. 
	
	Let $\bV$ be a $G$-representation over $X$, lifting $\ovl\rho$. Since $\ovl\rho$ is absolutely irreducible, $\bV$ is also absolutely irreducible. In particular, $\bV$ admits a lattice $\cV$ defined over a formal model $\cX$ of $X$ by Proposition \ref{prop:factorial} (which is independent of the results of this section). Now $\cV$ is a deformation of $\ovl\rho$, hence by the universal property of $\Spf R_{\ovl\rho}$, it is induced by a map $g\colon\cX\to\Spf R_{\ovl\rho}$. The generic fiber of $g$ defines an element of $F^\rig(X)$, as desired.
	%, and by Remark \ref{maxmodel} we can choose \cV to be defined over the formal model \Spf R_{\ovl\rho} of X. 
	%	By Proposition \ref{chenlatt}, \bV admits a sheaf of lattices, defined over a formal model \cX of X. 
	%	By Proposition \ref{Tform} and Remark \ref{qsqp}, every continuous representation $T\colon G\to\cO_X(X)$ lifting $\ovl\rho$ factors through a continuous pseudorepresentation $T^+\colon G\to\cO_X^+(X)$ lifting $\ovl\rho$. Since $\cO_X^+(X)$ is an element of $\wht\Art_{\cO_L}$ by Remark \ref{wideopenfact}, $T$ defines an element of $\Hom(R_{\ovl\rho}^\ps,\cO_X^+(X))$, hence an element of $\Hom(X,(\Spf R^\ps_{\ovl\rho})^\rig)=F^\rig(X)$, as desired.
\end{proof}

\begin{notation}
	We set $X_{\ovl\rho}=(\Spf R_{\ovl\rho})^\rig$ and $X_{\ovl\rho}=(\Spf R_{\ovl\rho})^\rig$. We still denote by $\rho^{\univ}$ (respectively, $\rho^{\sq,\univ}$) the composition of the universal deformation $\rho^\univ$ (respectively, the universal framed deformation $\rho^{\sq,\univ}$) of $\ovl\rho$ with $R_{\ovl\rho}\to\cO_{X_{\ovl\rho}}(X_{\ovl\rho})$ (respectively, $R_{\ovl\rho}^\sq\to\cO_{X_{\ovl\rho}^\sq}(X_{\ovl\rho}^\sq)$)
%	We set $X_{\ovl\rho}=(\Spf R_{\ovl\rho})^\rig$, and we still denote by $\rho^{\univ}$ the composition of the universal deformation $\rho^{\sq,\univ}$ of $\ovl\rho$ with $R_{\ovl\rho}\to\cO_{X_{\ovl\rho}}(X_{\ovl\rho})$. 
\end{notation}

Note that the sheaf $\bV_{\ovl\rho}^{\univ}$ is actually a $G$-representation on $X_{\ovl\rho}$, not just a sheaf, since it is attached to the representation $\rho^\univ$. If $X$ is a rigid analytic space $X$ over $L$ carrying a sheaf $\bV$ of $G$-representations lifting $\ovl\rho$ (and also carrying some associated $\iota$), then the morphism $f\colon X\to X_{\ovl\rho}$ provided to us by the universal property identifies $\bV$ with the pullback $f^\ast\bV_{\ovl\rho}^\univ$, which is a $G$-representation on $X$. We deduce the following.

\begin{lemma}
	Assume that $\End_{k_L}(\ovl\rho)=k_L$. Let $X$ be a rigid analytic space and $\bV$ a sheaf of $G$-representations on $X$ lifting $\ovl\rho$. Then $\bV$ is a $G$-representation on $X$.
\end{lemma}

Recall that, as functors on $\Art_{\cO_L}$, we have natural transformations $\Def_{\ovl\rho}^\sq\to\Def_{\ovl\rho}\to\PDef_{\ovl\rho}$, inducing a morphism $R_{\ovl\rho}^\ps\to R_{\ovl\rho}^\sq$, and, if $\End_{k_L}(\ovl\rho)=k_L$, $R_{\ovl\rho}^\ps\to R_{\ovl\rho}\to R_{\ovl\rho}^\sq$. The corresponding natural transformations of functors in the rigid setting are given as follows:
\begin{itemize}
	\item $\Def_{\ovl\rho}^\sq\to\Def_{\ovl\rho}$ maps a triple $(\bV,\iota,b)$ to the pair $(\bV,\iota)$,
	\item $\Def_{\ovl\rho}\to\PDef_{\ovl\rho}$ maps a pair $(\bV,\iota)$ to the trace of $\bV$.
\end{itemize}
By passing to the representing objects we obtain a morphism $X^\sq_{\ovl\rho}\to X_{\ovl\rho}^\ps$, and, if $\End_{k_L}(\ovl\rho)=k_L$, morphisms $X^\sq_{\ovl\rho}\to X_{\ovl\rho}\to X^\ps_{\ovl\rho}$. These morphisms are also induced by the ring homomorphisms recalled above. %Note that these spaces carry progressively weaker structures: $(\Spf R_{\ovl\rho}^\sq)^\rig$ carries a $G$-representation, $(\Spf R_{\ovl\rho})^\rig$ carries a sheaf of $G$-representations BUT SEEMS HAVE ACTUAL GREP???? IF IRREDUCIBLE, and $(\Spf R_{\ovl\rho}^\ps)^\rig$ only carries a pseudorepresentation.

%OK STILL GET THAT global free things have lattice, and if irreducible also other things

\begin{rem}\label{exlatt}
	The $G$-representation $\bV_{\ovl\rho}^{\sq,\univ}$ on $X^\sq_{\ovl\rho}$ admits a lattice: it is the $G$-representation on $\Spf R_{\ovl\rho}^\sq$ associated with $\rho^{\sq,\univ}\colon G\to\GL_n(R_{\ovl\rho}^\sq)$.
	
	When $\End_{k_L[G]}(\ovl\rho)=k_L$, the $G$-representation $\bV_{\ovl\rho}^\univ$ on $X_{\ovl\rho}$ admits a lattice: it is the $G$-representation on $\Spf R_{\ovl\rho}$ associated with $\rho^{\univ}$.
\end{rem}

\medskip

\section{Lattices in families of representations}\label{seclatt}

In the context of Raynaud's theory of formal models, Chenevier proves that every sheaf of $G$-representations on a reduced, quasi-compact, quasi-separated rigid analytic space $X$ admits a sheaf of lattices \cite[Lemme 3.18]{chenappuni}. We state a version of his result in the context of Berthelot's theory. The proof simply requires checking that Chenevier's argument goes through for a rigid analytic space that admits a formal model in the sense of Berthelot. Let $X$ be a rigid analytic space over $L$, and $\bV$ a sheaf of $G$-representations over $X$.

\begin{prop}\label{chenlatt}
	Assume that there exists a torsion-free formal model $\cX_0$ of $X$, and a torsion-free, coherent $\cO_{\cX_0}$-module $\cV_0$ such that $\cV_0\otimes_{\cO_L}L$ is isomorphic to $\bV$ as an $\cO_X$-module (without taking into account the action of $G$). Then there exists a formal model $\cX$ of $X$ and a sheaf of lattices for $\bV$ defined over $\cX$. %an integral model of the family of $G$-representations $\bV$ defined over $\cX$. 
%	There exists a pair $(\cX,\bT)$ consisting of a formal model $\mathcal{X}$ of $\mathfrak{X}$ and an integral model $\bT$ of $\bV$ defined over $\cX$. %subfamily $\mathbb{T}$ of representations of $G$ such that $\mathbb{T}[\frac{1}{p}]\cong \mathbb{V}$ as $\mathcal{O}_\mathfrak{X} [G]$-modules. 
\end{prop}

\begin{proof}
%	By Raynaud's theory \cite[Section 8.4, Lemma 4(e)]{Bo14}, there exist a formal model $\cX_0$ of $X$, of finite type over $\Spec\cO_L$ and torsion-free, and a coherent, torsion-free $\cO_{\cX_0}$-module $\cV_0$ such that $\cV_0\otimes_{\cO_L}L\cong\bV$. 
	Let $\cX_0, \cV_0$ be as in the statement. Let $\{\Spf\cA_i\}_{i\in I}$ be a covering of $\cX_0$ by open affine formal subschemes such that $\cV_0\vert_{\Spf\cA_i}$ is free, i.e. it is attached to a finite free $\cA_i$-module $V_{0,i}$. Clearly $\bV\vert_{\Spm(\cA_i[1/p])}$ is attached to the finite free $\cA_i[1/p]$-module $V_i\coloneqq V_{0,i}\otimes_{\cO_L}L$. %Since V_{0,i} is open in V_i for every i\in I, the product \prod_{i\in I}V_{0,i} is open in \prod_{i\in I}V_i. 
	The action of $G$ on $\bV$ induces a continuous action of $G$ on $\prod_{i\in I}V_i$, and the stabilizer $H$ of the open subring $\prod_{i\in I}V_{0,i}\subset \prod_{i\in I}V_i$ is open in $G$. Since $G$ is profinite, $H$ is of finite index in $G$. Let $\{g_i\}_i$ be a finite set of representatives for the left $H$-cosets in $G$. The finite sum $\cV_1=\sum_{i}g_i(\cV_0)$ of subsheaves of $\bV$ is a $G$-stable subsheaf of $\bV$. It has a natural structure of coherent, torsion-free $\cO_{\cX_0}$-module, and $\cV_1\otimes_{\cO_L}L\cong\bV$. 
	
	In general, it is not true that $\cV_1$ is locally free as an $\cO_{\cX_0}$-module; we rely on Raynaud's theory of admissible blow-ups in order to replace it with a locally free sheaf over a suitable formal model of $X$. 
	Let $\mathcal I$ be the Fitting ideal of the $\cO_{\cX_0}$-module $\cV_1$; it is an admissible ideal sheaf in $\cO_{\cX_0}$. We consider the blow-up $\cX$ of $\cX_0$ relative to $\mathcal I$, as in \cite[Chapter II, Sections 1.1(a-b)]{FK1}: it is another admissible formal model of $\cX$. By \cite[Chapter II, Section 1.2]{FK1} the strict transform $\cV$ of $\cV_1$ along $\cX\to\cX_0$ is a coherent, locally free $\cO_\cX$-module of finite type over $\cX$, and its generic fiber is still isomorphic to $\bV$. The $\cO_{\cX_0}$-linear action of $G$ on $\cV_1$ induces an $\cO_\cX$-linear action of $G$ on $\cV$, giving back the original action of $G$ on the generic fiber $\bV$. 
\end{proof}

%\begin{lemma}\label{chenlatt}
%	Every sheaf of $G$-representations on a reduced, quasi-compact, quasi-separated rigid analytic space $X$ admits a sheaf of lattices. 
%\end{lemma}

%\andr{Give example of case when a $G$-representation does not admit a lattice? E.g. a not residually constant one (impossible)? does one such example exist?? Cite Hellmann? --need something like phigamma attached to gal rep iff resid constant iff it is global Grep iff has lattice???}

\begin{rem}\label{wocoh}
The assumptions of Proposition \ref{chenlatt} are satisfied if $X$ is either of the following:
\begin{itemize}
\item A quasi-separated, quasi-compact rigid analytic space: $\cX_0$ and $\cV_0$ are provided to us by Raynaud's theory \cite[Section 8.4, Lemma 4(e)]{boschbook}.
\item A wide open rigid analytic space: we can choose $\cX_0=\Spf\cO_X^+(X)$ by the discussion after Definition \ref{defwo}. To construct $\cV_0$, consider the increasing affine covering $\{\cX_i\}_{i\ge 1}$ of $\cX_0$ described in \cite[Section 7.1.1]{DJ95}: the generic fibers of the $\cX_i$ give an increasing, admissible affinoid covering $\{X_i\}_{i\ge 1}$ of $X$. 
For every $i$, Raynaud's theory produces a torsion-free, coherent $\cO_{\cX_i}$-module $\cV_i$ such that $\cV_{i}\otimes_{\cO_L}L\cong\bV\vert_{X_i}$. Fix an $i_0$. Via the argument in \cite[Lemma 2.2]{lutkeformal}, we can modify $\cV_{i+1}$ so into another lattice $\cV_{i+1}'$ inside $\bV_{i+1}$, such that $\cV_{i+1}\otimes_{\cO_X(X_{i+1})}\cO_X(X_i)\cong\cV_i$ (note that \emph{loc. cit.} cannot be applied to our wide open space $X$, but we only need to apply it to $X_{i+1}$). Then by setting $\cV(A)=\cV_i(A)$ for every $i$ and every affine formal subscheme $a\subset\cX_i$ we obtain a lattice inside $\bV$.
%Then the inverse limit $\cV_0\coloneqq\lim_i\cV_i$ is a torsion-free, coherent $\cO_{\cX_0}$-module with generic fiber $\bV_$.
\end{itemize}
\end{rem}

	\subsection{Existence of lattices over wide open rigid analytic spaces}

	Let $X$ be an integral, strictly quasi-Stein (in particular, wide open) rigid analytic space over $L$. %, and let \bV be a G-representation on X.
	Let $A^+=\cO_X^+(X)$. By \Cref{wideopenfact} and \Cref{converse}, $A^+$ is a local profinite $\cO_L$-algebra and $X$ is the generic fiber of $\Spf A^+$. in particular, $A^+$ is a complete Hausdorff local ring, hence Henselian. 
	The total fraction ring $K_X$ of $A$ is a field since $A$ is integral. 
	The following is a simple consequence of \cite[Proposition 1.6.1]{bellchen}.
	
	\begin{prop}\label{prop:factorial}
	Let $\bV$ be an absolutely irreducible, residually multiplicity-free $G$-representation $\bV$ on $X$ of rank $d<p$. If $\bV$ is not residually absolutely irreducible, assume that $\cO_X^+(X)$ is a UFD.
	Then $\bV$ admits a lattice defined over $\Spf\cO_X^+(X)$.
	\end{prop}
	
	\begin{proof}
	Let $n$ be the rank of $\bV$, and let $\rho_\bV:G\to\GL_n(K_X)$ be the representation attached to $\bV$.  In order to prove the theorem, we need to exhibit a representation $G\to\GL_n(A^+)$ inducing $\bV$ upon extension of scalars to $A$.
	Since $X$ is wide open, the assumptions of Proposition \ref{chenlatt} are satisfied by Remark \ref{wocoh}. Therefore, there exists a sheaf of lattices $\cV$ for $\bV$, defined over some formal model $\cX$ of $X$. By Remark \ref{maxmodel}, we can assume that $\cX=\Spf A^+$ since $X$ is wide open. Let $T\colon G\to A^+$ be the pseudorepresentation attached to $\cV$, in the sense of Definition \ref{deftrace}. Under our assumption that $d<p$, we apply \cite{nyssen}, \cite[Corollaire 5.2]{rouquier} in case (i), and \cite[Proposition 1.6.1]{bellchen} in case (ii), to $R=A^+[G]$ and the local Henselian ring $A^+$, and obtain an $A^+$-algebra homomorphism $R\to\Mat_n(A^+)$ with trace $T$. Restricting it to $G$ gives a representation $\rho:G\to\GL_n(A^+)$ with trace $T$. Then $\rho\otimes_{A^+}K_X$ and $\rho_\bV$ share the same trace, and since they are absolutely irreducible, they are isomorphic (by a theorem of Taylor \cite{taylorgal}, there is a unique irreducible representation of trace $T$ over an algebraic closure $\ovl{K_X}$, and a standard descent argument shows that the uniqueness holds over $K_X$).
	\end{proof}
	
	\begin{rem}\label{rem:regular}
The assumption that $A^+$ is a UFD holds for instance if $X$ is a wide open disc (see Example \ref{exdisc}), or more generally if $X$ is a regular (strictly quasi-Stein) rigid space. Indeed, if $X$ is regular, then its model $\Spf A^+$ is \emph{rig-regular} in the sense of Temkin (see \cite[Remark 3.1.3(iv)]{temkin}), meaning that it is regular outside of its closed fiber. But since $X$ is strictly quasi-Stein, $A^+$ is local, equipped with the topology of its maximal ideal, so that the closed fiber is a single point.

In our arithmetic applications of Section of \ref{sec:arith}, we can typically assume that the neighborhood we work over is a wide open disc. Note that more deformation spaces of arithmetic interest are regular at sufficiently general points; for instance, the trianguline deformation spaces studied in \cite{brehelschloc}. %We refer to Example \ref{} for a discussion of the case when $X$ is a wide open disc, in which case $A^+$ is a UFD.
\end{rem}

\begin{rem}\label{quasioptimal}
The proof of Proposition \ref{prop:factorial} goes through if one only starts with a sheaf of $G$-representations, rather than a $G$-representation, on $X$, but produces as an output a lattice over $X$, so that a fortiori one has a $G$-representation. In particular, such a result is unlikely to hold unless every sheaf of $G$-representations on $X$ is a $G$-representation, which is true for $X$ quasi-Stein by Kiehl's theorem (see Remark \ref{rem:kiehl}). Since the required cohomological vanishing for $\cO_X$ is very close to quasi-Steinness (see for instance \cite{macpoistein} for a comparison of the two properties in the setting of Berkovich spaces), we do not expect Proposition \ref{prop:factorial} to hold if only the wide openness of $X$ is assumed.
\end{rem}
    
	We give a corollary of Proposition \ref{prop:factorial}.

	\begin{cor}%[of Proposition \ref{wideopenlatt}]
		Let $\bV$ be an absolutely irreducible, residually multiplicity free $G$-representation of rank $d<p$ on an integral, strictly quasi-Stein rigid analytic space $X$. If $\bV$ is not residually absolutely irreducible, assume that $\cO_X^+(X)$ is a UFD. Then there exists a morphism $f\colon X\to X_{\ovl\rho}$ such that $\bV\cong f^\ast\bV^{\sq,\univ}$ as $G$-representations.
	\end{cor}
	
	\begin{proof}
		By Proposition \ref{prop:factorial}, $\bV$ admits a lattice. The result then follows from the universal property of the quadruple $(X_{\ovl\rho}^\sq,\bV^{\sq,\univ},\iota^{\sq,\univ},b^{\sq,\univ})$.
	\end{proof}

\medskip

\section{Constancy modulo $p^n$ of sheaves of $G$-representations}

We denote by $L$ a $p$-adic field, with valuation ring $\cO_L$ having maximal ideal $\fm_L\subset\cO_L$ and residue field $k_L$. We fix a uniformizer $\pi_L$ of $L$. Recall that we have chosen a $p$-adic norm on $\Qp$ satisfying $\lvert p\rvert=p^{-1}$, so that $\lvert\pi_L\rvert=p^{1/e}$, where $e$ is the ramification index of $L/\Q_p$.

Given a sheaf of $G$-representations $\bV$ over a rigid analytic $L$-space $X$, we are interested in finding subdomains of $X$ over which $\bV$ is \emph{constant} modulo a certain power of $\pi_L$. Since such a notion will depend on the choice of a sheaf of lattices for $\bV$, we start by giving two definitions of constancy mod powers of $\pi_L$ for a sheaf of $G$-representations over a formal scheme. We rely on the definitions and discussion in \cite[Section 2]{tortilatt}. %following essentially \cite[Definition 2.3]{tortilatt} and the discussion after that.

Let $\cX$ be a formal scheme over $\cO_L$, and $\cV$ a sheaf of $G$-representations over $\cX$. For every $n\in\Z_{\ge 1}$, $p^n\cO_\cX$ is a closed ideal sheaf on $\cX$. We recall the following definition due to Torti, for which we refer to \cite[Section 2.1]{tortilatt}. %For an abelian group M, we write \ul M for the constant sheaf over \cX of section M.

\begin{defin}\label{Vndef}
For every $n\in\Z_{\ge 1}$, we define $\cV^{(n)}$ as the coherent $\cO_\cX$-module obtained by sheafifying the presheaf $\cV\otimes_{\cO_\cX}\cO_\cX/\pi_L^n\cO_\cX$, equipped with the $\cO_\cX$-linear action of $G$ induced by that on $\cV$.
%of G-representations \cV\otimes_{\ul{\cO_L}}\ul{\cO_L/\pi_L^n}, where \ul{\cO_L}\to\cV is the structure map and \ul{\cO_L}\to\ul{\cO_L/\pi_L^n} is reduction modulo \pi_L^n.
\end{defin}

Obviously, the above definition is independent of the choice of a uniformizer $\pi_L$ of $L$.

For every extension $E$ of $L$, we denote by $\cX_E$ the formal scheme $\cX\otimes_{\cO_L}\cO_E$ and by $\cV_E$ the coherent $\cO_{\cX_E}$-module sheaf $\cV\otimes_{\cO_L}\cO_E$, equipped with the $G$-action obtained by extending $\cO_E$-linearly the $G$-action on $\cV$.

%Assume from now on that \cX is an affine formal scheme. 
Borrowing the terminology of \cite[Section 8.3, Definition 1]{boschbook}, we call \emph{rig-point} of $\cX$ an equivalence class of morphisms $\Spf\cO\to\cX$, where $\cO$ is a valuation ring equipped with a finite local homomorphism $\cO_L\to\cO$, and equivalence is defined as in \cite[Section 7.1.10]{DJ95}. If $E$ is a finite extension of $L$, we say that a rig-point of $\cX$ is an \emph{$E$-rig-point} if it corresponds to the equivalence class of a morphism $\Spf\cO_E\to\cX$. Note that in \cite{boschbook}, the notion of rig-point is only introduced for admissible formal schemes.

Let $E/L$ be a finite extension with ramification index $e_{E/L}$. We define the natural number 
\[ \gamma_{E/L}(n)=(n-1)e_{E/L}+1, \]
following \cite[Definition 2.2]{taiwie}; we refer to \emph{loc. cit.} and Remark \ref{gammarem} below for some comments on this choice. %: it is the smallest exponent $m$ for which the injection $\cO_L\into\cO_E$ induces an injection $\cO_L/\pi_L^n\into\cO_E/\pi_E^m$. 

If $\cA$ is any sheaf of $\cO_\cX$-modules and $x$ an $E$-rig-point of $\cX$ defined by a morphism $\iota_x\colon\Spf\cO_E\to\cX$, we write $\cA_x$ for the fiber of $\cA$ at $x$, i.e. the $\cO_E$-module attached to the pullback of $\cA$ along $\iota_x$.

\begin{defin}\label{lcformdef}
We say that $\cV$ is
\begin{enumerate}[label=(\roman*)]
	\item \emph{pointwise constant mod $p^n$} if, for every finite extension $E$ of $L$, with uniformizer $\pi_E$, the isomorphism class of $\cV_{E,x}^{(n)}$ as an $\cO_E/\pi_E^{\gamma_{E/L}(n)}[G]$-module is independent of the choice of an $E$-rig-point $x$ of $\cX_E$.
	\item \emph{constant mod $\pi_L^n$} if there exists a finite, free $\cO_L/\pi_L^n$-module $V^{(n)}$, equipped with an $\cO_L/\pi_L^n$-linear action of $G$, such that $\bV^{(n)}\cong V^{(n)}\otimes_{\cO_L/\pi_L^n}\cO_\cX/\pi_L^n\cO_\cX$.
\end{enumerate}
When $n=1$, we use the terminology \emph{locally constant mod $p$ up to semisimplification}, with the obvious meaning.
\end{defin}

Part (i) of Definition \ref{lcformdef} is simply a rephrasing of \cite[Definition 2.3]{tortilatt}. 

Semisimplification does not make sense in general for a representation that is not defined over a field, hence why this notion only appears when $n=1$.

%\begin{rem}\mbox{ }
%\begin{itemize}
%\item Part (i) of Definition \ref{lcformdef} is simply a rephrasing of \cite[Definition 2.3]{tortilatt}.
%\item If $\cV$ is constant mod $\pi_L^n$, then it is pointwise constant mod $p^n$: indeed, for every $E$ as in Definition \ref{lcformdef}(i) and $E$-rig-point $x$ of $\cX$, the $\cO_E/\pi_E^n[G]$-module $\cV_{E,x}^{(n)}$ is isomorphic to $V^{(n)}\otimes_{\cO_L/\pi_L^n}\cO_E/\pi_E^n$.
%\end{itemize}
%\end{rem}

\begin{rem}\label{gammarem}
For $m,n\ge 1$, the injection $\cO_L\into\cO_E$ induces an injection 
\begin{equation}\label{gammaeq} \cO_L/\pi_L^n\into\cO_E/\pi_E^m \end{equation}
if and only if the equality of ideals
\begin{equation}\label{ideq} \pi_E^m\cO_E\cap\cO_L=\pi_L^n\cO_L \end{equation}
holds, if and only if $m\in\{(n-1)e_{E/L}+1,ne_{E/L}\}$. Therefore, for fixed $n$, $\gamma_{E/L}(n)=(n-1)e_{E/L}+1$ is the smallest value of $m$ with the above property, as remarked in point (iv) after \cite[Definition 2.2]{tsawie}. Clearly, $\gamma_{E/L}(n)\ge n$.

%Fix $n$ and pick any $m$ for which \eqref{gammaeq} is defined and injective.
Clearly \eqref{gammaeq} is not an isomorphism in general for $m=\gamma_{L/E}(n)$, nor for any other $m$ for which it is defined (simply pick any unramified extension $E/L$). 
It seems that injectivity is only sufficient to deduce congruences in $\cO_E$ from congruences in $\cO_L$, but in fact one can also go the other way around, as shown in point (v) after \cite[Definition 2.2]{tsawie}: for every $\alpha,\beta\in E$,
\begin{equation}\label{alphabeta} \alpha-\beta\in\pi_E^m\cO_E\iff\alpha-\beta\in\pi_L^n\cO_L. \end{equation}
%False: Our choice of $\gamma_{E/L}(n)$ is motivated, as in \emph{loc. cit.}, by the fact that it gives the stronger implication in the direction $\implies$ of \eqref{alphabeta}.
In Section \ref{gammasec} below, we explain how to use the above property to relate constancy and pointwise constancy, thus justifying Definition \ref{lcformdef}.
\end{rem}

Let $X$ be a rigid analytic space over $L$. 

\begin{defin}\label{lcdef}
Let $\bV$ be a sheaf of $G$-representations over $X$, and let $n\in\Z_{\ge 1}$. We say that $\bV$ is:
\begin{enumerate}
	\item \emph{pointwise constant mod $p^n$} if it admits a sheaf of lattices, defined over a formal model $\cX$ of $X$, that is pointwise constant mod $p^n$;
	\item \emph{constant mod $\pi_L^n$} if it admits a sheaf of lattices, defined over a formal model $\cX$ of $X$, that is constant mod $\pi_L^n$.
\end{enumerate}
When $n=1$, we use the terminology \emph{locally constant mod $p$ up to semisimplification}, with the obvious meaning.
\end{defin}

%Definitions \ref{Vndef} and \ref{lcdef}(ii) are essentially borrowed from \cite[Definition 2.3]{tortilatt}.

\begin{rem}\label{piconst}
Let $\pi\colon Y\to X$ be a surjective morphism of rigid analytic spaces, and $\bV$ a sheaf of $G$-representations on $X$. Let $n\ge 1$. Then $\bV$ is pointwise constant mod $p^n$ on a subspace $U\subset X$ if and only if $\pi^\ast\bV$ is pointwise constant mod $p^n$ on $\pi^{-1}(U)$. The analogous statement with pointwise constancy mod $p^n$ replaced with constancy mod $\pi_L^n$ is false; the pullback $\pi^\ast\bV^{(n)}$ might be constant even if $\bV^{(n)}$ is not.
\end{rem}

\subsection{Constancy modulo $p^n$ of rigid analytic functions}

Let $\cX$ be an $\cO_L$-formal scheme. Let $f\in\cO_\cX(\cX)$. For every $n\in\Z_{\ge 1}$, we write $f^{(n)}$ for the image of $f$ under the natural projection $\cO_\cX(\cX)\to\cO_\cX(\cX)/\pi_L^n\cO_\cX(\cX)$. 

Let $E$ be a finite extension of $L$. We write $f_E$ for the element of $\cO_{\cX_E}(\cX_E)$ defined by $f\otimes 1$ via the natural map $\cO_\cX(\cX)\otimes_{\cO_L}\cO_E\to\cO_{\cX_E}(\cX_E)$. 
If $x$ is an $E$-rig-point of $\cX$ attached to a morphism $\iota_x\colon\Spf\cO_E\to\cX$, we denote by $f_x$ the evaluation of $f$ at $x$, i.e. the element of $\cO_E$ defined by the pullback $\iota_x^\ast f$. When we write $f_x^{(n)}$, we are first evaluating at $x$, and then reducing modulo $\pi_E^{(n)}$ the resulting element of $\cO_E$.

We give two notions of mod $\pi_L^n$ constancy for elements of $\cO_\cX(\cX)$, along the lines of what we did for sheaves of $G$-representations in Definition \ref{lcformdef}. 

\begin{defin}\label{lcfun}
Let $n\in\Z_{\ge 1}$. We say that $f\in\cO_\cX(\cX)$ is:
\begin{enumerate}[label=(\roman*)]
	\item \emph{pointwise constant mod $p^n$} if, for every finite extension $E$ of $L$ with uniformizer $\pi_E$, the element $f_{E,x}^{(\gamma_{E/L}(n))}\in\cO_E/\pi_E^{\gamma_{E/L}(n)}$ is independent of the choice of an $E$-rig-point $x$ of $\cX_E$;
	\item \emph{constant mod $\pi_L^n$} if $f^{(n)}$ belongs to the image of the structure map $\cO_L/\pi_L^n\to\cO_\cX(\cX)/\pi_L^n\cO_\cX(\cX)$.
\end{enumerate}
\end{defin}

As was the case with Definition \ref{lcformdef}, condition (ii) in Definition \ref{lcfun} is stronger than (i). For an example of a formal scheme and a function on it satisfying (i) but not (ii), we refer to Example \ref{exdisc} below.

%\begin{rem}
%Write $A=\cO_\cX(\cX)$. It is easy to check from the definition that, for every finite extension $L'/L$, $f\in A$ is pointwise constant mod $p^n$ if and only if $f\otimes 1\in A\otimes_{\cO_L}\cO_{L'}$...
%\end{rem}

Now let $X$ be a rigid analytic space over $L$. For a formal model $\cX$ of $X$, we consider functions in $\cO_\cX(\cX)$ as elements of $\cO_X(X)$ via \eqref{dejong}.

\begin{defin}
Let $n\in\Z_{\ge 1}$. We say that $f\in\cO_X(X)$ is pointwise constant mod $p^n$ (respectively, constant mod $\pi_L^n$) if there exists a formal model $\cX$ of $X$ such that $f\in\cO_\cX(\cX)$ and $f$ is pointwise constant mod $p^n$ (respectively, constant modulo $\pi_L^n$) as a function on $\cX$.
\end{defin}

\subsection{Constancy versus pointwise constancy}\label{gammasec}

We use Remark \ref{gammarem} to compare the notions of constancy and pointwise constancy for functions on formal schemes, and, as a consequence, for sheaves of $G$-representations. The following lemma shows that pointwise constancy mod $\pi_L^n$ fits between constancy mod $p^n$ and $p^{n+1}$. Via the calculations of Section \ref{secUV}, we will see that it actually sits \emph{strictly} between these two conditions.

Let $\cX=\Spf A$ be an affine $\cO_L$-formal scheme with $A$ an integral domain. %\andr{check where else to assume reducedness}

\begin{lemma}\label{gammafun}
Let $f\in A$, and $n\in\Z_{\ge 1}$.
\begin{enumerate}[label=(\roman*)]
\item If $f$ is constant mod $\pi_L^n$, then it is pointwise constant mod $p^n$.
\item If $f$ is pointwise constant mod $p^{n+1}$, then it is constant mod $\pi_L^{n}$.
\end{enumerate}
\end{lemma}

\begin{proof}
%Let E/L be a finite extension and x,y be two E-rig points. 
If $f$ is constant mod $\pi_L^n$, then by definition its image under $A\onto A/\pi_L^n$ belongs to the image of $\cO_L/\pi_L^n\to A/\pi_L^n$. By tensoring along the injection $\cO_L/\pi_L^n\into\cO_E/\pi_E^{\gamma_{E/L}(n)}$, one obtains that the image of $f\otimes 1\in A_E\coloneqq A\otimes_{\cO_L}\cO_E$ under $A_E\onto A_E/\pi_E^{\gamma_{E/L}(n)}$ also belongs to the image of $\cO_L/\pi_L^n\to A/\pi_L^n\to A_E/\pi_E^{\gamma_{E/L}(n)}$. Then, specializing at two different $E$-rig-points of $\cX_E$ obviously yields the same element of $\cO_E/\pi_E^{\gamma_{E/L}(n)}$ (belonging to the subring $\cO_L/\pi_L^n$). This proves (i).
%By Remark \ref{gammarem}, f(x)-f(y)\in\pi_E^{\gamma_{E/L}(n)} if and only if f(x)-f(y)

Now assume that $f$ is pointwise constant mod $p^{n+1}$. We show that $f$ is constant mod $\pi_L^n$. %Fix an L-rig-point x of \cX, and let g=f-f(x)\in A. Then g does not belong to \pi_L^nA. 
%For every finite extension $E/L$ and $E$-rig-point $x\colon A\to\cO_E$ of $\cX$, let $I_x$ be the ideal $(\ker x,\pi_E^{ne_{E/L}})$. As $x$ varies among the rig-points of $\cX$, one has $\bigcap_xI_x=\pi_L^nA$. 
Fix a rig-point $x_0$ of $\cX$, defined over a finite extension $E_0/L$, and let $g=f\otimes 1-1\otimes f(x_0)\in A_0\coloneqq A\otimes_{\cO_L}\cO_{E_0}$. We wish to show that  %Since $f$ is not constant mod $\pi_L^n$, 
$g$ is constant mod $\pi_L^n$, i.e. $g\in\pi_{E_0}^{ne(E_0/L)}A_0$. %its image $\ovl g$ under $A_0\to A_0\pi_L^nA_0=A_0/\pi_{E_0}^{ne(E_0/L)}$ is non-zero. 
Let $k$ be the largest integer such that $g\in\pi_{E_0}^kA_0$, and assume by contradiction that $k<ne(E_0/L)$. By implicitly replacing $g$ with $\pi_{E_0}^{-k}g$, we can assume that $g\in A_0\setminus\pi_{E_0}A_0$, and that $g(x_0)\in\pi_{E}\cO_E$ for every finite extension $E/E_0$ and $E$-rig-point of $A_0$ (since $f$ is pointwise constant mod $p^{n+1}$). Let $\ovl g$ be the image of $g$ in $A_0/\pi_{E_0}$.

Since $A$ is reduced, so is $A_0$, and we can find a covering of $\Spf A_0$ by reduced, admissible affine formal schemes $\Spf B_i$ (note that $A_0$ is not necessarily admissible; if it is, we can simply take $\Spf A_0$ as a covering). Let $g_i$ be the image of $g$ under $A_0\to B_i$. If we prove that $g_i\in\pi_{E_0} B_i$ for every $i$, then we deduce that $g\in\pi_{E_0}A_0$ since the preimage of $\prod_i\pi_{E_0}B_i$ under $A_i\to\prod_iB_i$ is $\pi_{E_0}A_0$. 
Since the rig-points of $B_i$ form a subset of the rig-points of $A_0$, we have that $g_i(x)\in\pi_E\cO_E$ for every finite extension $E/E_0$ and $E$-rig-point of $B_i$. The specialization map from the rig-points of $\Spf B_i$ to the closed points of $\Spec B_i/\pi_{E_0}$ is surjective (see \cite[7.1.5/4]{bgr}), so we deduce that $\ovl g$ vanishes on every closed point of the reduced $\Spec B_i/\pi_{E_0}$, which implies $\ovl g_i=0$.
%Since $A_0\in\Ad_{\cO_{E_0}}$, $A_0/\pi_{E_0}$ is a reduced quotient of a polynomial ring over $\cO_{E_0}/\pi_{E_0}$. 
%Since ?? , $g$ belongs to the Jacobson radical of $A_0$, hence its image $\ovl g$ in $A_0/\pi_{E_0}$ belongs to the Jacobson radical of $A_0/\pi_{E_0}$. Since $A_0\in\Ad_{\cO_{E_0}}$, $A_0/\pi_{E_0}$ is a reduced quotient of a polynomial ring over $\cO_{E_0}/\pi_{E_0}$, hence its Jacobson radical is 0.
\end{proof}

With the next lemma, we reinterpret (pointwise) constancy of the reduction of a free lattice in terms of that of functions on the underlying formal scheme. 
Let $\cV$ be a free, rank $d$ sheaf of $G$-representations over an affine formal $\cO_L$-scheme $\cO_X$. By choosing a basis for $\cV$, we attach to it a continuous representation $\rho_\cV\colon G\to\GL_d(\cO_\cX)$. Let $n\in\Z_{\ge 1}$.

\begin{lemma}\label{gammalatt}
The lattice $\cV$ is pointwise constant mod $p^n$ (respectively, constant mod $\pi_L^n$) if and only if the matrix coefficients of $\rho_\cV$ are pointwise constant mod $p^n$ (respectively, constant mod $\pi_L^n$).
In particular:
\begin{enumerate}[label=(\roman*)]
\item if $\cV$ is constant mod $\pi_L^n$, then it is pointwise constant mod $p^n$;
\item if $\cV$ is pointwise constant mod $p^{n+1}$, then it is constant mod $\pi_L^{n}$.
\end{enumerate}
\end{lemma}

%We deduce a simple lemma from Remark \ref{gammarem}. Keep the notation of Definition \ref{lcformdef}.
%
%\begin{lemma}\label{gammalemma}
%If \cV is constant mod \pi_L^n, then it is pointwise constant mod p^n.
%\end{lemma}

\begin{proof}
The proof of the first statement is straightforward. The second statement follows by combining the first one with Lemma \ref{gammafun}. 
%Assume that \cV is constant mod \pi_L^n. 
%Choosing a basis for the lattice \cV, we can attach to it a continuous representation \rho_\cV\colon G\to\GL_d(\cO_\cX). The reduction of \rho_\cV modulo \pi_L^n is of the form \rho_\cV^{(n)}\otimes_{\cO_L/\pi_L^n}\cO_\cX/\pi_L^n for a continuous representation \rho_\cV^{(n)}\colon G\to\GL_n(\cO_L/\pi_L^n). The specialization of \rho_\cV at E is 
\end{proof}

\subsection{Residue subdomains of rigid analytic spaces}\label{secUV}

Let $\cX$ be formal scheme over $\cO_L$, and $X=\cX^\rig$. We consider functions in $\cO_\cX(\cX)$ as elements of $\cO_X(X)$ via \eqref{dejong}. In particular, for a point $x\in X(L)$, it makes sense to consider the ideal $I_x=\{f\in\cO_\cX(\cX)\,\vert\,f(x)=0 \}$
of $\cO_\cX(\cX)$.

\begin{defin}\label{resdef}
	Let $n\in\Z_{\ge 1}$. 
	The \emph{mod $p^n$ wide open $\cX$-residue neighborhood} of $x$ in $X$ is the open subdomain
	\[ U_{x,\cX}^{(n)}=\{y\in X\,\vert\, \lvert f(y)\rvert<p^{(1-n)/e} \quad \forall f\in I_x\}. \]
 The \emph{mod $\pi_L^n$ affinoid $\cX$-residue neighborhood} of $x$ in $X$ is the rational subdomain
	\[ V_{x,\cX}^{(n)}=\{y\in X\,\vert\, \lvert f(y)\rvert\le p^{-n/e} \quad \forall f\in I_x\}. \]
\end{defin}

We emphasize in the notation that $U_{x,\cX}^{(n)}$ and $V_{x,\cX}^{(n)}$ depend on the formal model $\cX$ of $X$, not only on the rigid analytic space $X$. 

\begin{rem}\label{wores}
If $n=1$, then $U_{x,\cX}^{(1)}$ is the $\cX$-residue subdomain $\spec^{-1}(\spec(x))\subset X$ defined in \cite[7.1.10]{DJ95}, also called the ``tube'' of $x$ in the work of Berthelot. In particular, if $X$ is an irreducible wide open and $\cX=\Spf\cO_X^+(X)$, then $\cO_\cX(\cX)=\cO_X^+(X)$ is a local ring, so that the special fiber consists of a single point and $U_{x,\cX}^{(1)}=X$.

For every $x\in X(L)$, each of the collections $\{U_{x,\cX}^{(n)}\}_{n\ge 1}$ and $\{V_{x,\cX}^{(n)}\}_{n\ge 1}$ is a fundamental system of open neighborhoods of $x$.
\end{rem}

Assume from now on that $\cX=\Spf\cA$ is affine. By \cite[Lemma 7.1.9]{DJ95}, the $L$-points of $\cX^\rig$ are in bijection with the rig-points of $\cX$, i.e. the maximal ideals of $\cA$ not containing $\pi_L$; in particular, the ideal $I_x$ of $\cA$ is non-zero for every $x\in X(L)$. %\andr{also non-zero in non-affine case?? or can build residue subdomains by patching??}

For every $x$ and $n$, $U_{x,\cX}^{(n)}$ and $V_{x,\cX}^{(n)}$ are non-empty admissible open neighborhoods of $x$ in $X$: indeed, since $\cA$ is Noetherian $I_x$ is finitely generated, say by $(f_1,\ldots,f_m)$, and
\begin{align*} U_{x,\cX}^{(n)}&=\{y\in X\,\vert\, \lvert f_i\rvert<p^{(1-n)/e} \quad \forall i=1,\ldots,m\}, \\
	V_{x,\cX}^{(n)}&=\{y\in X\,\vert\, \lvert f_i\rvert\le p^{-n/e} \quad \forall i=1,\ldots,m\} \end{align*}
which are admissible by \cite[9.1.4/5]{bgr}, and non-empty since they both contain $x$. 

We give a better description of $U_{x,\cX}^{(n)}$. Let $\cA_x^{(n)}$, respectively $\cB_x^{(n)}$, be the images of $\cA$ under the maps $\cA\into\cO_X(X)\to\cO_X(U_{x,\cX}^{(n)})$, respectively $\cA\into\cO_X(X)\to\cO_X(V_{x,\cX}^{(n)})$, the second map being in both cases the natural restriction. %, and let \cB_x^{(n)} be the image of \cA under the maps $\cA\into\cO_X(X)\to\cO_X(V_{x,\cX}^{(n)})$.

\begin{lemma}\label{Anwo}
	For every $x$ and $n$, $U_{x,\cX}^{(n)}$ is the wide open subdomain $(\Spf\cA_x^{(n)})^\rig\subset(\Spf\cA)^\rig$, and $V_{x,\cX}^{(n)}$ is the affinoid subdomain $\Spm\cB_x^{(n)}[1/\pi_L]\subset(\Spf\cA)^\rig$. 
	If $(f_1,\ldots,f_m)$ is any finite set of generators of $I_x$, then 
	\begin{align*} \cA_x^{(n)}&=\cA[[\pi_L^{1-n}f_1,\ldots,\pi_L^{1-n}f_m]], \\
		\cB_x^{(n)}&=\cA\langle\pi_L^{-n}f_1,\ldots,\pi_L^{-n}f_m\rangle. \end{align*}
Moreover, $V_{x,\cX}^{(n)}\subsetneq U_{x,\cX}^{(n)}\subsetneq V_{x,\cX}^{(m)}$ if $m<n$, and $\bigcap_{n\in\Z_{\ge 1}}U_{x,\cX}^{(n)}=\bigcap_{n\in\Z_{\ge 1}}V_{x,\cX}^{(n)}=\{x\}$. 
\end{lemma}

\begin{proof}
The descriptions of $U_{x,\cX}^{(n)}$ and $V_{x,\cX}^{(n)}$ and the regular functions on them follow from standard calculations in rigid analytic geometry.

If $m<n$, the inclusions $V_{x,\cX}^{(n)}\subset U_{x,\cX}^{(n)}\subset V_{x,\cX}^{(m)}$ are trivial; they are strict since the outer terms are affinoid while the central one is wide open. The last statement follows from $\bigcap_{n\in\Z_{\ge 1}}U_{x,\cX}^{(n)}=\bigcap_{n\in\Z_{\ge 1}}V_{x,\cX}^{(n)}=\{y\in X\,\vert\, f(y)=0\quad\forall f\in I_x\}=x$.
\end{proof}

We will write $\cU_{x,\cX}^{(n)}=\Spf\cA_x^{(n)}$ and $\cV_{x,\cX}^{(n)}=\Spf\cB_x^{(n)}$ for the formal models of $U_{x,\cX}^{(n)}$ and $V_{x,\cX}^{(n)}$ provided to us by Lemma \ref{Anwo}. 

\begin{rem}\label{fincov}
Let $\pi\colon\cY\onto\cX$ be a finite cover of affine formal schemes, inducing a map $Y\to X$ of generic fibers that we still denote by $\pi$. Then a direct calculation shows that, for every $y\in Y(\Qp)$ and $n\ge 1$, $\pi(U_{y,\cY}^{(n)})\subset U_{\pi(y),\cX}^{(n)}$ and $\pi(V_{y,\cY}^{(n)})\subset V_{\pi(y),\cX}^{(n)}$.
\end{rem}

Our interest in the neighborhoods $U_{x,\cX}^{(n)}$ and $V_{x,\cX}^{(n)}$ lies in Theorem \ref{Unconst} below, that in turn is an immediate consequence of the following simple lemma. 
Let $n\in\Z_{\ge 1}$, and 
\begin{align*} \pi_{\cA,n}&\colon\cA\into\cA_x^{(n)}\to\cA^{(n)}_x/\pi_L^n\cA^{(n)}_x \\
	\pi_{\cB,n}&\colon\cA\into\cB_x^{(n)}\to\cB^{(n)}_x/\pi_L^n\cB^{(n)}_x \end{align*}
be the compositions of the natural map $\cA\into\cA_x^{(n)}$ and $\cA\into\cB_x^{(n)}$ with reduction modulo $\pi_L^n$.

\begin{lemma}\label{funcon}
	Every $f\in\cA$ is:
	\begin{enumerate}[label=(\roman*)]
	\item pointwise constant mod $p^n$ on $\cU_{x,\cX}^{(n)}$, hence on $U_{x,\cX}^{(n)}$; %on the value of $\pi_{\cA,n}(f)$ at $y\in U_{x,\cX}^{(n)}(L)$ is independent of $y$.
	\item constant mod $\pi_L^n$ on $\cV_{x,\cX}^{(n)}$, hence on $V_{x,\cX}^{(n)}$. %The function $\pi_{\cB,n}(f)$ is constant, in the sense that it belongs to the image of the structure map $\cO_L\to\cB_x^{(n)}/\pi_L^n\cB_x^{(n)}$.
	\end{enumerate}
\end{lemma}

%Since the choice of such formal models is uniquely determined by the choice of the formal model \cX of X, we sometimes say that a function is (pointwise) constant on 

%Obviously property (ii) is stronger than (i). 
%One sees from Example \ref{exdisc} below that it is in general not true for the function $\pi_{\cA,n}(f)$, i.e. one needs to restrict to the affinoid $V_{x,\cX}^{(n)}$ for this stronger version of mod $\pi_L^n$ constancy to hold.

\begin{rem}
By Lemma \ref{gammalatt}, constancy mod $\pi_L^n$ is an intermediate condition between pointwise constancy mod $p^n$ and mod $p^{n+1}$: this is reflected in the fact that 
\[ U_{x,\cX}^{(n+1)}\subset V_{x,\cX}^{(n+1)}\subset U_{x,\cX}^{(n)}. \]
In Section \ref{optimalsec}, we show that, under certain conditions on $\cX$, the constancy neighborhoods $U_{x,\cX}^{(n)}$ and $V_{x,\cX}^{(n)}$ are optimal, so that the above inclusions actually follow from Lemma \ref{gammalatt}.
\end{rem}

%As we already remarked, constancy mod $\pi_L^n$ is stronger than pointwise constancy mod $p^n$. This is reflected in the fact that $U_{x,\cX}^{(n)}(L)$ is strictly contained in $V_{x,\cX}^{(n)}(L)$, as one sees explicitly in the following example.

\begin{proof}
	Let $f(x)\in\cO_L$ be the value of $f$ at $x$. Then $g=f-f(x)\in I_x$. Now by definition of $U_{x,\cX}^{(n)}$, for every finite extension $E/L$ and $E$-point $y$ of $U_{x,\cX}^{(n)}$, $\lvert g(y)\rvert<p^{\frac{1-n}{e}}$, i.e. $v_p(g(y))>(n-1)/e$, where $g(y)\in E$. Let $e^\prime$ be the ramification index of $E/L$. Since the valuation $v_p$ on $E$ is discrete and $v_p(\pi_E)=1/(ee^\prime)$, we deduce that $v_p(g(y))\ge (n-1)/e+1/(ee^\prime)\ge v_p(\pi_E^n)$, as desired.
	
%	the description of $\cA_x^{(n)}$ in Lemma \ref{Anwo} shows that there exists $g_0\in\cA_x^{(n)}$ such that $g=\pi_L^{n-1}g_0$. ???
	
	For part (ii), keep the above notation. Via the description of $\cB_x^{(n)}$ in Lemma \ref{Anwo}, we can find $g_1\in\cB_{x}^{(n)}$ such that $g=\pi_L^ng_1$. In particular, the image of $f$ under $\cB_x^{(n)}\to\cB_x^{(n)}/\pi_L^n\cB_x^{(n)}$ coincides with the image of $f(0)$ under $\cO_L\to\cO_L/\pi_L^n\cO_L$.
\end{proof}

\begin{rem}\label{UVchar}
	%Let E/L be a finite extension with valuation ring \cO_E and uniformizer \pi_E. For a \cO_L-formal scheme \cX, ther is a natural map from E-rig-points of \cX to E-rig-points of \cX\otimes_{\cO_L}\cO_L/\pi_L^n: we map x\colon\cX\to\cO_E to 
	Let $\cX$ be an affine formal model of $X$, $x\in X(L)$ and $n\in\Z_{\ge 1}$. As one sees directly from their definitions, the subspaces $U_{x,\cX}^{(n)}$ and $V_{x,\cX}^{(n)}$ are characterized by the following properties:
	\begin{itemize}
		\item Let $\wtl x\colon\cO_L\to\cX$ be the rig-point attached to $x\in X(L)$. 
		The subdomain $U_{x,\cX}^{(n)}$ is characterized by the following property: for every finite extension $E/L$, a point $y\in X(E)$ belongs to $U_{x,\cX}^{(n)}$ if and only if the composition of the associated rig-point $\wtl y\colon \cX\to\cO_E$ with $\cO_E\to\cO_E/\pi_E^n$ coincides with the composition of $\wtl x$ with $\cO_E\to\cO_E/\pi_E^n$. 
		\item Let $\cO_\cX(\cX)\to\cO_\cX(\cX)\otimes_{\cO_L}\cO_L/\pi_L^n$ be the reduction map mod $\pi_L^n$. Then $U_{x,\cX}^{(n)}$ is the largest among the subspaces $U\subset X$ such that $x\in U(L)$ and every element of $\cO_\cX(\cX)$ is constant mod $\pi_L^n$ over $U$.
	\end{itemize}
%	\andr{proof or obvious??}
\end{rem}

\begin{ex}\label{exdisc}
	Let $\cA=\cO_L\langle\zeta_1,\ldots,\zeta_s\rangle[[\xi_1,\ldots,\xi_t]]$, so that $X$ is the product of an $s$-dimensional affinoid unit disc and a $t$-dimensional wide open unit disc. Let $x=0$, so that $I_x=(\zeta_1,\ldots,\zeta_s,\xi_1,\ldots,\xi_t)$, 
	%\[ U_{x,\cX}^{(n)}=\{y\in X\,\vert\, \lvert \zeta_i(y)\rvert<p^{(n-1)/e} \, \forall i=1,\ldots,s, \text{ and }\lvert \xi_i(y)\rvert<p^{(n-1)/e} \, \forall i=1,\ldots,t\}, \]
	\[ \cA_x^{(n)}=\Spf\cO_L[[\pi_L^{1-n}\zeta_1,\ldots,\pi_L^{1-n}\zeta_s,\pi_L^{1-n}\xi_1,\ldots,\pi_L^{1-n}\xi_t]], \]
	and $U_{x,\cX}^{(n)}=(\Spf\cA_x^{(n)})^\rig$, the wide open $(s+t)$-dimensional disc of center 0 and radius $p^{\frac{1-n}{e}}$. Note that $\cA_x^{(n)}$ is a UFD, being a ring of formal power series over a PID. If one replaces $x$ with any other point of $X$, one simply obtains a wide open disc of the same radius, centered at $x$. Similarly, $V_{x,\cX}^{(n)}$ is the affinoid disc of center $x$ and radius $p^{-n/e}$. 
	
	If for instance $n=1$ and $x=0$, then $\cA_0^{(1)}=\Spf\cO_L[[\zeta_1,\ldots,\zeta_s,\xi_1,\ldots,\xi_t]]$. If $E$ is a finite extension of $L$ with uniformizer $\pi_E$, an $E$-point of $U_{0,\cX}^{(1)}$ corresponds to a choice of values of $\zeta_1,\ldots,\zeta_s,\xi_1,\ldots,\xi_t$ in the maximal ideal $\fm_E=(\pi_E)$ of the valuation ring $\cO_E\subset E$. In particular, the value of an arbitrary $f\in\cA_0^{(1)}$ at any $E$-point of $U_{0,\cX}^{(1)}$ is congruent to $f(0)$ modulo $\pi_L$.
	
	The image of $\cA$ under $\cA_0^{(1)}\to\cA_0^{(1)}/\pi_L\cA_0^{(1)}$ is isomorphic to $k_L[\zeta_1,\ldots,\zeta_s,\xi_1,\ldots,\xi_t]$: in particular, the structure map $k_L\to\cA/\pi_L\cA_0^{(1)}$ is not surjective, so that not every element of $\cA$ is constant mod $\pi_L$ over $U_{0,\cX}^{(1)}$. On the other hand, $\cB_x^{(1)}$ is the ring of power-bounded functions over the affinoid disc $V_{0,\cX}^{(1)}$ of center 0 and radius $p^{-1}$, isomorphic to $\cO_L[[\pi_L^{-1}\zeta_1,\ldots,\pi_L^{-1}\zeta_s,\pi_L^{-1}\xi_1,\ldots,\pi_L^{-1}\xi_t]]$: the image of $\cA$ under $\cB_0^{(1)}\to\cB_0^{(1)}/\pi_L\cB_0^{(1)}$ is obviously $k_L$, so that every element of $\cA$ is constant mod $\pi_L$ over $V_{0,\cX}^{(1)}$.
\end{ex}

%\andr{pi or piL--remove all L??}

%\andr{change r to rho??}

\begin{ex}\label{exann}
Let $\cA=\cO_L\langle \zeta_1,\zeta_2\rangle/(\zeta_1\zeta_2-\pi_L^m)$, so that $X$ is the 1-dimensional affinoid annulus of radii $p^{-m/e}$ and 1 (in either variable). Let $(x_1,x_2)\in X(L)$. Then $I_{(x_1,x_2)}=(\zeta_1-x_1,\zeta_2-x_2)$. Therefore:
\begin{itemize}
\item $U_{x,\cX}^{(n)}$ is the locus where $\lvert\zeta_1-x_1\rvert<p^{\frac{1-n}{e}}$ and $\lvert\zeta_2-x_1\rvert<p^{\frac{1-n}{e}}$. Using the fact that $\zeta_1\zeta_2=\pi_L^m$ and $x_1x_2=\pi_L^m$, we can rewrite the second inequality as 
\begin{equation}\label{zeta2} \lvert \zeta_1-x_1\rvert<p^{\frac{1-n+m}{e}}\lvert\zeta_1x_1\rvert. \end{equation} Since $\lvert x_1\rvert\ge p^{-m/e}$, if $n-1>m$ then the first inequality implies that $\lvert\zeta_1\rvert=\lvert x_1\rvert$. Therefore, \eqref{zeta2} becomes $\lvert\zeta_1-x_1\rvert<p^{\frac{1-n+m}{e}}\lvert x_1\rvert^2$, where the right-hand side lies in the interval $[p^{\frac{1-n-m}{e}},p^{\frac{1-n+m}{e}}]$. In conclusion, if $n-1>m$, then $U_{x,\cX}^{(n)}$ is the wide open disc of center $x_1$ and radius $\min\{p^{\frac{1-n}{e}},p^{\frac{1-n+m}{e}}\lvert x_1\rvert^2\}$.
\item A calculation completely analogous to the previous one shows that, if $n>m$, $V_{x,\cX}^{(n)}$ is the affinoid disc of center $x_1$ and radius $\min\{p^{\frac{-n}{e}},p^{\frac{-n+m}{e}}\lvert x_1\rvert^2\}$.
\end{itemize}
\end{ex}

Note that in both Examples \ref{exdisc} and \ref{exann}, reparameterizing a disc or an annulus to give it a different radius does not give different local constancy domains, as one would expect and can check directly from Definition \ref{resdef}.

%As one sees immediately in Example \ref{exdisc}, Lemma \ref{indepy} only implies that the value of $\pi_n(f)$ on the $L$-points of $U_{x,\cX}^{(n)}$ is constant, but not that $\pi_n(f)$ is ``constant'' in the sense that it belongs to the image of the structure map $\cO_L\to\cA_x^{(n)}$. For this stronger property, one has to further restrict the function to an \emph{affinoid} subdomain of $U_{x,\cX}^{(n)}$. We give the corresponding statement below.

\subsection{Explicit local constancy of $G$-representations}\label{secstcri}

We derive from Lemma \ref{funcon} a consequence about the mod $\pi_L^n$ variation of sheaves of $G$-representations.

\begin{thm}\label{Unconst}
Let $X$ be a rigid analytic space over $L$, and $\bV$ a $G$-representation on $X$ admitting a free lattice defined over an affine formal model $\cX$ of $X$. Then, for every $x\in X(L)$ and $n\in\Z_{\ge 1}$, $\bV$ is pointwise constant modulo $p^n$ over $U_{x,\cX}^{(n)}$, and constant modulo $\pi_L^n$ over $V_{x,\cX}^{(n)}$.
\end{thm}

\begin{rem}
Since every sheaf of $G$-representations $\bV$ on $X$ locally admits a free lattice over an affine formal model by Proposition \ref{chenlatt}, we can deduce the following statement for an arbitrary sheaf of $G$-representations $\bV$ on $X$. For every $x \in X(L)$, there exists a positive integer $N_0$ (which will depend on $x$) such that $\bV$ is constant (resp. pointwise constant) modulo $\pi_L^n$ (resp. modulo $p^n$) over $V_{x,\cX}^{( n+N_0 )}$ (resp. $U_{x,\cX}^{(n+ N_0 )}$). Indeed, this can simply be deduced from Theorem \ref{Unconst} and from the fact that the collections of $U_{x,\cX}^{(n)}$ and $V_{x,\cX}^{(n)}$ form two fundamental systems of open neighborhoods of the point $x$. The integer $N_0$ is a priori necessary in order to restrict ourselves a neighborhood of $x$ over which $\bV$ admits a free lattice. As we do not need the theorem in this generality, we content ourselves with Theorem \ref{Unconst}.
\end{rem}

\begin{proof}
The assumptions of Theorem \ref{Unconst} can only be satisfied if $X$ admits an affine formal model. Let $\cV$ be a free lattice for $\bV$, defined over the affine formal model $\cX$. After choosing an $\cA$-basis of $\cV(\cX)$, we can attach to the $G$-action on $\cV(\cX)$ a continuous group homomorphism
\[ \rho_\cV\colon G\to\GL_d(\cA). \]
Every matrix coefficient of $\rho_\cV$ is:
\begin{itemize}
\item pointwise constant mod $p^n$ over $U_{x,\cX}^{(n)}$ by Lemma \ref{funcon}, so that $\rho_\cV$ is pointwise constant mod $p^n$ over $U_{x,\cX}^{(n)}$;
\item constant mod $\pi_L^n$ over $V_{x,\cX}^{(n)}$, so that the matrix coefficients of the representation
\[ \rho_\cV^{(n)}\colon G\to\GL_d(\cA)\to\GL_d(\cA/\pi_L^n\cA), \]
obtained by composing $\rho_\cV$ with the natural projection, lie in the image of
\[ \GL_d(\cO_L/\pi_L^n\cO_L)\to\GL_d(\cA/\pi_L^n\cA), \]
where we are applying the structure map to each coefficient.
In particular, $\rho_\cV^{(n)}$ is obtained by extension of scalars from a representation $G\to\GL_d(\cO_L/\pi_L^n\cO_L)$. 
Any free $\cO_L/\pi_L^n\cO_L$-module $V^{(n)}$ of rank $d$, equipped with the above action of $G$ in any choice of basis, will satisfy the condition required in Definition \ref{lcdef}.\qedhere
\end{itemize}
\end{proof}

\begin{rem}\label{invconj}
The statement of Theorem \ref{Unconst} is independent of the choice of $\cA$-basis of $\cV(\cX)$ made in the proof. This is expected, since different choices of basis amount to conjugating $\rho_\cV$ by matrices with coefficients in $\cA$, and for every subspace $U$ of $\cX$, an $\cA$-linear combination of elements of $\cO_\cX(\cX)$ that are constant mod $\pi_L^n$ (respectively, pointwise constant mod $p^n$) over $U$ is still constant mod $\pi_L^n$ (respectively, pointwise constant mod $p^n$).
\end{rem}

%NEED AFFINE for actual locconst???? compare with rhobar constant according to chenevier
%but chenevier states over wideopen...

\begin{rem}\label{UVX}
The neighborhoods $U_{x,\cX}^{(n)}$ and $V_{x,\cX}^{(n)}$ depend on the choice of a formal model $\cX$ of $X$ supporting a lattice for $\bV$, so different choices of $\cX$ will give different constancy neighborhoods for $\bV$. In some cases, there is an optimal choice of formal model, giving the exact mod $p^n$ constancy loci for $\bV$: we discuss this in Section \ref{optimalsec}.
\end{rem}

In the special case of wide open spaces, we have proven the existence of integral models under certain assumptions on $\bV$. In precise terms, we have the following:

\begin{cor}\label{woloc}
Let $X$ be a strictly quasi-Stein rigid analytic space over $L$, and let $\bV$ be an absolutely irreducible sheaf of $G$-representation on $X$ of rank $d<p$. Then $\bV$ is a $G$-representation on $X$, and if either $\bV$ is residually absolutely irreducible, or it is multiplicity-free and $\cO_X^+(X)$ is a UFD, then for every $x\in X(L)$ and $n\in\Z_{\ge 1}$, $\bV$ is pointwise constant modulo $p^n$ over the explicit wide open $U_{x,\cX}^{(n)}$ and constant modulo $\pi_L^n$ over the explicit open affinoid $V_{x,\cX}^{(n)}$.
%Let $X$ be a wide open rigid analytic space over $L$, and $\bV$ an absolutely irreducible, residually multiplicity-free $G$-representation on $X$. For every $x\in X(L)$ and $n\in\Z_{\ge 1}$, $\bV$ is pointwise constant modulo $p^n$ over the explicit wide open $U_{x,\cX}^{(n)}$ and constant modulo $\pi_L^n$ over the explicit open affinoid $V_{x,\cX}^{(n)}$.
\end{cor}

%Thanks to Remark \ref{maxmodel}, one is always allowed to choose $\cX=\Spf\cO_X^+(X)$.

\begin{proof}
The proof is immediate from combining Theorem \ref{Unconst}, Proposition \ref{prop:factorial} and Remark \ref{rem:kiehl}.
%We prove (i). Let $d$ be the rank of $\bV$, and $X^\sq$ be the $\GL_d$-torsor of trivializations of $X$. By construction of $X^\sq$, $\pi^\ast\bV$ is a $G$-representation over $X^\sq$ admitting a lattice defined over $\Spf\cO^+_{X^\sq}(X^\sq)$. %, and by Proposition \ref{wideopenlatt}, $\pi^\ast\bV$ admits a lattice over $\Spf\cO_{X^\sq}^+(X^\sq)$. 
%Therefore, we can apply Theorem \ref{Unconst} to $\pi^\ast\bV$. The conclusion then follows from Remarks \ref{piconst} (using the fact that $\pi$ is induced from a map of formal models) and \ref{fincov}.
\end{proof}

\begin{rem}\label{fielddef}
Corollary \ref{woloc} has the interesting consequence that the mod $\pi_L^n$ representation attached to a point $y$ of the neighborhood $U_{x,\cX}^{(n)}$ is always defined over $\cO_L/\pi_L^n$, even if $y$ is not defined over $L$. A statement of this kind does not seem to appear in any of the previous works dealing with local constancy.
\end{rem}

\subsection{Explicit local constancy of $G$-pseudorepresentations}\label{secps}

We give a very obvious application of Lemma \ref{funcon} to pseudorepresentations over rigid analytic spaces (Proposition \ref{Tconst}). When applied to pseudorepresentations obtained as traces of sheaves of representations that are not residually absolutely irreducible, our result is strictly weaker than what we already proved: at the very least, the trace is insensitive to semisimplification, so that there is no hope of studying the mod $\pi^n$ variation of a sheaf of representations via pseudorepresentation-theoretic methods. On the other hand, in the residually absolutely irreducible case, we obtain a small improvement on Theorem \ref{Unconst}, in the sense that we can produce a mod $p^n$ pointwise constancy neighborhood for a sheaf of $G$-representations, as opposed to a $G$-representation (Corollary \ref{pstorep}).

We start with a simple definition. As before, let $L$ be a $p$-adic field, with valuation ring $\cO_L$ having maximal ideal $\fm_L$ and residue field $k_L$. Let $X$ be a rigid analytic space over $L$, and let $T\colon G\to\cO_X(X)$ be a continuous pseudorepresentation of a profinite group $G$.

%\andr{state special case, if A wideopen than A1 is A??}

\begin{defin}
	We say that $T$ is pointwise constant mod $p^n$ (respectively, constant mod $\pi_L^n$) if, for every $g\in G$, the function $T(g)\in\cO_X(X)$ is pointwise constant mod $p^n$ (respectively, constant mod $\pi_L^n$) in the sense of Definition \ref{lcfun}.
\end{defin}

Let $X$ be a flat, normal rigid analytic space over $L$ and $T\colon G\to\cO_X(X)$ be a pseudorepresentation. %\andr{Can one use the results of the previous section to actually say anything nontrivial about modulo $p^n$ variation???}
Assume that $X$ admits an affine formal model. Then by Lemma \ref{Tform}(ii) and Remark \ref{qsqp}(ii), there exists a flat, affine formal model $\cX=\Spf\cA$ of $X$ and a continuous pseudorepresentation $\cT\colon G\to\cO_\cX(\cX)$ that induces $T$ via the natural map $\cO_\cX(\cX)\to\cO_X(X)$. Given this observation, the next proposition is a straightforward consequence of Lemma \ref{funcon}. For every $y\in X(L)$, we denote by $\ev_y$ the specialization map of regular functions at $y$. We keep all the notation from Section \ref{secUV}.

\begin{prop}\label{Tconst}\mbox{ }
	\begin{enumerate}
		\item For every $x\in X(L)$, $T$ is pointwise constant mod $p^n$ on $U_{x,\cX}^{(n)}$. %the pseudorepresentation $\ev_y\ccirc\pi_{n,\cA}\ccirc\cT$ is independent of $y\in U_{x,\cX}^{(n)}(L)$.
		\item For every $x\in X(L)$, $T$ is constant mod $\pi_L^n$ on $V_{x,\cX}^{(n)}$. %, i.e. it takes values in the image of $\cO_L$ under the structure map $\cO_L\to\cB_{n,\cX}^{(n)}/\pi_L^n\cB_{n,\cX}^{(n)}$.
	\end{enumerate}
\end{prop}

The exact analogue of Remark \ref{UVX} about the optimality of $\cX$ applies here, but we postpone a discussion of it to Section \ref{optimalsec}.

We remark that this is compatible with the tautological fact that pseudorepresentations, or Chenevier's determinants, are residually constant over local rings (see \cite[Definition 3.12]{chendet}): via Remark \ref{wores}, Proposition \ref{Tconst} immediately gives the following.

\begin{cor}
	If $X$ is an irreducible wide open, then $T$ is pointwise constant modulo $p$. %\andr{did we define constancy?? borrow from emiliano??}
\end{cor}

%\andr{Compare--A pseudorep over local ring is res constant??? Check Chenevier, Hellmann}
%\andr{Look at Chenevier-determinants 3.12 and lecture4 for: pseudorepresentations over local rings are residually constant}
%\andr{but compare chenevier lecture 5: const over affinoid?? so it's impossible to glue pseudoreps along a sheaf?? or maybe a sheaf over admissible covering is automatically residually constant}

%\andr{adjust all valuations considering that val of uniformizer is 1/e??}

We recall an immediate consequence of a result of Carayol. %We write \tr\rho for the trace of a representation \rho. 

\begin{thm}[{cf. \cite[Théorème 1]{cartrace}}]\label{carayol}
	Let $V$ be a finite dimensional $L$-vector space and $\rho_1,\rho_2\colon G\to\GL(V)$ two residually absolutely irreducible continuous representations. Let $n\in\Z_{\ge 1}$. If the traces $\tr\rho_1(g)$ and $\tr\rho_2(g)$ are congruent modulo $\pi_L^n$ for every $g\in G$, then for every choice of lattices $\cV_1$ and $\cV_2$ for $\rho_1$ and $\rho_2$, $\cV_1\otimes_{\cO_L}\cO_L/\pi_L^n$ and $\cV_2\otimes_{\cO_L}\cO_L/\pi_L^n$ are isomorphic as $\cO_L/\pi_L^n[G]$-modules. 
\end{thm}

If $T$ is residually absolutely irreducible, then we can combine Theorem \ref{carayol} with Proposition \ref{Tconst} to deduce the following. We keep the notation of Proposition \ref{Tconst}.

\begin{cor}\label{pstorep}
	Assume that $T$ is the trace of a sheaf of $G$-representations $\bV$ over $X$, and that it is residually absolutely irreducible. Then $\bV$ is pointwise constant mod $p^n$ on $U_{x,\cX}^{(n)}$. 
\end{cor}

\begin{ex}
	We mention a simple counterexample to Carayol's theorem if $\rho$ is residually reducible: let $\rho$ be the unramified representation of $G_{\Q_p}$ mapping $\Frob_p$ to the diagonal element of eigenvalues $1+p$ and $1-p.$ Then the pseudorepresentation attached to $\rho$ is constant equal to 2 modulo $p^2$, but the reduction of $\rho$ modulo $p^2$ is not isomorphic to the identity. This corresponds to the fact that, for $n\in\Z_{\ge 1}$, one cannot detect the mod $p^n$ Frobenius eigenvalues from the mod $p^n$ reduction of its characteristic polynomial, as discussed in \cite{annghimed}.
\end{ex}

%Is it useful to lift pseudorep to family of reps??

%For pseudorep, situation even easier because it's just one function landing in A0??

\subsection{Traces of sheaves of $G$-representations}\label{trsec}

A $G$-representation over an affinoid is pointwise constant mod $p$ up to semisimplification: indeed, the semisimplification of its mod $p$ reduction is uniquely determined by the semisimplification of the mod $p$ reduction of its associated pseudorepresentation, which is constant by \cite[Lemma 2.3(iv)]{chenlec5}. We show that the notion of sheaf of $G$-representations does not allow for more flexibility in this sense. 

Let $X$ be a rigid analytic space over $L$. Let $T\colon G\to\cO_X(X)$ be a continuous pseudorepresentation, and let $\bV$ be a sheaf of $G$-representations over $X$. %Let \tr\bV\colon G\to\cO_X(X) be the trace of \bV.
We record an obvious consequence of a lemma of Chenevier.

\begin{prop}\label{traceconst}
	If $X$ is connected, then $T$ is pointwise constant modulo $p$. In particular, $\bV$ is pointwise constant mod $p$ up to semisimplification.
\end{prop}

\begin{proof}
	Let $\fU$ be an admissible affinoid covering of $X$; by splitting every element of $\fU$ into its connected components, we can assume that every $U\in\fU$ is connected. %\andr{discuss that it's actually SpecA connected??} 
	By \cite[Lemma 2.3(iv)]{chenlec5}, $T\vert_U$ is pointwise constant mod $p$ for every $U\in\fU$. Since $\fU$ is admissible, $T$ is pointwise constant mod $p$ over the whole of $X$. The result about $\bV$ follows by applying the first statement to $T=\tr\bV$.
\end{proof}

\subsection{Comparison with a result of Kedlaya and Liu}\label{klsec}

Let $K$ be a $p$-adic field. We refer to \cite[Section 2]{kedliufam} for the definition of $(\varphi,\Gamma_K)$-modules over an affinoid base $X$, and of the Berger--Colmez functor attaching such an object to a $p$-adic Galois representation over $X$. Since we only need the theory in this section, we do not go into any more details here. We briefly explain a relation to our work of the results of Kedlaya--Liu concerning $(\varphi,\Gamma_K)$-modules over local coefficient algebras.

Let $X$ be a wide open space over $L$. 
Then $\cO_X(X)$ is not a local coefficient algebra in the sense of \cite[Definition 4.1]{kedliufam}: indeed, if $X$ is written as an increasing union of affinoids $\Spm A_i, i\ge 0$, with compact restriction maps $A_{i+1}\to A_i$, $\cO_X(X)$ is the projective limit of the Banach algebras $A_i$, hence a Fréchet $\Q_p$-algebra, but not a Banach $\Q_p$-algebra itself. %\andr{Say why?}
This corresponds to the fact that the Berthelot generic fiber of $\Spf\cO^+_X(X)$ is not $\Spm(\cO^+_X(X)\otimes_{\Z_p}\Q_p)$. We record a standard example of this discrepancy.

\begin{ex}\label{exlog}
If one starts with the ring $\Z_p\langle T\rangle$ of analytic functions bounded by 1 on the closed unit disc, completes it at the ideal $(T)$ and tensors with $\Q_p$ over $\Z_p$ (as indicated in \cite[Definition 4.1]{kedliufam} as a way to produce a local coefficient algebra) one obtains $\Q_p\otimes_{\Z_p}\Z_p[[T]]$, which is strictly smaller than the ring of analytic functions on the open unit disc: for instance, $\log(1+T)$ is an analytic function on the open disc whose expansion as a power series in $T$ contains unbounded denominators, hence it does not belong to $\Q_p\otimes_{\Z_p}\Z_p[[T]]$.
\end{ex}

More generally, we can apply the recipe of \emph{loc. cit.} as follows: if $X=\Spm A$ is an $L$-affinoid, $x\in X(\Qp)$ and $\cX$ a formal model of $X$, completing $\cO_\cX(\cX)$ at the ideal $I_x$ defined in Section \ref{secUV} and tensoring with $\Q_p$ over $\Z_p$ produces a local coefficient algebra $A_x$ (i.e., the algebra $\cA_{x}^{(1)}\otimes_{\Z_p}\Q_p$, with the notation of Section \ref{secUV}). By the previous remarks, $\cA_x$ is strictly smaller than the ring of analytic functions on the wide open $U_{x,\cX}^{(1)}$. However, if one starts with an étale $(\varphi,\Gamma_K)$-module $D$ over an affinoid $X$, one can produce a $(\varphi,\Gamma_K)$-module $D_x$ over $\cA_x$ simply by tensoring through the natural map $A\to\cA_x$. By applying \cite[Thoerem 0.1]{kedliufam}, one can convert $D_x$ into an $A_x$-linear $G_K$-representation, hence a $G_K$-representation over $U_{x,\cX}^{(1)}$ after extending scalars to $\cO_X(U_{x,\cX}^{(1)})$. Therefore we have the following result, where we keep the notation above (so that, crucially, $X$ is affinoid).

\begin{thm}[Kedlaya--Liu]\label{klth}
	For every $x\in X(\Qp)$, the $(\varphi,\Gamma_K)$-module $D\vert_{U_{x,\cX}^{(1)}}$ is attached to a $G_K$-representation over $U_{x,\cX}^{(1)}$. In particular, $X$ admits a (not necessarily admissible) covering $\fU$ such that, for every $U\in\fU$, $D\vert_U$ is attached to a $G_K$-representation over $U$.
\end{thm}

Note that the fact that the mod $p$ reduction is not necessarily constant along a family of étale $(\varphi,\Gamma_K)$-modules is exploited in the recent breakthrough \cite{egstack}.

The following is an immediate corollary of Proposition \ref{traceconst}. Note that Kedlaya and Liu hint to the fact that a $(\varphi,\Gamma_K)$-module over a rigid analytic space can only come from a representation of $G_{K}$ if it is residually constant up to semisimplification \cite[Introduction]{kedliufam}; this can be easily deduced from the aforementioned lemma of Chenevier. We give a slightly stronger version of the statement, where we replace a $G_{K}$-representation with a sheaf. %The authors do not seem to prove this claim in \emph{loc. cit.}. \andr{OK but already follows from chenevier?}

\begin{cor}\label{phiG}
	Let $D$ be a sheaf of $(\varphi,\Gamma_K)$-modules on $X$. For every $x\in X(\Qp)$, let $\ovl\rho_x$ be the mod $p$ $G_K$-representation attached to $D_x$. If $D$ is attached to a sheaf $\bV$ of $G_K$-representations, then the semisimplification of $\bV_x$ is independent of $x\in X(\Qp)$.
\end{cor}

\begin{rem}
	By Kedlaya--Liu's Theorem \ref{klth}, one can find for every $x\in X(\Qp)$ an open neighborhood $U$ of $x$ such that $D\vert_U$ is attached to a $G_K$-representation $\bV_U$ on $U$. By Proposition \ref{traceconst}, $\bV_U$ is constant mod $p$ up to semisimplification; write $\ovl\bV_U$ for its mod $p$ reduction. If the resulting covering of $X$ is admissible, then $\ovl\bV_U$ is independent of $U$: indeed, we can glue the $\bV_U$ to a sheaf of $G_K$-representations, and apply Corollary \ref{phiG}.
\end{rem}

\begin{ex}\label{chenex}
	We show what Corollary \ref{phiG} gives for a standard example due to Chenevier \cite[Remarque 4.2.10]{bercol}. Let $X=\Spm A$ be a $\Q_p$-affinoid, and write $A^+$ for the subring of power-bounded elements of $A$. Let $T$ be an invertible element of $A^+$. 
	Let $D$ be the rank 1 overconvergent $(\varphi,\Gamma)$-module over $X$ defined, in a basis $e$, by $\varphi(e)=Te$ and $\gamma(e)=e$ for every $\gamma\in\Gamma$. Then:
	\begin{itemize}
 \item $D$ is isomorphic to the $(\varphi,\Gamma)$-module $\cR_X(\delta)$ attached, in the standard way, to the character $\delta\colon\Q_p^\times\to A^\times$ mapping $p$ to $T$ and defined as the identity on $\Z_p^\times$; it is étale, since $T$ is a unit in $A^+$. In particular, whenever we specialize $D$ at a point $x\in X$, we obtain the $(\varphi,\Gamma)$-module attached to the character $G_{\Q_p}^\times\to\Q_p^\times$ obtained from $\delta$ via local class field theory.
		\item $D$ is attached to a character $\chi\colon G_{\Q_p}\to A^\times$ if and only if the image of $T$ under $A^+\to A^+/p$ is constant, i.e. belongs to $\ovl\F_p^\times$. 
  Indeed, $\chi$ has to be an unramified character mapping $\Frob_p$, a Frobenius at $p$, to $T$. In order for $\chi$ to exist, one must be able to define $\Frob_p^{n}$ for every $n\in\wht\Z$, for which it is necessary and sufficient that $T^{p^k}-1$ converges to 0 (in the $p$-adic topology of $A$) if $k\in\Z$ goes to $\infty$. This happens if and only if $T\in A^\circ$ is mapped to an element of $\ovl\F_p$ under $A^+\to A^+/p$. \\ %This is obviously not the case if T is transcendental and of norm 1, because the map A^\circ\to A^\circ/p
  Note that, since $A^+$ is equipped with the $p$-adic topology, the above condition on $T$ is equivalent to $T$ being of the form $u+g$ for some $u\in\Zp^\times$ and $g$ a topologically nilpotent element. 
		As a simple example of a case when $D$ is not attached to a Galois character, one can take $A=\Q_p\langle U,V\rangle/(UV-1)$ and $T=U$, so that $X$ is the annulus of inner and outer radius 1. If one were to take for instance $T=1+pU$, one would indeed be able to convert $D$ into a Galois character over the whole $X$.
		
		\item We can find a \emph{not necessarily admissible} open covering $\fU$ of $X$ such that $D\vert_U$ is attached to a character of $G_{\Q_p}$ for every $U\in\fU$. 
		With the notations of Section \ref{secUV}, if $x$ is a point of $\Spm A$, the residue subdomain $U_{x,\Spf\cA}^{(1)}=(\Spf\cA_x^{(1)})^\rig$ is a wide open, and the topology on $\cA_x^{(1)}$ is the $(p,I_x)$-adic one. For $T$ as above, we can write $T=T(x)+(T-T(x))$, where $T-T(x)\in I_x$ and $T(x)$ is an element of $\ovl\Z_p^\times$; in particular, $T^n$ makes sense in the $(p,I_x)$-adic topology of $\cA_x^{(1)}$ for every $n\in\wht\Z$, so that $D\vert_(U_{x,\Spf\cA}^{(1)})$ is attached to the unramified character of $G_{\Q_p}$ mapping $\Frob_p$ to $T$. \\
        When $x$ varies over $X(\Qp)$, $\Spm A$ is covered by the wide open subdomains $U_{x,\Spf\cA}^{(1)}$, but such a covering is not admissible in general: if it were, one would be able to glue the characters defined over each subdomain to a global character, which we know to be impossible by the above considerations.
        
        \item Kedlaya and Liu generalize the above picture to an overconvergent, étale $(\varphi,\Gamma)$-module $D$ of arbitrary rank by choosing a unit $u\in\Zp^\times$ and working over the locus $X_u\subset X$ on which the matrix attached to $\varphi-u$ in some basis has all of its coefficients of positive $p$-adic valuation (i.e. $\varphi$ reduces to $u$ mod $p$ in $A^+/p$). An argument along the lines of the above shows that $D\vert_{X_0}$ is attached to a $G_{\Q_p}$-representation over $X$. As in the previous example, it is not possible to glue the resulting $G_{\Q_p}$-representations into a sheaf as $u$ varies.
	\end{itemize}
	%By \ref{??}, the image of any element of A is constant (i.e. an element of \ovl\F_p) in \cA_x^{(1)}	
	Note that the above is also an explicit example of a case when the mod $p$ local constancy covering defined in \cite[Definition 2.3]{tortilatt} cannot be taken to be admissible.
\end{ex}

%Remark \ref{chenex} gives one reason why, in Theorem \ref{wideopenlatt}, we looked at certain conditions on the underlying rigid space X, and on the representation it carries, to guarantee the existence of a lattice.

%\begin{rem}%\andr{Does the remark make sense?? Do Kedlaya--Liu actually discuss lattices or just freeness?? ONLY freeness-comment something?? use pseudorep to get freeness??} 
%	\andr{Make better comment about \cite[Theorem 0.1]{kedliufam} and the discussion thereafter; where to put it?}

%if it were it would have to be of the form R\otimes_{\Z_p}\Q_p for some local, complete, Noetherian \cO_L-algebra R, and one would certainly be able to take as such an R the ring \cO_X^\circ(X) WHY???
%However, \cO_X(X) is strictly larger than \cO_X^\circ(X)\otimes_{\Z_p}\Q_p, since its elements will have bounded norm on every affinoid contained in X, but not necessarily on X itself. A standard example is that of the function \log(1+T) on the wide open disc D^\circ(0,1) with coordinate T: having unbounded denominators, it does not belong to \Z_p[[T]]\otimes_{\Z_p}\Q_p (even if it belongs to L\langle p^{-r}T\rangle for every r\in\Q_{>0}, after extension of scalars to an L over which such a Tate algebra is defined).

%\end{rem}

\subsection{Optimal constancy neighborhoods}\label{optimalsec}%{Constancy modulo $p^n$ of universal (pseudo-)representations}

%We show that the constancy neighborhoods of Proposition \ref{Unconst} are optimal for universal (pseudo-)representations,
In this section we expand on Remark \ref{UVX}, by studying under what conditions the neighborhoods $U_{x,\cX}^{(n)}$ and $V_{x,\cX}^{(n)}$ give the exact locus on which a (pseudo-)representation over an affine formal scheme $\cX$ is constant. 

\begin{notation}
	Let $\rho\colon G\to\GL_d(\cO_\cX(\cX))$ be a continuous representation. We denote by $\cO_L[\rho(G)]$ the closure in $\cO_\cX(\cX)$ of the  $\cO_L$-subalgebra generated by the coefficients of $\rho(g)$, when $g$ varies over $G$.
	
	Let $T\colon G\to\cO_\cX(\cX)$ be a continuous pseudorepresentation. We denote by $\cO_L[T(G)]$ the closure in $\cO_\cX(\cX)$ of the $\cO_L$-subalgebra generated by the elements $T(g)$, when $g$ varies over $G$.
\end{notation}

For the rest of the section, assume that $X$ is a rigid space admitting an affine formal model. 

\begin{prop}\label{optimal}
	Let $n\in\Z_{\ge 1}$.
	\begin{enumerate}[label=(\roman*)]
		\item Let $\bV$ be a $G$-representation over $X$. Assume that $\bV$ admits a model $\rho\colon G\to\GL_d(\cO_\cX(\cX))$, defined over a formal model $\cX$ of $X$, such that $\cO_L[\rho(G)]=\cO_\cX(\cX)$, and let $U$ be an open subspace of $X$ such that $\rho$ is constant mod $\pi_L^n$ over $U$. Then $U$ is contained in $U_{x,\cX_{\ovl\rho}}^{(n)}$.
		\item Let $T\colon G\to\cO_X(X)$ be a continuous pseudorepresentation. Assume that the image of $T$ lies in $\cO_\cX(\cX)$ for a formal model $\cX$ of $X$, and that $\cO_L[T(G)]=\cO_\cX(\cX)$. Let $U$ be an open subspace of $X$ such that $T$ is constant mod $\pi_L^n$ over $U$. Then $U$ is contained in $V_{x,\cX}^{(n)}$.
	\end{enumerate}
\end{prop}

\begin{proof}
	We prove (i). Since $\rho$ is constant over $U$, up to conjugating $\rho$, for every $g\in G$ all of the matrix coefficients of $\rho(g)$ are constant mod $\pi_L^n$ over $U$. Therefore, the same is true for every element of $\cO_L[\rho(G)]=\cO_\cX(\cX)$. By definition, $U_{x,\cX_{\ovl\rho}}^{(n)}$ is the largest subspace of $X$ over which every element of $\cO_\cX(\cX)$ is constant mod $\pi_L^n$, so the conclusion follows. The proof of (ii) goes exactly in the same way.
\end{proof}

\begin{rem}\label{optimalrem}\mbox{ }
\begin{enumerate}[label=(\roman*)]
\item In the setting of Proposition \ref{optimal}(i) (respectively, (ii)), if a formal model $\cX$ with the property $\cO_L[\rho(G)]=\cO_\cX(\cX)$ (respectively, $\cO_L[T(G)]=\cO_\cX(\cX)$), exists, then it is optimal for the study of the mod $p^n$ variation of $\rho$ (respectively, $T$), since starting our construction with $\cX$ provides us with the largest possible neighborhoods on which $\rho$ is (pointwise) constant mod $p^n$. 
As we discuss below, these equalities are satisfied for the (pseudo-)representations carried by universal (pseudo-)deformation rings (see \Cref{Rgen}).
\item Again in the setting of Proposition \ref{optimal}(i), if a formal model $\cX$ of $X$ such that $\cO_L[\rho(G)]=\cO_\cX(\cX)$ does not exist, one can replace $X$ altogether by the rigid generic fiber $Y$ of $\cY\coloneqq\Spf\cO_L[\rho(G)]$, where $\cO_L[\rho(G)]$ is an adic ring with the subspace topology inherited from $\cO_\cX(\cX)$. The adic map $\cO_L[\rho(G)]\subset\cO_\cX(\cX)$ induces a map of formal schemes $\cX\to\cY$ and of generic fibers $X\to Y$ (note that, if the mod $\pi_L$ reduction of $\rho$ is constant over $\cX$, then $\cY$ will be initial among the affine formal schemes $\cZ$ for which a factorization $\cX\to \cZ\to \Spf\cR_{\ovl\rho}^\square$ exists). One can then study the local constancy of the representation $\rho$ over $Y$ and pull the results back along $X\to Y$. A similar discussion can be made in the setting of Proposition \ref{optimal}(ii).
\item The ring $\cO_L[T(G)]$ is what is called the \emph{trace algebra} of $T$ in \cite[Section 2.1.2]{conlanmedbig}, and the condition $\cO_L[T(G)]=\cO_\cX(\cX)$ is one of the conditions defining the \emph{admissibility} of a pseudorepresentation in \cite[Section 5.2]{bellim} (cf. \cite[Definition 2.8]{conlanmedbig}). %\andr{can one check this condition in some arithmetic applications?? using determinant?? actually if det is not congruent mod pn the representation cannot be, so if for instance det is 1+T everything is very simple}
\item In the setting of Proposition \ref{optimal}(ii), if $\cX$ is affine then by \cite[Lemma 2.3(ii)]{chenlec5}, $\cO_L[T(G)]$ is a semilocal profinite ring, local if $X$ is connected. In particular, the condition $\cO_L[T(G)]=\cO_\cX(\cX)$ can only be satisfied if $\cO_\cX(\cX)$ is profinite, which in turn can only happen if $X$ is a wide open.
\item Let $\cX=\Spf A$ be an affine formal scheme, and $\rho\colon G\to\GL_d(A)$ a continuous representation such that $\cO_L[\rho(G)]=A$. For every surjective $\cO_L$-algebra homomorphism $\pi\colon A\to B$, the representation $\pi\ccirc\rho\colon G\to\GL_d(B)$ (i.e., the pullback of $\rho$ to a $G$-representation on $\Spf B$) satisfies $\cO_L[\pi\ccirc\rho(G)]=\pi(\cO_L[\rho(G)])=B$. In particular, the conclusion of Proposition \ref{optimal}(i) also applies to $\pi\ccirc\rho$. The same remark holds if we replace $\rho$ with a pseudorepresentation. \\
One could, for instance, apply our results to each factor of the semistable deformation rings described by Breuil and Mézard in \cite[Theorem 5.3.1]{bremez}, to obtain optimal subspaces on which the mod $p^n$ reduction of semistable representations is constant (for weight at most $p-1$).
\end{enumerate}
\end{rem}

\begin{ex}
	Let $\rho\colon G_{\Q_p}\to\Z_p[[T]]^\times$ be the unramified character mapping any lift $\Frob_p$ of the Frobenius to $1+T$, and let $\bV$ be the associated $G_{\Q_p}$-representation over the generic fiber $D^\circ(0,1)$ of $\Spf\Z_p[[T]]$, the wide open unit disc. Then $\Z_p[\rho(G_{\Q_p})]=\Z_p[[T]]$, so that the constancy neighborhoods $U_{x,\cX_{\ovl\rho}}^{(n)}$ and $V_{x,\cX_{\ovl\rho}}^{(n)}$ are optimal by Proposition \ref{optimal}(i).
	
	On the other hand, if we consider the unramified character $\wtl\rho\colon G_{\Q_p}\to\Z_p[[T]]^\times$ mapping $\Frob_p$ to $1+pT$, then $\Z_p[\rho(G_{\Q_p})]=\Z_p[[1+pT]]$ is strictly smaller than $\Z_p[[T]]$ (clearly, it does not contain $T$), and one sees immediately that the constancy neighborhoods we defined are not optimal: for instance, $U_{x,\cX_{\ovl\rho}}^{(1)}$ is the affinoid disc of radius $p^{-1}$, but $\wtl\rho$ is constant mod $p$ over the whole $\Spf\Z_p[[T]]$; $V_{x,\cX_{\ovl\rho}}^{(2)}$ is the wide open disc of radius $p^{-1}$, but $\wtl\rho$ is actually pointwise constant mod $p^2$ over the whole $\Spf\Z_p[[T]]$. 
\end{ex}

We show that, for universal (pseudo-)representations, the obvious formal models are optimal in the sense of Remark \ref{optimalrem} above. %Keep all the notations introduced. 
Let $V$ be a $k_L$-vector space of finite dimension $d$ and $\ovl\rho\colon G\to\GL(V)$ a continuous representation. %Let \F be a finite field of characteristic p.

As before, we denote by $X_{\ovl\rho}^\sq$ the universal pseudodeformation space for $\ovl\rho$, equipped with its model $\cX_{\ovl\rho}^\sq=\Spf\cR_{\ovl\rho}^\sq$ and with the universal deformation $\rho^{\sq,\univ}$. We write $X^\ps_{\ovl\rho}$ for the universal pseudodeformation space for $\tr\ovl\rho$, equipped with its model $\cX_{\ovl\rho}^\ps=\Spf\cR^\ps_{\ovl\rho}$ and with the universal pseudodeformation $T^\univ$.
Finally, if $\End_{k_L[G]}=k_L$, we denote by $X_{\ovl\rho}$ the universal pseudodeformation space for $\ovl\rho$, equipped with its model $\cX_{\ovl\rho}=\Spf\cR_{\ovl\rho}$ and with the universal deformation $\rho^\univ$.

The following lemma is key to the proof of the constancy statements in Proposition \ref{optimal}.

\begin{lemma}\label{Rgen}\mbox{ }
\begin{enumerate}[label=(\roman*)]
\item Assume that $\End_{k_L[G]}=k_L$. Pick any basis for the finite free $\cR_{\ovl\rho}$ module on which $\rho^\univ$ acts, and let $R$ be the $\cO_L$-subalgebra of $\cR_{\ovl\rho}$ generated by the coefficients of $\rho^\univ(g)$ in this basis, when $g$ varies over $G$. Then $R=\cR_{\ovl\rho}$.
\item Let $R$ be the $\cO_L$-subalgebra of $\cR_{\ovl\rho}$ generated by the coefficients of $\rho^{\sq,\univ}(g)\in\GL_d(R^\sq_{\ovl\rho})$ when $g$ varies over $G$. Then $R=\cR_{\ovl\rho}$. 
\item Let $R$ be the $\cO_L$-subalgebra of $\cR^\ps_{\ovl\rho}$ generated by the elements $T(g)$ when $g$ varies over $G$. Then $R=\cR^\ps_{\ovl\rho}$.
\end{enumerate}
\end{lemma}

\begin{proof}
We prove (i). Clearly, $\rho^{\univ}$ can be defined over $R$, i.e. there exists a continuous representation $\rho_R\colon G\to\GL_d(R)$ such that $\rho^\univ\cong\rho_R\otimes_R\cR_{\ovl\rho}$. Then, the pair $(R,\rho_R)$ is a universal deformation of $\ovl\rho$, hence the inclusion $R\to\cR_{\ovl\rho}$ is an isomorphism by universality. The proof of (ii) and (iii) is exactly the same.
\end{proof}

\begin{prop}\label{optimaluniv}
Let $n\in\Z_{\ge 1}$. 
\begin{enumerate}[label=(\roman*)]
\item Assume that $\End_{k_L[G]}=k_L$. %Let X_{\ovl\rho} be the universal pseudodeformation space for \ovl\rho, equipped with its model \cX_{\ovlrho}=\Spf\cR_{\ovl\rho} and with the universal deformation \rho^\univ. 
If $\rho^\univ$ is pointwise constant mod $p^n$ (respectively, constant mod $\pi_L^n$) over an open subspace $U$ of $X_{\ovl\rho}$, then $U$ is contained in $U_{x,\cX_{\ovl\rho}}^{(n)}$ (respectively, $V_{x,\cX_{\ovl\rho}}^{(n)})$ for some $x\in X_{\ovl\rho}(\Qp)$.
\item If $\rho^{\sq,\univ}$ is pointwise constant mod $p^n$ (respectively, constant mod $\pi_L^n$) over an open subspace $U$ of $X_{\ovl\rho}^\sq$, then $U$ is contained in $U_{x,\cX^\sq_{\ovl\rho}}^{(n)}$ (respectively, $V_{x,\cX^\sq_{\ovl\rho}}^{(n)})$ for some $x\in X^\sq_{\ovl\rho}(\Qp)$.
\item If $T^\univ$ is pointwise constant mod $p^n$ (respectively, constant mod $\pi_L^n$) on an open subspace $U$ of $X^\ps_{\ovl\rho}$, then $U$ is contained in $U_{x,\cX_{\ovl\rho}^\ps}^{(n)}$ (respectively, $V_{x,\cX^\ps_{\ovl\rho}}^{(n)})$ for some $x\in X^\ps_{\ovl\rho}(\Qp)$.
\end{enumerate}
\end{prop}

\begin{proof}
We prove the pointwise constancy part of (i). By the universal property of $\cR_{\ovl\rho}$, the universal deformation of $\ovl\rho$ over $\cO_L/\pi_L^n$-algebras is 
\[ \rho^\univ\otimes_{\cO_L}\cO_L/\pi_L^n\colon G\to\cR_{\ovl\rho}\otimes_{\cO_L}\cO_L/\pi_L^n. \]
Let $U$ be an open subspace of $X_{\ovl\rho}$ such that $\rho^\univ$ is pointwise constant mod $p^n$ over U, and let $x\colon\cR_{\ovl\rho}\otimes_{\cO_L}\cO_L/\pi_L^n\to\cO_L/\pi_L^n$ be the point of $\cR_{\ovl\rho}\otimes_{\cO_L}\cO_L/\pi_L^n$ corresponding to the common mod $p^n$ reduction of all the points in $U$.

Let $E$ is a finite extension of $L$ with valuation ring $\cO_E$ and uniformizer $\pi_E$. For every $y\in U(E)$, let $\wtl y\colon\cR_{\ovl\rho}\to\cO_E$ be the rig-point of $\Spf\cR_{\ovl\rho}$ attached to $y$. Then the composition of $y$ with $\cO\to\cO/\pi_E^n$ factors through $x$. By Remark \ref{UVchar}, $U_{x,\cX_{\ovl\rho}}^{(n)}$ is precisely the locus of points of $X_{\ovl\rho}$ with this property. %is the associated rig-point of  \cO is the valuation ring of a finite extension E/L, and \red_n\colon\cR_{\ovl\rho}\to\cO is a rig-point whose composition with \cO\to\cO/\pi_E^n factors through x, 

The proofs of (ii) and (iii) in the case of pointwise constancy go exactly as that for (i).

For the constancy statements in all of (i,ii,iii), we simply combine Proposition \ref{optimal} and Lemma \ref{Rgen}.
%Now assume instead that U is an open subspace of \ovl\rho such that \rho^\univ is constant mod p^n over U. Then the mod p^n reduction of the restriction of \rho^\univ to U is isomorphic to \rho_n\otimes_{\cO_L/\pi_L^n}\cO_U(U) for a continuous representation \rho_n\colon G\to\GL(V), V a free \cO_L/\pi_L^n-module of rank d. Let x be the point of 
%Let A be an \cO_L/\pi_L^n-algebra. By the universal property of \cR_{\ovl\rho}, deformations of \ovl\rho to A are in bijection with maps \cR_{\ovl\rho}\to A, and every such map factors uniquely through a map
%\[ \cR_{\ovl\rho}^\ps\otimes_{\cO_L}\cO_L/\pi_L^n\to A. \]
\end{proof}

%THIS HAS TO DO WITH OPTIMAL MODEL!!! also compare CHENEVIER???

%\andr{USE bellchen prop 1.6.4!!??}

\medskip

\section{Arithmetic applications}\label{sec:arith}

We apply Corollary \ref{woloc} to various specific families of representations of arithmetic interest. In doing so, we refine some known results about the local constancy modulo powers of $p$ of families of crystalline or semistable local Galois representations (Proposition \ref{appcr}, Corollary \ref{cryscor}, Remarks \ref{tortirem} and \ref{rember}), and deduce some new ones both for local and for global modular representations (Corollary \ref{crysglob}, Proposition \ref{appst}, Theorem \ref{bll}, Corollary \ref{stglob}, Proposition \ref{ecurvelc}, Corollary \ref{corhida}).

We assume throughout this section that $p>2$: since we are dealing with sheaves of representations of rank $d=2$, the assumption $d<p$ of Corollary \ref{woloc} is satisfied.

%\andr{APPLY to eigencurve???? problem is that we need explicit shape for it??? oh, do the annulus calculation and show independence of radius, comment independence of radius in general???}

\subsection{Sheaves of trianguline $G_{\Q_p}$-representations}\label{shtri}
	
%We begin this section by recalling the notions of trianguline, regular and rigidified $(\varphi, \Gamma)$-modules. Our notation is analogous to that of Chenevier \cite{Che13} and Hellmann \cite{He16}, even though their families of representations are defined over adic rather than rigid analytic spaces.
	
	%Let $d$ be a positive integer.
	
	Following \cite[Section 2.2]{brehelsch}, we recall two constructions of a deformation space for trianguline representations of $G_{\Q_p}$, a ``functorial'' and a ``geometric'' one. We will use the interplay between the two constructions in order to produce a sheaf of trianguline representations of $G_{\Q_p}$. Note that all of these section can be rewritten, with very little change, with $G_{\Q_p}$ replaced by $G_K$, $K$ a $p$-adic field. However, we confine ourselves to the case $K=\Q_p$ in our applications.
	
	Let $\Aff_{\mathbb{Q}_p}$ be the category of affinoid algebras over $\mathbb{Q}_p$. For every $A\in\Aff_{\Q_p}$, we denote by $\mathcal{R}_A$ the Robba ring with coefficients in $A$, defined as in \cite[Notation 2.1.1]{kedpotxia}.

	%Let $\mathcal{T}$ be the $\mathbb{Q}_p$-rigid analytic space parametrizing the continuous characters of $\mathbb{Q}_p^{\times}$. 
	Recall that the functor $\Aff_{\Q_p}\to\Sets, \,
		B\mapsto\Hom_\cont(\mathbb{Q}_p^\times,B^{\times})$ 
	%\begin{align*} \Aff_{\Q_p}&\to\Sets \\
	%	B&\mapsto\Hom_\cont(\mathbb{Q}_p^\times,B^{\times}) \end{align*}
	is pro-represented by a $\Q_p$-rigid analytic space $\cT$. Similarly, the functor $\Aff_{\Q_p}\to\Sets,\,B\mapsto\Hom_\cont(\mathbb{Z}_p^\times,B^{\times})$
%	\begin{align*} \Aff_{\Q_p}&\to\Sets \\
%		B&\mapsto\Hom_\cont(\mathbb{Z}_p^\times,B^{\times}) \end{align*}
	is pro-represented by a $\Q_p$-rigid analytic space $\cW$, which is isomorphic to a disjoint union of $p-1$ wide open rigid analytic 1-discs. 
	From the isomorphism $\Q_p^\times\cong\Z_p^\times\times p^\Z$, we deduce an isomorphism 
	$\mathcal{T}\cong\mathcal{W}\times \mathbb{G}_m$ 
	of $\mathbb{Q}_p$-rigid analytic spaces. %Here $\mathcal{W}$ denotes the space characterizing continuous characters of $\mathbb{Z}_p^\times$ and itself consists of a finite disjoint union of open rigid analytic balls. 
	We introduce notation for a few special continuous characters $\Q_p^\times\to\Q_p^\times$:
	\begin{itemize}
		\item $x$ is the identity character, 
		\item $\chi$ is the character attached by local class field theory to the cyclotomic character of $G_{\Q_p}$, that satisfies $\chi(p)=1$ and is the identity on $\mathbb{Z}_p^\times$;
		\item $\lvert x\rvert$ is the character mapping $y\in\Q_p^\times$ to $p^{- v_p (y)}$.
	\end{itemize}
	Clearly, $\chi=x\lvert x\rvert$. We will often abuse of notation and write $x$ for both the identity character and an element of $\Q_p$, so that $\lvert x\rvert=p^{-v_p(x)}$.
	
	Inside of the rigid analytic space $\mathcal{T}$, we identify an admissible open $\mathcal{T}^{\reg}$ whose points are \emph{regular} in the sense of Colmez and Chenevier (see for example \cite[Section 2.27]{chendens}). 
	For any $B\in\Aff_{\mathbb{Q}_p}$, we say that a character $\delta\colon\mathbb{Q}_p^{\times} \rightarrow B^{\times}$, seen as a point in $\mathcal{T}(B)$, is regular if, for every $z\in\Spm B$ with residue field a finite extension $K(z)/\Q_p$, the specialization $\delta_z\colon\mathbb{Q}_p^\times \rightarrow K(z)^\times$ of $\delta$ at $z$ is not of the form $x^i$ or $\chi x^{-i}$ for any non-negative integer $i$.
	
	When studying trianguline representations of any dimension $d\geq 1$, we will work over the rigid analytic space of parameters $\mathcal{T}^d$, the product of $d$ copies of $\mathcal{T}$. We denote by $\mathcal{T}_d^\reg$ the admissible open of $\mathcal{T}^d$ defined by the following condition: if $B\in \Aff_{\mathbb{Q}_p}$,
	\[ \mathcal{T}_d^\reg (B) \coloneqq\{(\delta_i)_i \in \mathcal{T}^d (B) \colon \delta_i / \delta_j \in \mathcal{T}(A)^\reg \text{ for all }1\leq i <j \leq d\}. \] 
	
%	\andr{what can be generalized to G=GK?? mention it should be possible?? use nakamura??}
	
	We recall a crucial definition. %(we will keep as much as possible the notation of Chenevier, see \cite{Che13}):
	Let $A\in \Aff_{\mathbb{Q}_p}$ and $d\in\Z_{\ge 1}$. We refer to \cite[Definition 2.2.12]{kedpotxia} for the definition of a $(\varphi,\Gamma)$-module over $\cR_A$.
	Recall from \cite[Theorem 6.2.14]{kedpotxia} that every $(\varphi,\Gamma)$-module of rank 1 over $\cR_A$ is isomorphic to $\cR_A(\delta)$ for some continuous character $\delta\colon\Q_p^\times\to A^\times$. %\andr{indeed, the line bundle from KPX is trivial over an affinoid??}
	
	\begin{defin}
		A \emph{trianguline} $(\varphi, \Gamma)$-module of rank $d$ over $\cR_A$ is a pair $(D,\Fil^\bullet D)$ consisting of a $(\varphi, \Gamma)$-module over $\cR_A$ and an increasing, exhaustive, separated filtration $(\Fil^\bullet D)_{i=1,\ldots,d}$ such that, for every $i\in\{0,\ldots,d-1\}$, $\Fil^{i+1}/\Fil^i$ is a $(\varphi,\Gamma)$-module of rank 1 over $\cR_A$, hence isomorphic to $\cR_A(\delta_i)$ for a continuous character $\delta_i\colon\Q_p^\times\to A^\times$. We call the characters $(\delta_i)_{i=0,\ldots,d-1}$ the \emph{parameters} of $(D,\Fil^\bullet D)$. \\
		A \emph{regular, trianguline, rigidified} $(\varphi, \Gamma)$-module of rank $d$ over $\mathcal{R}_A$ is a triple $(D, \Fil^\bullet D, \nu)$ where:
		\begin{enumerate}[label=(\arabic*)]
			% \item $D$ is a trianguline $(\varphi, \Gamma)$-module over $\mathcal{R}_A$ of rank $d$;
			\item $(D, \Fil^\bullet D)$ is a trianguline $(\varphi,\Gamma)$-module of rank $d$ of whose parameters $(\delta_i)_{i=1,\ldots,d}$ define a point of $\mathcal{T}_d^{\reg} (A)$;
			\item $\nu=(\nu_i)_{i=0,\ldots,d-1}$ is a family of isomorphisms $\nu_i\colon \Fil^{i+1} (D) / \Fil^{i} (D) \rightarrow \mathcal{R}_A (\delta_i)$ of $(\varphi, \Gamma)$-modules. %where $i=0, \dots \text{rank}_{\mathcal{R}_A} (D)-1$.
		\end{enumerate}
		We also refer to the datum $\nu$ as a \emph{rigidification} of the regular trianguline $(\varphi,\Gamma)$-module $(D,\Fil^\bullet D)$. \\
		We say that two regular, trianguline, rigidified $(\varphi,\Gamma)$-modules $D_1$ and $D_2$ are isomorphic if there is an isomorphism of $(\varphi,\Gamma)$-modules $D_1\to D_2$ compatible with the filtrations and commuting with the rigidifications.
	\end{defin}
	
	% The data (D,\Fil_\bullet) defines a regular trianguline $(\varphi,\Gamma)$-module, while the \nu_i give a ``rigidification'' of it.
	
\noindent	We introduce a functor $F_d\colon\Aff_{\mathbb{Q}_p} \rightarrow \Sets$ by defining $F_d(A)$ as the set of trianguline, regular and rigidified $(\varphi, \Gamma)$-modules of rank $d$ up to isomorphism, for every $A\in\Aff_{\mathbb{Q}_p}$. We recall the following result:
	
	\begin{thm}[{\cite[Theorem B]{chendens}}]
		The functor $F_d$ is representable by a rigid analytic space $S_d$ over $\mathbb{Q}_p$ which is smooth and irreducible of dimension $\frac{d(d+3)}{2}$. 
	\end{thm}

In \cite{chendens}, as well as in \cite[Theorem 2.4]{helsch}, such a functor is denoted $F_d^\square$, and the representing space $S_d^\square$, to emphasize the choice of a rigidification. We stick instead to the notation of \cite{brehelsch}, where, as we recall below, the notation ${}^\sq$ is reserved to a covering of a subspace of $S_d^\sq$ trivializing a vector bundle of $G_{\Q_p}$-representations. 

	As part of the universality statement, $S_d$ carries a universal trianguline, regular and rigidified $(\varphi,\Gamma)$-module $D_d^\univ$. For every point $x$ of $S_d^\square$ of residue field $K(x)$, we denote by $D_x$ the $(\varphi,\Gamma)$-module over $\cR_{K(x)}$ obtained by specializing $D_d^\univ$ at $x$. For every affinoid subdomain $U=\Spm A$ of $S_d^\univ$, we denote by $D_U$ the $(\varphi,\Gamma)$-module over $\cR_A$ obtained by pulling back $D_d^\univ$ along the embedding $U\into S_d^\square$.
	
%	\andr{extend scalars to L above??--I guess an L-affinoid is also a Qp-affinoid, as for KPX}

Consider a continuous representation $\ovl\rho\colon G_{\Q_p}\to\GL_n(k_L)$, and implicitly base change the space $S_d^\sq$ to $L$. 
Following the proof of \cite[Theorem 2.6]{brehelsch}, we introduce a functor $F_{d,\ovl\rho}^\square\colon\Aff_{\mathbb{Q}_p} \rightarrow \Sets$ associating with $A$ the set of triples $(\rho, \Fil^\bullet D_\rig(\rho), \nu)$ where $\rho\colon G_{\Q_p}\to\GL_d(A)$ is a continuous representation, $D_\rig(\rho)$ is the $(\varphi, \Gamma)$-module over $\mathcal{R}_A$ defined by Berger and Colmez in \cite[Théorème A]{bellchen}, $\Fil^\bullet$ is a triangulation of $D_\rig(\rho)$, and $\nu$ a rigidification of it, such that the parameter of $\Fil^\bullet$ is a triangulation of $D_\rig(\rho)$ belongs to $\cT^\reg(A)$. By the proof of \cite[Theorem 2.6]{brehelsch}, 
$F_{d,\ovl\rho}^\square$ is representable by a rigid analytic space $S_{d,\ovl\rho}^\square$ (denoted by $S_d^\sq(\ovl r)$ in \emph{loc. cit.}).

We sketch the construction of $S_{d,\ovl\rho}^\square$, following the proof of \cite[Theorem 2.6]{brehelsch}. Let $S_d^{\sq,0}$ be the locus in the trianguline variety $S_d$ where $D_d^\univ$ is étale (it is in general a na\"ive open, i.e. a not necessarily admissible union of open subspaces of $S_d^\sq$), and let $S_d^\adm$ be the maximal subspace of $S_d^{\sq,0}$ with the property that $D_d^\univ\vert_{S_d^\adm}$ is attached to a sheaf of $G_{\Q_p}$-representation $\bV$ via the construction of Berger and Colmez; it is well defined as a rigid subspace of $S_d^\sq$ by \cite[Theorem 1.2]{hellfam}. We let $S_d^{\sq,\adm}$ be the space of trivializations of the vector bundle $\bV$, that naturally comes equipped with a projection $\pi\colon S_d^{\sq,\adm}\to S_d^\adm$, and with a $G_{\Q_p}$-representation $\pi^\ast\bV$ (not just a sheaf, since we trivialized $\bV$). Since $G_{\Q_p}$ is topologically finitely generated, there exists an admissible subspace of $S_d^{\sq,\adm}$ over which $\pi^\ast\bV$ admits a lattice $\cV$. 
Then $S_{d,\ovl\rho}^\square$ is defined as the subspace of $S_d^{\sq,\adm}$ where the reduction of $\cV$ is $\ovl\rho$. Observe that the above construction equips $S_{d,\ovl\rho}^\square$ with a map $\psi_{\ovl\rho}\colon S_{d,\ovl\rho}^\square\to S_d$, that is not an embedding. %attached to the morphism of functors $(\rho, \Fil^\bullet D_\rig(\rho), \nu)\mapsto (D_\rig(\rho), \Fil^\bullet D_\rig(\rho), \nu)$, but \psi_{\ovl\rho} is not an embedding. 

As in \cite[Definition 2.4]{brehelsch}, we define $X^\square_{\ovl\rho,\tri}$ (noted $X^\square_\tri(\ovl\rho)$ in \emph{loc. cit.}) as the Zariski-closure of the subset
\[ U_{\ovl\rho,\tri}^\square\subset X_{\ovl\rho}^\square\times\cT_L^n \]
consisting of pairs $(\rho,(\delta_i)_i)$ of a lift $\rho\colon G_{\Q_p}\to\GL_n(L)$ of $\ovl\rho$ and the parameter $(\delta_i)_i$ of a triangulation of $D_\rig(\rho)$. We equip $X^\square_{\ovl\rho,\tri}$ with its reduced structure of closed analytic subspace of $X_{\ovl\rho}^\square\times\cT_L^n$. We denote by $X^{\square,\reg}_{\ovl\rho,\tri}$ the preimage of $\cT_{L,d}^\reg$ under the natural morphism $\omega_2\colon X^\square_{\ovl\rho,\tri}\to\cT_L^n$ mapping a pair $(\rho,(\delta_i)_i)$ to its parameter.

The rigid analytic spaces we introduced fit in a commutative diagram
\begin{equation}
\begin{tikzcd}[column sep=4em]
S_{d,\ovl\rho}^\square\arrow[dr,"\pi_{\ovl\rho}",swap]\arrow{drr}{\kappa}\arrow{d}\arrow{rr}{\psi_{\ovl\rho}} & & S_d\\
X_{\ovl\rho}^\square & X^\square_{\ovl\rho,\tri}\arrow{l}{\omega_1}\arrow[r,"\omega_2",swap] & \cT_L^n
\end{tikzcd}
\end{equation}
where:
\begin{itemize}
\item $\omega_1,\omega_2$ are the compositions of the embedding $X_{\ovl\rho,\tri}^\square\subset X_{\ovl\rho}^\square\times\cT_L^n$ with the projections to the two factors, 
\item $\pi_{\ovl\rho}$ is obtained by factoring through $X_{\ovl\rho,\tri}^\square$ the morphism $S_{d,\ovl\rho}^\square\to X_{\ovl\rho}^\square\times\cT_L^n$ mapping a triple $(\rho,\Fil^\bullet D_\rig(\rho), \nu)$ to the pair consisting of $\rho$ and the parameter of $\Fil^\bullet D_\rig(\rho)$,
\item the map $\psi_{\ovl\rho}$ is the one defined above. %embedding $\iota_{\ovl\rho}$ is attached to the morphism of functors $(\rho, \Fil^\bullet D_\rig(\rho), \nu)\mapsto (D_\rig(\rho), \Fil^\bullet D_\rig(\rho), \nu)$.
\end{itemize}

\begin{rem}
Chenevier \cite[Corollaire 3.18 and the discussion thereafter]{chendens} defines a na\"ive open $S_d(\ovl\rho)$ of $S_d$ (in his notation, $S_d^\sq$), equipped with a map $S_d(\ovl\rho)\to\fX_d(\ovl\rho)$ to the $\ovl\rho$-component of a character space for $G_{\Q_p}$. However, this character space is in general only equipped with a universal determinant, or pseudorepresentation. We prefer to work with the space defined in \cite{brehelsch}, that is equipped with a map to the framed deformation space of $\ovl\rho$, since it allows us to pullback a $G_{\Q_p}$-representation and a lattice from the deformation space.
\end{rem}

%FALSE, need adm--Observe that every rigid analytic subspace of the étale locus S_d^0\subset S_d carries a sheaf of G_{\Q_p}-representations.
Our interest in the construction lies in the next lemma. Let $U$ be a strictly quasi-Stein subspace of $S_d$, contained in the étale locus $S_d^{0}$, and such that $\cO^+_U(U)$ is a UFD. Let $\cU=\Spf\cO^+_U(U)$.

\begin{lemma}\label{existsGrep}
Assume that 
\begin{itemize}
\item either $\ovl\rho$ is irreducible, or it is multiplicity-free and $\cO^+_U(U)$ is a UFD;
\item for every $x\in U(L)$, the representation of $G_{\Q_p}$ attached to $x$ is a lift of $\ovl\rho$.
\end{itemize}
Then the space $U$ carries a sheaf of $G_{\Q_p}$-representations $\bV$, and, for every $n\ge 1$ and $x\in U(L)$, $\bV$ is pointwise constant mod $p^n$ over $U_{x,\cU}^{(n)}$. If $\ovl\rho$ is multiplicity-free, then $\bV$ is a $G_{\Q_p}$-representation that is constant mod $\pi_L^n$ over $V_{x,\cU}^{(n)}$, for every $n\ge 1$. 
%\begin{enumerate}[label=(\roman*)]
%\item The space $U$ carries a sheaf of $G_{\Q_p}$-representations $\bV$, and, for every $n\ge 1$ and $x\in U(L)$, $\bV$ is pointwise constant mod $p^n$ over $U_{x,\cU}^{(n)}$.
%\item If $U$ is strictly quasi-Stein, $\ovl\rho$ is multiplicity-free, and $\bV$ is absolutely irreducible, then $\bV$ is a $G_{\Q_p}$-representation that is constant mod $\pi_L^n$ over $V_{x,\cU}^{(n)}$, for every $n\ge 1$. %conclusion of Theorem \ref{Unconst} holds for the sheaf of G_{\Q_p}-representations carried by $U$. %is equipped with a $G$-representation $\bV$ such that $D_\rig(\bV_x)\cong D_x$ for every $x\in U(L)$. %, and $\bV$ admits a lattice. 
%\andr{am I somehow checking the lift of ovlrho on points??}
%\end{enumerate}
\end{lemma}

%The point of part (i) is that, even if the sheaf of $G_{\Q_p}$-representations is not a $G_{\Q_p}$-representation (i.e., it is not supported by a free $\cO_U$-module), we can still deduce that it is pointwise constant on the subdomains of $U$ defined in Section \ref{secUV}.

\begin{proof}
%Since $U$ is the tube of the unique point in its special fiber, it is contained in the admissible locus $S_d^{\adm}$ by \cite[Theorem 1.2(iii)]{hellfam}. In particular, $U$ is contained in the image of the map $\psi_{\ovl\rho}$ described above. 
Since the residual representation is isomorphic to $\ovl\rho$ at every $L$-point of $S_d^\square$, $U$ is contained in the image of the map $\psi_{\ovl\rho}$ described above. 
Since $U$ is the tube of the unique point in its special fiber, it is contained in the admissible locus $S_d^{\adm}$ by \cite[Theorem 1.2(iii)]{hellfam}, so that $D_d^\univ\vert_U$ is attached to a sheaf of $G_{\Q_p}$-representations $\bV_U$. Now the conclusion follows from Corollary \ref{woloc}.
%Statement (ii) is an immediate consequence of Corollary \ref{quasistein}.
%$U\subset S_{d}^\square$ factors through $S_{d,\ovl\rho}^\square\into S_d^\square$, i.e., $U\subset S_{d,\ovl\rho}^\sq$. In particular, restricting $\pi_{\ovl\rho}$ to $U$ gives us a morphism $U\to X_{\ovl\rho}^\sq$.
%
%By Remark \ref{exlatt}, $X_{\ovl\rho}^\sq$ is equipped with a $G_{\Q_p}$-representation $\bV_{\ovl\rho}^{\sq,\univ}$. % admitting a lattice $\cV_{\ovl\rho}^{\sq,\univ}$.
%Via Remark \ref{pullback}, we can pull $\bV_{\ovl\rho}^{\sq,\univ}$ back to a $G$-representation $\bV$ over $\psi_{\ovl\rho}^{-1}(U)$ satisfying the obvious compatibility we require. %, and $\cV_{\ovl\rho}^{\sq,\univ}$ to a lattice in $\bV$.
\end{proof}

%Note that, if $d$ is the rank of $\bV$, then the cover $\psi_{\ovl\rho}\colon\psi_{\ovl\rho}^{-1}(U)\to U$ is the $\GL_d$-torsor of trivializations of $\bV$, that appears in the proof of Corollary \ref{woloc}. \andr{???}

%\andr{whole section can be done for GK}

\subsection{Crystalline representations of dimension 2}\label{seccris}

In this section, we apply Theorem \ref{Unconst} to the study of $G_{\mathbb{Q}_p}$-stable lattices in crystalline representations of dimension 2. We start by recalling how these representations are classified, up to twist, by their weight and the trace of their crystalline Frobenius, and how to see them inside of the universal trianguline deformation space introduced in Section \ref{shtri}.

As usual, let $L$ be a $p$-adic field, with valuation ring $\cO_L$ having maximal ideal $\fm_L$ and residue field $k_L$. For an integer $k\ge 2$ and an element $a_p\in\fm_L$, let $V_{k,a_p}$ be the dual of the crystalline, $L$-linear representation whose associated filtered $\varphi$-module $D_{k,a_p}$ is a 2-dimensional $L$-vector space equipped with the structures defined, in a basis $(e_1,e_2)$, by
\[ \varphi=\begin{pmatrix}0 & -1 \\ p^{k-1} & a_p\end{pmatrix},\quad \Fil^iD_{k,a_p}=\begin{cases}D_{k,a_p}\text{ if }i\le 0 \\ Le_1\text{ if }1\le i\le k-1 \\ 0\text{ if }i\ge k\end{cases}. \]
Observe that the two eigenvalues of the crystalline Frobenius are the roots of the polynomial $T^2-a_pT+p^{k-1}$. 

By \cite[Proposition 3.1]{brequel}, every 2-dimensional, crystalline $L$-linear representation of $G_{\Q_p}$ is isomorphic to $V_{k,a_p}$ for some $k,a_p$, up to twist with a crystalline character. 
Following \cite[Proposition 2.4.1]{bellchen}, we recall that $V_{k,a_p}$ admits a \emph{non-critical} triangulation (i.e. one whose associated Frobenius filtration is in general position with respect to the Hodge filtration). The parameters $\delta_1,\delta_2\colon\Q_p^\times\to L^\times$ of one such non-critical triangulation are given as follows: if $(\varphi_1,\varphi_2)$ is the ordering of eigenvalues (``refinement'') of the crystalline Frobenius $\varphi$ attached to the triangulation, and $(k_1,k_2)$ the ordering of the Hodge--Tate weights of $V_{k,a_p}$ determined by the corresponding filtration of $D_\crys(V_{k,a_p})$, then for $i=1,2$ $\delta_i$ is the unique character satisfying 
\begin{itemize}
	\item $\delta_i (\gamma)=\gamma^{-k_i}$ for all $\gamma \in \mathbb{Z}_p^\times$, and
	\item $\delta_i (p)=\varphi_i p^{-k_i}$.
\end{itemize}

If $\varphi_1$ is the root of $T^2-a_pT+p^{k-1}$ of smaller $p$-adic valuation, then the ordering $(\varphi_1,\varphi_2)$ always corresponds to a non-critical triangulation, with $k_1=0,k_2=k-1$ (see \cite[Remark 2.4.6]{bellchen}). 
With this choice, characters $\delta_1,\delta_2$ as above and Colmez's notation for trianguline representations \cite[Section 0.3]{colmez}, we obtain an isomorphism $V_{k,a_p}\cong V(\delta_1,\delta_2)$. Now let $D^\circ(0,1)$ be the wide open unit disc over $\Q_p$. The $\Q_p$-points of $D^\circ(0,1)$ are in bijection with $G_{\Q_p}$-orbits of elements of $\Qp$ of positive valuation. Since $\delta_1$ and $\delta_2$ obviously vary in a $p$-adically analytic way as $a_p$ varies over $D^\circ(0,1)$, mapping $a_p\in D(0,1)$ to the point $V_{k,a_p}\cong V(\delta_1,\delta_2)$ of the trianguline variety $S_2^\sq$ gives an embedding $D^\circ(0,1)\into S_2^\sq$ compatible with the $(\varphi,\Gamma)$-modules at every specialization on each side. 

%If $V$ is generic, \cite[Proposition 2.4.1]{bellchen} gives an explicit description of the parameters $(\delta_1,\ldots,\delta_d)$ of any triangulation $\Fil^\bullet$ of $D_\rig(V)$ in terms of the Hodge--Tate weights of $V$ and the eigenvalues of the crystalline Frobenius: if $(\varphi_1,\ldots,\varphi_d)$ is the ordering of Frobenius eigenvalues attached to $\Fil^\bullet$, then $\delta_i\colon\Q_p^\times\to E^\times$ is the unique character satisfying 
%\begin{itemize}
%	\item $\delta_i (\gamma)=\gamma^{-k}$ for all $\gamma \in \mathbb{Z}_p^\times$, and
%	\item $\delta_i (p)=\varphi_i p^{-k}$.
%\end{itemize}

%Under the assumption that $V$ is generic crystalline, Bella\"iche and Chenevier also proved that the parameter $(\delta_i : \mathbb{Q}_p^\times \rightarrow \mathbb{E}^\times)_i$ of the triangulation $\mathcal{T}$ can be expressed in terms of $\mathcal{F}$ via the formulas $\delta_i (p)=\varphi_i p^{-s_i}$ and $\delta_i (\gamma)=\gamma^{-k_i}$ for all $\gamma \in \mathbb{Z}_p^\times$.

Because of Corollary \ref{pstorep}, the representations $V_{k,a_p}$ are not specializations of a sheaf of $G_{\Q_p}$-representations on the whole $D^\circ(0,1)$. However, if we restrict ourselves to a small enough subdomain of $D^\circ(0,1)$, we can interpolate them with a $G_{\Q_p}$-representation, as follows. 

Let $\ovl\rho\colon G_{\Q_p}\to\GL_2(k_L)$ be a continuous, multiplicity-free representation. Let $U$ be a strictly quasi-Stein subspace of $D^\circ(0,1)$, and let $\cU$ be its formal model $\Spf\cO_U^+(U)$. Assume that $V_{k,a_p}$ is a lift of $\ovl\rho$ for every $a_p\in U(L)$, and if $\ovl\rho$ is not absolutely irreducible, assume further that $\cO_U^+(U)$ is a UFD (for instance, take $U$ to be a wide open subdisc of $D^\circ(0,1)$). Thanks to the above discussion, we can embed $D^\circ(0,1)$, hence $U$, in $S_2^\sq$. Then Lemma \ref{existsGrep} implies the following.

%\begin{lemma}\label{crfam}
%	The rigid analytic space $U$ is equipped with a sheaf of $G_{\Q_p}$-representations $\bV$ that specializes to $V_{k,a_p}$ at any $a_p\in U(L)$.
	%	Let $k\in\Z_{\ge 2}$ and $\cL\in\bP^{1,\rig}$. There exist an open affinoid neighborhood $U$ of $\cL$ in $\bP^{1,\rig}$ and a family of $G_{\Q_p}$-representations $\bV$ over $\bP^{1,\rig}$ with the property that, for every $\cL\in U(\Qp)$, the specialization of $\bV$ at $\cL$ is the representation $V_{k,\cL}$. %\andr{IS this true??? not just locally??? can maybe write explicitly???}
%\end{lemma}
%
%\begin{proof}
%Thanks to the above discussion, we can embed $D^\circ(0,1)$ in $S_2^\sq$. 
%The result then follows from Lemma \ref{existsGrep}, applied with $d=2$.
	%	Let $\Omega, \cA, T$ be the objects produced by Proposition \ref{existence} for $d=2$ and $x=\Phi_k(\cL)$. Then $U=\Phi_k^{-1}(\Omega\cap\Phi_k(\bP^{1,\rig}))$ is an open affinoid neighborhood of $\cL$ in $\bP^{1,\rig}$, equipped with the family of $G_{\Q_p}$-representations $\Phi_k^\ast(\wtl T[1/p])$. 
	%??The result follows directly from pulling back the universal trianguline family via the rigid analytic open immersion $\Phi_k : \mathbb{P}^{1, \rig}\rightarrow \mathcal{S}_2^{\square, 0}$ sending $\Phi_k (\mathcal{L})$ to the unique étale, trianguline $(\varphi, \Gamma)$-module $D_{k, \mathcal{L}}$ such that $D_{\rig} (D_{k, \mathcal{L}}) \cong V_{k, \mathcal{L}}$ \textcolor{red}{(there should be some dual here).}
%\end{proof}

%Once we equip $U$ with a $G$-representation via the above lemma, Corollary \ref{Unconst} immediately implies the following.

\begin{prop}\label{appcr}
	The rigid analytic space $U$ is equipped with a $G_{\Q_p}$-representation $\bV$ that specializes to $V_{k,a_p}$ at any $a_p\in U(L)$. Moreover, for every $a_{p,0}\in U(L)$ and $n\in\Z_{\ge 1}$, the $G_{\Q_p}$-representation $\bV$ is pointwise constant mod $p^n$ over $U_{a_{p,0},\cU}^{(n)}$. %In particular, for every finite extension $E$ of $L$ and $a_p\in U(E)$, there exist lattices $\cV_{k,a_{p,0}}$ in $V_{k,a_{p,0}}$ and $\cV_{k,a_p}$ in $V_{k,a_p}$ and an isomorphism \[ \cV_{k,a_{p,0}}\otimes_{\cO_E}\cO_E/\pi_E^n\cong\cV_{k,a_{p}}\otimes_{\cO_E}\cO_E/\pi_E^n \] of $\cO_E/\pi_E^n[G]$-modules. %\andr{define terminology?? introduce bar notation}
	\begin{enumerate}[label=(\roman*)]
		\item The $G_{\Q_p}$-representation $\bV$ is pointwise constant mod $p^n$ over $U_{a_{p,0},\cU}^{(n)}$; in particular, for every finite extension $E$ of $L$ and $a_p\in U(E)$, there exist lattices $\cV_{k,a_{p,0}}$ in $V_{k,a_{p,0}}$, $\cV_{k,a_p}$ in $V_{k,a_p}$, and an isomorphism \[ \cV_{k,a_{p,0}}\otimes_{\cO_E}\cO_E/\pi_E^n\cong\cV_{k,a_{p}}\otimes_{\cO_E}\cO_E/\pi_E^n \] of $\cO_E/\pi_E^n[G]$-modules. %\andr{define terminology?? introduce bar notation}
		\item The $G_{\Q_p}$-representation $\bV$ is constant modulo $\pi_L^n$ over $V_{a_{p,0},\cU}^{(n)}$.
	\end{enumerate}
	%	Fix $k\in\Z_{\ge 2}$, $n\in\Z_{\ge 1}$, and let $E$ be a $p$-adic field. 
	%	For every $\cL\in\bP^{1,\rig}(E)$, there exists an open affinoid neighborhood $\mathcal{U}_n$ of $\mathcal{L}$ inside $\mathbb{P}^{1, \rig}$ with the following property: if $\mathcal{L}_1 \in \mathcal{U}_n$ for some $n\in\mathbb{Z}_{\geq 1}$, then there exist two $G_{\mathbb{Q}_p}$-stable lattices $T_1 \subset V_{k, \mathcal{L}_1}$ and $T_2 \subset V_{k, \mathcal{L}_2}$ such that
	%	\[ T_1 \cong T_2 \mod \pi^n \]
	%	as $\mathcal{O}_E/p^n \mathcal{O}_E[G_{\mathbb{Q}_p}]$-modules. \\
	%	Moreover, if $U$ is an open affinoid neighborhood of $\cL$ as in Lemma \ref{stfam}, then we can choose $U_n=U_\varphi^{(n)}$ for a presentation $\varphi$ of $\cO_U^+(U)$ centered at $\wtl\cL$, as in Definition \ref{defAn}.
\end{prop}

Recall that if $U$ is a wide open disc of radius $p^m$ and $a_{p,0}\in U(L)$, then $U_{a_{p,0},\cU}^{(n)}$ is a wide open disc of center $a_{p,0}$ and radius $p^{m-n}$, and $V_{a_{p,0},\cU}^{(n)}$ is an affinoid disc of center $a_{p,0}$ and radius $p^{m-n}$. For instance, if we know that the isomorphism class of $V_{k,a_p}$ is constant modulo $\pi_L$ for every $a_p$ with $\lvert a_p-a_{p,0}\rvert<p^m$, we deduce that it is constant modulo $\pi_L^n$ for every $a_p$ with $\lvert a_p-a_{p,0}\rvert<p^{m+1-n}$. 

Proposition \ref{appcr} allows us to recover mod $\pi_L^n$ congruences among crystalline representations from mod $\pi_L$ congruences. In particular, it allows one to slightly improve Torti's result on higher congruences \cite[Theorem 1.1]{tortired} starting from the mod $\pi_L$ results of Berger--Li--Zhu, Berger, and Torti himself, as in the following corollary.

\begin{cor}\label{cryscor}
Let $k\ge 2$ be an integer, and $a_{p,0}\in\fm_L\setminus\{0\}$. Then one can take the domain $U$ in Proposition \ref{appcr} to be the wide open disc of center $a_{p,0}$ and radius $2v_p(a_{p,0})+\alpha(k-1)$, where $\alpha(k-1)=\sum_{s\ge 1}\lfloor\frac{s}{p^{n}(p-1)}\rfloor$. 
In particular: %Proposition \ref{appcr}(i) holds for every $a_p\in\fm_E$ satisfying
%\[ v_p(a_p-a_{p,0})>2v_p(a_{p,0})+\alpha(k-1)+n-1. \]
\begin{enumerate}[label=(\roman*)]
\item Proposition \ref{appcr}(i) holds for every $a_p\in\fm_E$ satisfying
\[ v_p(a_p-a_{p,0})>2v_p(a_{p,0})+\alpha(k-1)+n-1; \]
\item Proposition \ref{appcr}(ii) holds over the affinoid disc of center $a_{p,0}$ and valuation radius 
\[ 2v_p(a_{p,0})+\alpha(k-1)+n. \]
\end{enumerate}
\end{cor}

\begin{proof}
Let $U$ be the wide open disc defined in the statement. By \cite[Theorem A]{berloccon} in the residually irreducible case (and \cite[Theorem 1.1.1]{berlizhu} for the case $a_{p,0}=0$), and the $m=1$ case of \cite[Theorem 1.1]{tortired} in the general case, we can find a lattice $\cV_{k,a_{p,0}}$ in $V_{k,a_{p,0}}$ such that $V_{k,a_p}$ is a lift of the mod $\pi_L$ reduction $\cV_{k,a_{p,0}}^{(1)}$ for every $a_p\in U(L)$. Therefore, we can apply Proposition \ref{appcr} to $U$, and we obtain the corollary. 
\end{proof}

\begin{rem}\label{tortirem}
	Torti \cite[Theorem 1.1]{tortired} proves a mod $p^n$ pointwise constancy result over the affinoid disc of Corollary \ref{cryscor}(ii). We improve his statement to a mod $p^n$ pointwise constancy result over the larger wide open disc defined by the inequality in Corollary \ref{cryscor}(i). %and to a mod $\pi_L^n$ constancy result over the same affinoid disc he considers.
\end{rem}

\begin{rem}\label{rember}
%\begin{enumerate}[label=(\roman*)]
%\item Proposition \ref{appcr} allows us to recover mod $\pi_L^n$ congruences among crystalline representations from mod $\pi_L$ congruences. In particular, it allows one to slightly improve Torti's result on higher congruences \cite[Theorem 1.1]{tortired} starting from the mod $\pi_L$ results of Berger--Li--Zhu, Berger, and Torti himself, as follows. 
%
%Let $a_{p,0}\in\fm_L\setminus\{0\}$. By \cite[Theorem A]{berloccon} in the residually irreducible case (and \cite[Theorem 1.1.1]{berlizhu} for the case $a_{p,0}=0$), and by the mod $\pi_L$ representations $\ovl V_{k,a_p}$ and $\ovl V_{k,a_{p,0}}$ are isomorphic if $v_p(a_p-a_{p,0})>2v_p(a_{p,0})+\alpha(k-1)$, where $\alpha(k-1)$ is an explicit positive integer depending only on $k$ and $p$. 
%
%Then, one can pick $U$ to be the wide open disc of center $a_{p,0}$ and radius $2v_p(a_{p,0})+\alpha(k-1)$ inside of $D(0,1)$, and Proposition \ref{appcr} gives us that, for every $n\ge 1$, $V_{k,a_p}$ and $V_{k,a_{p,0}}$ admit isomorphic mod $\pi_L^n$ reductions if
%\[ v_p(a_p-a_{p,0})>2v_p(a_{p,0})+\alpha(k-1)+n-1. \]
%In the statement of \cite[Theorem 1.1]{tortired}, the inequality 
%\[ v_p(a_p-a_{p,0})\ge 2v_p(a_{p,0})+\alpha(k-1)+n \]
%appears instead: this means that \emph{loc. cit.} gives a mod $\pi_L^m$ constancy result over the affinoid disc of center $a_{p,0}$ and (valuation) radius $2v_p(a_{p,0})+\alpha(k-1)+n$, whereas we give a local constancy result on the wide open disc of center $a_{p,0}$ and (valuation) radius $2v_p(a_{p,0})+\alpha(k-1)+n-1$. 
As in \cite[Section 3]{berloccon}, one can instead keep $a_p\ne 0$ fixed and let $k$ vary over the points of a wide open disc $D(k_0,r)$ centered at a fixed $k_0>3v_p(a_p)+\alpha(k-1)+1$, inside of the weight space $\cW$ we introduced in Section \ref{shtri}. For $r$ sufficiently small (but not explicit), Berger embeds such a disc inside of the trianguline variety $S_2^\sq$ in a way compatible with the $(\varphi,\Gamma)$-modules on the two spaces. By \cite[Theorem B]{berloccon}, there exists a non-explicit $m\in\Z_{\ge 1}$ such that $\ovl V_{k,a_p}\cong\ovl V_{k_0,a_p}$ for every $k\in D(k_0,p^{-m})$. Therefore, Lemma \ref{existsGrep} allows one to equip $D(k_0,p^{-m})$ with a sheaf of $G$-representations, and gives us that, for every $n\ge 1$, $V_{k,a_p}$ and $V_{k_0,a_p}$ admit lattices with isomorphic mod $\pi_L^n$ reductions if 
\[ v_p(k-k_0)>m+n-1. \]
Again, this is a small improvement on \cite[Theorem 1.2]{tortired}, in that we prove pointwise constancy mod $p^n$ for $k$ varying over a wide open disc, rather than over an affinoid disc contained in it. We also weaken the assumption on $k_0$ in \emph{loc. cit.}: we only need it to satisfy Berger's assumption that $k_0>3v_p(a_p)+\alpha(k-1)+1$, rather than Torti's stronger condition that also depends on the depth $n$ of the congruence we are looking at.
\end{rem}

For small values of $p$, we can derive a global consequence of Remark \ref{rember} via a result of Chenevier. We write $G_{\Q,p\infty}$ for the Galois group over $\Q$ of the maximal extension of $\Q$ unramified away from $p$ and $\infty$.

\begin{cor}\label{crysglob}
Assume that $p\ge 7$, and let $\rho_1,\rho_2\colon G_{\Q,p\infty}\to\GL_2(L)$ be two continuous representations, crystalline at $p$. Let $k_1,a_{p,1}$ and $k_2,a_{p,2}$ be the corresponding local data at $p$. Then, for every $n\ge 1$, $\rho_1$ and $\rho_2$ admit isomorphic mod $\pi_L^n$ reductions if either:
\begin{enumerate}
	\item $k_1=k_2, a_{p,1}\ne 0$ and $v_p(a_{p,2}-a_{p,1})>2v_p(a_{p,0})+\alpha(k-1)+n-1$;
	\item $a_{p,1}=a_{p,2}\ne 0$ and $v_p(k-k_0)>m+n-1$, with $m$ being the non-explicit constant appearing in \cite[Theorem B]{berloccon}.
\end{enumerate}
If $\rho_1$ and $\rho_2$ are attached to two $\GL_{2/\Q}$-eigenforms $f_1$ and $f_2$ of level $\Gamma_1(p)$, then the Hecke eigensystems of $f_1$ and $f_2$ are congruent modulo $\pi_L^n$ away from $p$.
\end{cor}

\begin{proof}
By \cite[Proposition 1.8]{chenquelques}, it is enough to check mod $\pi_L^n$ congruences locally at $p$. The first statement then follows from Remark \ref{rember}. The statement about eigenforms follows by writing Hecke eigensystems in terms of Frobenius eigenvalues in the usual way.
\end{proof}

\subsection{Semi-stable representations of dimension 2}

%Let's first set some notation. Let $x$ and $|x|$ denote respectively the identity character on $\mathbb{Q}^{\times}_p$ and the character of $\mathbb{Q}^{\times}_p$ sending $x$ to $p^{- v_p (x)}$. Hence, we have that the cyclotomic character is $\chi=x|x|$. \\

In this section, we apply Theorem \ref{Unconst} to the study of $G_{\mathbb{Q}_p}$-stable lattices in semi-stable representations of dimension 2. We start by recalling how such representations are classified, up to twist, by their weight and $\cL$-invariant, and how to see them inside of the universal trianguline deformation space introduced in Section \ref{shtri}.
%First we have to introduce the notation for the semistable representation (in dimension 2) of the form $V_{k, \mathcal{L}}$ (not all of them are of this type of course!) from a $(\varphi, \Gamma)$-module. \textcolor{red}{AFTER THAT}.

%The first step is to write the semistable representation $V_{k, \mathcal{L}}$ in terms of trianguline representations. According to the work of Colmez (\textcolor{red}{reference} ), trianguline representations of dimension two are parametrized by a $\mathbb{Q}_p$-rigid analytic space $\mathcal{S}_2$ whose points are of the form $s=(\delta_1, \delta_2, \mathcal{L})$ where $\delta_1, \delta_2$ are multiplicative characters of $\mathbb{Q}^{\times}_p$. 

Let $\pi$ be a square root of $p$ in $\Qp$. Let $k$ be an integer at least 2 and $\cL\in\bP^1(\Qp)=\Qp\cup\{\infty\}$.
We denote by $V_{k, \mathcal{L}}$ the dual of the semistable representation of $G_{\Q_p}$ whose associated $(\varphi,N)$-module $D_{k,\cL}$ is a 2-dimensional $\Qp$-vector space equipped with the structures defined, in a basis $(e_1,e_2)$, by
\[ \varphi=\begin{pmatrix}\varpi^{k-2} & 0 \\ 0 & \varpi^{k-2}\end{pmatrix},\quad N=\begin{pmatrix}0 & 0 \\ 1 & 0\end{pmatrix}, \quad \Fil^iD_{k,\cL}=\begin{cases}D_{k,\cL}\text{ if }i\le 0 \\ \Qp(e_1+\cL e_2)\text{ if }1\le i\le k-1 \\ 0\text{ if }i\ge k\end{cases} \]
if $\cL\in\Qp$, and by
\[ \varphi=\begin{pmatrix}\varpi^{k-2} & 0 \\ 0 & \varpi^{k-2}\end{pmatrix},\quad N=0, \quad \Fil^iD_{k,\cL}=\begin{cases}D_{k,\cL}\text{ if }i\le 0 \\ \Qp(e_1+e_2)\text{ if }1\le i\le k-1 \\ 0\text{ if }i\ge k\end{cases} \]
if $\cL=\infty$. 
The above description for $V_{k,\infty}$ can be obtained by rewriting $V_{k,\cL}$, $\cL\in\Qp$, in the basis $(e_1,\cL e_2)$, and letting $\cL\to\infty$.

Clearly, $V_{k,\cL}$ is semistable and non-crystalline of Hodge--Tate weights $(0,k-1)$ for every $\cL\in\Qp$. By a standard calculation, every semistable, non crystalline representation of $G_{\Q_p}$ of Hodge--Tate weights $(0,k-1)$ is, up to twist with a semistable character of $G_{\Q_p}$, isomorphic to $V_{k,\cL}$ for a unique $\cL\in\Qp$. On the other hand, $V_{k,\infty}$ is crystalline.

We will use the notation introduced in Section \ref{shtri} for characters of $\Q_p^\times$. 
Let $\alpha\colon\Q_p^\times\to\Qp^\times$ be the character $\pi^{v_p(x)}\lvert x\rvert^{-1}$. 
Let $\delta_{1,k}=\lvert x\rvert\alpha$ and $\delta_{2,k}=x^{-k}\alpha$. %\andr{doublecheck??}
We denote by $V(\delta_{1,k},\delta_{2,k},\cL)$ the unique trianguline representation attached to the above data as in \cite[Section 0.3]{colmez}. Via \cite[Proposition 4.18]{colmez}, one checks that the representations $V_{k,\cL}$ and $V(\delta_{1,k},\delta_{2,k},\cL)$ are isomorphic for every $k\in\Z_{\ge 2}$ and $\cL\in\bP^1(\Qp)$.

%\begin{lemma}
%	For every $k\in\Z_{\ge 2}$ and $\cL\in\bP^1(\Qp)$, the representations $V_{k,\cL}$ and $V(\delta_{1,k},\delta_{2,k},\cL)$ are isomorphic.
%	%Let $\pi\in \overline{\mathbb{Q}}_p$ such that $\pi^2=p$. Let $V_{k, \mathcal{L}}$ be the semi-stable representation of HT weights $\{0, k-1\}$ and invariant $\mathcal{L} \in\mathbb{P}^{1, \rig} (\mathbb{E})$.
%	%We have an isomorphism of trianguline representations:
%	%$$V_{k, \mathcal{L}} \cong V(\delta_{1, k}, \delta_{2, k}, \mathcal{L}),$$
%	%where $\delta_{1,k} = |x|\alpha$, $\delta_{2, k}=x^{-k} \alpha$ and $\alpha$ is the multiplicative character of $\mathbb{Q}^\times_p$ such that $\alpha(x)=\pi^{v_p (x)} \lvert x\rvert^{-1}$. 
%\end{lemma}
%
%\begin{proof}
%	This follows from a direct computation, as a corollary of \cite[Proposition 4.18]{colmez}. %Mention that $\omega(\delta_{1, k})=0$ and $\omega(\delta_{2, k})=-k$ \andr{complete----}
%\end{proof} 

Let $\bP^{1,\rig}$ be the rigid analytic projective line over $\Q_p$, parameterized with a variable $\cL$. Because of Corollary \ref{pstorep}, the representations $V_{k,\cL}$ are not specializations of a sheaf of $G_{\Q_p}$-representations on the whole $\bP^1$. However, if we restrict ourselves to a small enough subdomain of $\bP^1$, we can interpolate them with a $G_{\Q_p}$-representation, as follows. 

Let $\ovl\rho\colon G_{\Q_p}\to\GL_2(k_L)$ be a continuous, multiplicity-free representation. Let $U$ be a strictly quasi-Stein subspace of $\bP^{1,\rig}$, and let $\cU$ be its formal model $\Spf\cO_U^+(U)$. Assume that $V_{k,\cL}$ is a lift of $\ovl\rho$ for every $\cL\in U(L)$. If $\ovl\rho$ is not absolutely irreducible, assume further that $\cO_U^+(U)$ is a UFD (for instance, take $U$ to be a wide open subdisc of $\bP^{1,\rig}$). The construction of the trianguline deformation space $S_{2}^\square$ in \cite[Section 0.2]{colmez} provides us with a closed immersion $\Phi_k\colon\bP^{1,\rig}\into S_2^\square$, mapping $\cL\in\bP^{1,\rig}$ to the point $\Phi_k(\cL)$ of $S_2^\square$ corresponding to the trianguline representation $V_{k,\cL}$. We use $\Phi_k$ to embed $U$ in $S_2^\square$, and we deduce the following proposition from Lemma \ref{existsGrep}. %The result now follows from Lemma \ref{existsGrep}, applied with $d=2$. 

\begin{prop}\label{appst}
The rigid analytic space $U$ is equipped with a sheaf of $G_{\Q_p}$-representations $\bV$ that specializes to $V_{k,\cL}$ at any $\cL\in U(L)$. Moreover, for every $\cL_0\in U(L)$ and $n\in\Z_{\ge 1}$: %, the $G$-representation $\bV$ is pointwise constant mod $p^n$ over $U_{\cL_0,\cU}^{(n)}$; in particular, for every finite extension $E$ of $L$ and $\cL\in U(E)$, there exist lattices $\cV_{k,\cL_0}$ in $V_{k,\cL_0}$ and $\cV_{k,\cL}$ in $V_{k,\cL}$ and an isomorphism \[ \cV_{k,\cL_0}\otimes_{\cO_E}\cO_E/\pi_E^n\cong\cV_{k,\cL}\otimes_{\cO_E}\cO_E/\pi_E^n \] of $\cO_E/\pi_E^n[G]$-modules.
		\begin{enumerate}[label=(\roman*)]
			\item The $G_{\Q_p}$-representation $\bV$ is pointwise constant mod $p^n$ over $U_{\cL_0,\cU}^{(n)}$; in particular, for every finite extension $E$ of $L$ and $\cL\in U(E)$, there exist lattices $\cV_{k,\cL_0}$ in $V_{k,\cL_0}$ and $\cV_{k,\cL}$ in $V_{k,\cL}$ and an isomorphism \[ \cV_{k,\cL_0}\otimes_{\cO_E}\cO_E/\pi_E^n\cong\cV_{k,\cL}\otimes_{\cO_E}\cO_E/\pi_E^n \] of $\cO_E/\pi_E^n[G]$-modules. %\andr{define terminology?? introduce bar notation}
			\item The $G_{\Q_p}$-representation $\bV$ is constant modulo $\pi_L^n$ over $V_{a_{\cL_0},\cU}^{(n)}$.
		\end{enumerate}
%For every $\cL_0\in U(L)$ and $n\in\Z_{\ge 1}$, 
%\begin{itemize}
%\item the modulo $\pi_L^n$ isomorphism class of the $G$-representation $V_{k,\cL}$ is independent of $\cL\in U_{\cL_0,\cU}^{(n)}(L)$; \andr{define terminology??}
%\item the $G$-representation $\bV\vert_{V_{\cL_0,\cU}^{(n)}}$ is constant modulo $\pi_L^n$.
%\end{itemize}
%	Fix $k\in\Z_{\ge 2}$, $n\in\Z_{\ge 1}$, and let $E$ be a $p$-adic field. 
%	For every $\cL\in\bP^{1,\rig}(E)$, there exists an open affinoid neighborhood $\mathcal{U}_n$ of $\mathcal{L}$ inside $\mathbb{P}^{1, \rig}$ with the following property: if $\mathcal{L}_1 \in \mathcal{U}_n$ for some $n\in\mathbb{Z}_{\geq 1}$, then there exist two $G_{\mathbb{Q}_p}$-stable lattices $T_1 \subset V_{k, \mathcal{L}_1}$ and $T_2 \subset V_{k, \mathcal{L}_2}$ such that
%	\[ T_1 \cong T_2 \mod \pi^n \]
%	as $\mathcal{O}_E/p^n \mathcal{O}_E[G_{\mathbb{Q}_p}]$-modules. \\
%	Moreover, if $U$ is an open affinoid neighborhood of $\cL$ as in Lemma \ref{stfam}, then we can choose $U_n=U_\varphi^{(n)}$ for a presentation $\varphi$ of $\cO_U^+(U)$ centered at $\wtl\cL$, as in Definition \ref{defAn}.
\end{prop}

Recall that if $U$ is a wide open disc of radius $p^m$ and $\cL_0\in U(L)$, then $U_{\cL_0,\cU}^{(n)}$ is a wide open disc of center $\cL_0$ and radius $p^{m-n}$, and $V_{\cL_0,\cU}^{(n)}$ is an affinoid disc of center $\cL_0$ and radius $p^{m-n}$. For instance, if we know that the isomorphism class of $V_{k,\cL}$ is constant modulo $\pi_L$ for every $\cL$ with $\lvert\cL-\cL_0\rvert<p^m$, we deduce that it is constant modulo $\pi_L^n$ for every $\cL$ with $\lvert\cL-\cL_0\rvert<p^{m+1-n}$. 

%\begin{proof}
%	The result follows by applying Corollary \ref{reductions} to the family $\bV$ over $U$ produced by Lemma \ref{stfam}.
%\end{proof}

In the special case $\cL_0=\infty$, we know explicitly of such an $m$ by the work of Bergdall--Levin--Liu \cite{berlevliu}, so we are able to compare the mod $p^n$ reduction of certain semistable representation with that of a crystalline representation. %describe the reduction of certain semistable representations modulo powers of $p$.

\begin{thm}\label{bll}
Assume that $k\ge 4$ and $p\ne 2$. %\andr{odd throughout??}
Then, if $L$ is a $p$-adic field and $\cL\in\bP^{1,\rig}(L)$ satisfies
\[ v_p(\cL)<2-\frac{k}{2}-v_p((k-2)!)+1-n, \]
there exist lattices $\cV_\cL$ and $\cV_\infty$ in $V_{k,\cL}$ and $V_{k,\infty}$, respectively, such that $\cV_\cL\otimes_{\cO_L}\cO_L/\pi_L^{\gamma_{L/\Q_p}(n)}\cong\cV_\infty\otimes_{\cO_L}\cO_L/\pi_L^{\gamma_{L/\Q_p}(n)}$ as $\cO_L/\pi_L^{\gamma_{L/\Q_p}(n)}[G_{\Q_p}]$-modules.
\end{thm}

\begin{proof}
By \cite[Theorem 1.1]{berlevliu}, we are allowed to apply Proposition \ref{appst} the wide open disc $U$ of center $\infty$ defined by the inequality $v_p(\cL)<2-\frac{k}{2}-v_p((k-2)!)$. Such a disc is defined over $\Q_p$, hence why the exponent $\gamma_{L/\Q_p}(n)$ appears. 
\end{proof}

When $k<p$, one can write down similar results for various other ranges of values of $\cL$ (in particular, for ranges including the $\cL$-invariants of modular eigenforms), thanks to the computations of mod $\pi_L$ reductions of semistable representations performed by Breuil--Mezard \cite{bremez} and Guerberoff--Park \cite{gueparht}. %\andr{should we write the results down?? also Arsovski??}

In the same way as for crystalline representations, we can rely \cite[Proposition 1.8]{chenquelques} to derive a global consequence of Remark \ref{rember} for small values of $p$. %We write $G_{\Q,p\infty}$ for the Galois group over $\Q$ of the maximal extension unramified away from $p$ and $\infty$.

\begin{cor}\label{stglob}
	Assume that $p\ge 7$, and let $\rho\colon G_{\Q,p\infty}\to\GL_2(L)$ be a continuous representation, semistable at $p$, with associated local data $k,\cL$. Let $\rho_\infty$ be a representation whose restriction to $G_{\Q_p}$ is isomorphic to $V_{k,\infty}$. Then, for every $n\ge 1$, $\rho$ and $\rho_\infty$ admit isomorphic mod $\pi_L^n$ reductions if $v_p(\cL)<2-\frac{k}{2}-v_p((k-2)!)+1-n$.
%	If \rho_1 and \rho_2 are attached to two \GL_{2/\Q}-eigenforms f_1 and f_2 of level \Gamma_1(p), then the Hecke eigensystems of f_1 and f_2 are congruent modulo \pi_L^n away from p.
\end{cor}

\begin{proof}
	By \cite[Proposition 1.8]{chenquelques}, it is enough to check mod $\pi_L^n$ congruences locally at $p$. The statement then follows from Theorem \ref{bll}.
\end{proof}

Note that we do not state any automorphic consequence: standard estimates for the valuation of the $\cL$-invariant show that no eigenform can have an $\cL$-invariant of valuation smaller than $2-\frac{k}{2}-v_p((k-2)!)$.

\subsection{Pseudorepresentations along the eigencurve}\label{seceigen}

We give a small and not very explicit application to the modulo $p^n$ variation of the Galois (pseudo-)representations carried by the eigencurves for $\GL_{2/\Q}$. In this non-explicit form, similar statements should hold for more general eigenvarieties.

Fix a prime-to-$p$ integer $N\in\Z_{\ge 1}$. Let $\cE$ be the $p$-adic eigencurve of tame level $\Gamma_1(N)$, as constructed by Coleman--Mazur, Buzzard, and Chenevier. It comes equipped with a morphism to the weight space $\cW$ we introduced in Section \ref{shtri}. Let $h$ be a positive real number, $\kappa$ a point of $\cW$, and for every $r\in p^{\Q}$ consider the wide open disc $D^\circ(\kappa,r)\subset\cW$ of center $\kappa$ and radius $r$. Consider the subdomain $\cE_r^{\le h}$ of the eigencurve whose points have slope $\le h$ and weight in $D^\circ(\kappa,r)$. By \cite[Section II.3.3]{eigenbook}, there exists a radius $r\in p^{\Q}$ such that the restriction of the weight map gives us a \emph{finite} map $\omega\colon\cE_r^{\le h}\to D^\circ(\kappa,r)$. In particular, $\cE_r^{\le h}$ is itself a wide open in which classical points are dense, and a standard argument of Chenevier provides us with a continuous pseudorepresentation $G_{\Q}\to\cO_\cE(\cE_r^{\le h})$. For every irreducible component $X$ of $\cE_r^{\le h}$, we obtain this way a continuous pseudorepresentation $T\colon G_\Q\to\cO_X(X)$.

We will apply the following lemma. Let $\cX=\Spf\cA$ and $\cY=\Spf\cB$ be two affine formal schemes with rigid analytic generic fibers $X$ and $Y$, respectively. Let $f\colon\cX\to\cY$ be a finite map of affine $\cO_L$-formal schemes, attached to a homomorphism $f^\ast\colon\cB\to\cA$. Let $f^\rig\colon X\to Y$ be the map induced by $f$ on rigid generic fibers. Let $y\in Y(L)$ be a point with a single preimage $x$ under $f^\rig$.

\begin{lemma}\label{XYn}
There exists $n_0\in\Z_{\ge 0}$ such that, for every $n\ge n_0$, $U_{x,\cX}^{(n)}=f^{\rig,-1}U_{y,\cY}^{(n)}$ and $V_{x,\cX}^{(n)}=f^{\rig,-1}V_{y,\cY}^{(n)}$. 
\end{lemma}

\begin{proof}
We use the notation of Section \ref{secUV}. By the finiteness of $f$, for every element $Y\in I_y$ there exist $a_0,\ldots,a_d\in\cO_L$ such that $P(Y)\coloneqq\sum_{i}a_iY^i\in I_x$. For $n$ sufficiently large, $\lVert P(Y)\rVert<p^n$ implies $\lVert a_dY^d\rVert<p^n$, hence $\lVert Y\rVert< p^{n/d}$ since $a_d\in\cO_L$. This proves that $f^{\rig,-1}U_{y,\cY}^{(n)}\subseteq U_{x,\cX}^{(n)}$. The opposite inclusion is obvious (for every $n\ge 1$), since every element of $I_y$ can be seen as a function on $\cX$ via $f$.
%Since $x$ is the only preimage of $y$, $I_y=f^\ast(I_x)$. 
\end{proof}

Now pick an irreducible component $X$ of $\cE_r^{\le h}$, and let $\omega_X\colon X\to D^\circ(\kappa,r)$ denote the restriction of the weight map to $X$. Let $x$ be a point of $X$, and $\kappa=\omega_X(x)$. After possibly choosing some $r^\prime\in p^\Q, r^\prime<r$ and (implicitly) replacing $X$ with the irreducible component of $\cE_r^{\le h}\times_{D^\circ(y,r)}D^\circ(y,r^\prime)$ containing $X$, we can assure that $\omega_X^{-1}(\omega_X(x))=\{x\}$. Then, Proposition \ref{Tconst} and Lemma \ref{XYn} give us the following proposition. %As usual, we denote by $\ev_z$ the evaluation of regular functions at a point $z\in X(L)$. 
%We write $\cX=\Spf\cO_X^+(X)$ and $\cD=\Spf\cO_\cW^+(D^\circ(y,r^\prime))$ for the standard formal models of the two wide opens $X$ and $D^\circ(y,r^\prime)$. We write $\cT\colon G_\Q\to\cO_X^+(X)$ for the pseudorepresentation induced from $T$.

\begin{prop}\label{ecurvelc}
There exists $n_0\in\Z_{\ge 0}$ such that, for every $n\ge n_0$, the pseudorepresentation $T$ is pointwise constant mod $p^n$ over $\omega_X^{-1}(D^\circ(\kappa,p^{1-n}r^\prime))$, and constant mod $p^n$ over $\omega_X^{-1}(D(\kappa,p^{-n}r^\prime))$.
%mod $\pi_L^n$ pseudorepresentation $\ev_z\ccirc\cT$ is independent of the choice of a point $z$ of $X$ of weight in $D^\circ(y,p^{1-n}r^\prime)$. Moreover, the pseudorepresentation $\cT$ is constant mod $\pi_L^n$ over the preimage in $X$ of the affinoid disc $D(y,p^{-n}r^\prime)$.
\end{prop}

If we restrict ourselves to the ordinary part of the eigencurve, i.e. we consider the case $h=0$ in the above picture, the situation is much simpler: every ordinary irreducible component $X$ of $\cE$ is the rigid generic fiber of the formal scheme attached to a (schematic) Hida family; in particular, $X$ is finite and flat over $\cW$, so that we can take $r=1$ as a radius adapted to $h=0$. By \cite{hatnewirr}, conditionally on the spectral halo conjecture, this optimal situation only happens in the ordinary case, as the weight map is never finite on nonordinary components of $\cE$.

We state a simple corollary of Proposition \ref{ecurvelc} for ordinary components.

\begin{cor}\label{corhida}
Let $X$ be an ordinary component of $\cE$, and let $\kappa\in\cW(\Qp)$ be a weight with the property that $X$ has a single point $x$ of weight $\kappa$. Then there exists $n_0\in\Z_{\ge 0}$ such that, for every $n\ge n_0$, the pseudorepresentation $T$ is pointwise constant mod $p^n$ over $\omega_X^{-1}(D^\circ(\kappa,p^{1-n}))$, and constant mod $p^n$ over $\omega_X^{-1}(D(\kappa,p^{-n}))$. In particular, if $L$ is a $p$-adic field, then the Hecke eigensystems away from $p$ of any two overconvergent eigenforms attached to $L$-points of $X$ of weight in $D^\circ(\kappa,p^{1-n})$ are congruent mod $\pi_L^n$.
\end{cor}

\begin{rem}\mbox{ }
\begin{enumerate}
\item If the residual $G_\Q$-representation attached to the Hida family is absolutely irreducible, then Corollary \ref{pstorep} allows one to deduce from Corollary \ref{corhida} a mod $p^n$ pointwise constancy result for the $G_\Q$-representation attached to the Hida family, not just its trace.
\item We could have deduced the congruence in the last statement of Corollary \ref{corhida} by applying Theorem \ref{Unconst} directly to the Hecke eigenvalue of each Hecke operator away from $p$, given that it is interpolated by a power-bounded rigid analytic function on $X$.
%\item By applying Corollary \ref{pstorep}, we can rewrite Corollary \ref{corhida} as a mod $p^n$ constancy result for the actual $G_\Q$-representations attached to eigenforms. 
\item For some $p$-adic families, it might be possible to show that the constancy neighborhoods of Proposition \ref{ecurvelc} and Corollary \ref{corhida} are optimal, by checking that the conditions of Proposition \ref{optimal} are satisfied by the associated (pseudo-)representation. This is the case, for instance, of a Hida family with coefficients in $\Z_p[[T]]$ (i.e., everywhere étale over the weight space), given that $1+T$ belongs to the image of the determinant. Note that in this special case Lemma \ref{XYn} holds trivially with $n_0=1$, so that the subsequent results also hold with $n_0=1$. %(though in this situation, one can also trivially observe that the (pseudo-)representation can only be constant mod p^n if det is...).
\end{enumerate}
\end{rem}

\subsubsection{An explicit example}

We thank Alexandre Maksoud for pointing us to the following (conjectural) example of an explicit equation for a Hida family. In \cite[Section 7.3]{ghadim}, Dimitrov and Ghate guess the existence of a 3-adic Hida family of tame level 13 and ring of functions
\[ \mathbb I=\Z_3[[T]][Y]/(Y^2+T), \]
$T$ the weight variable. Clearly $\Z_3[[T]]\to\mathbb I$ is only ramified at $T=0$, which corresponds to weight 1 in the normalization of \emph{loc. cit.}. Then Corollary \ref{corhida} gives that any two classical eigenforms appearing in $\mathbb I$ are congruent mod $3^n$ if their weights are congruent to 1 mod $3^{n-1}$.

\bigskip

\noindent\textbf{Declarations.} 

\smallskip

\noindent \textit{Funding:} The paper was written while the first author was a postdoctoral fellow at the University of Luxembourg and at Heidelberg University. During the revision of the paper, the first author was employed at Heidelberg University and funded by the Deutsche Forschungsgemeinschaft (DFG, German Research Foundation) - Project-ID 444845124 – TRR 326.
The second author was a postdoctoral fellow at the Université de la Polynésie française, and was supported in part by the Agence Nationale de la Recherche under grant number ANR-20-CE40-0013.

\smallskip

\noindent \textit{Data availability:} No datasets were generated or analysed during the current study.

\smallskip

\noindent \textit{Conflict of interest:} No conflict of interest arose in the preparation of this paper.

\smallskip

\printbibliography

\end{document}